\documentclass[11pt,a4paper,reqno]{amsart}

\usepackage{amsmath,amssymb,mathtools}
\usepackage{enumitem}
\usepackage{needspace}
\usepackage{microtype}
\usepackage[hidelinks,pdfusetitle]{hyperref}

\numberwithin{equation}{section}

\newtheorem{theorem}{Theorem}[section]
\newtheorem{proposition}[theorem]{Proposition}
\newtheorem{lemma}[theorem]{Lemma}
\newtheorem{corollary}[theorem]{Corollary}
\theoremstyle{definition}

\newtheorem{example}[theorem]{Example}
\theoremstyle{remark}
\newtheorem{remark}[theorem]{Remark}

\newcommand{\E}{\mathbb E}
\newcommand{\Pp}{\mathbb P}
\newcommand{\R}{\mathbb R}

\newcommand{\cP}{\mathcal P}
\newcommand{\K}{\mathcal K}
\newcommand{\RV}{\mathrm{RV}}
\newcommand{\PD}{\mathrm{PD}}
\newcommand{\Law}{\mathcal L}
\newcommand{\1}{\mathbf 1}
\newcommand{\dd}{\,\mathrm d}

\DeclareMathOperator{\Var}{Var}

\title[BREIMAN'S CONJECTURE]
{Breiman's conjecture and normalized jumps of subordinators}
\author{Jacopo Lenzi}
\thanks{ORCID: 0000-0003-2882-4223.}
\address{Department of Biomedical and Neuromotor Sciences\\
University of Bologna, Via San Giacomo 12, 40126 Bologna, Italy}
\email{jacopo.lenzi2@unibo.it}
\date{September 10, 2026}

\hypersetup{
  pdftitle={Breiman's conjecture and normalized jumps of subordinators},
  pdfauthor={Jacopo Lenzi},
  pdfkeywords={randomly weighted mean, Breiman's conjecture, self-normalized sum,
  regular variation, Poisson--Dirichlet distribution, normalized random measure}
}

\begin{document}
\raggedbottom

\begin{abstract}
We prove Breiman's conjecture under the first-moment assumption.
Let \(Y_1,Y_2,\ldots\) be iid nonnegative random variables with
\(\mathbb P\{Y_1>0\}>0\), normalized by their sum.  If the resulting
randomly weighted sum converges to a nondegenerate law for one fixed
integrable, nonconstant mark distribution, then the tail of \(Y_1\) is
regularly varying.  More generally, any full-sequence limit for one such
mark, including a constant limit, determines the asymptotic regime of the
ranked weights: one big jump, a Poisson--Dirichlet partition, or dust.  It
consequently determines the limit for every integrable mark, with
convergence in the \(1\)-Wasserstein metric, and the limits of independently
marked empirical measures.  The inverse step is based on a countable
power-sum theorem: signed Fourier--Mellin identities extract a positive
limiting expected power sum from one nondegenerate marked limit, without a
moment of order greater than one.  A ratio--Tauberian argument then recovers
the tail index.  The same method classifies ratios formed from the marked jumps
of a nonzero, unkilled, driftless subordinator at zero and at infinity,
assuming infinite activity at zero.  A nondegenerate limit is equivalent to regular
variation of the L\'evy tail with index in \((-1,0]\).  A constant limit is
equivalent to disappearance of the largest normalized jump, or,
analytically, to slow variation of the integrated L\'evy tail.  The latter
condition need not imply regular variation of the L\'evy tail with index
\(-1\).  A Cauchy-mark example shows that the conclusion can fail without
integrability of the mark.
\end{abstract}

\maketitle

\section{Introduction}
\label{sec:introduction}

Let \(Y_1,Y_2,\ldots\) be iid nonnegative random variables, let
\(S_n=\sum_{i=1}^nY_i\), and attach iid real marks
\(X_1,X_2,\ldots\), independently of the \(Y_i\)'s.  This paper studies
the inverse problem associated with the self-normalized sum, defined on
\(\{S_n>0\}\) by
\begin{equation}
 \frac{\sum_{i=1}^nY_iX_i}{S_n}
 =\sum_{i=1}^n\frac{Y_i}{S_n}X_i.
 \label{eq:intro-ratio}
\end{equation}
The law of \eqref{eq:intro-ratio} is only a one-dimensional observation of
the normalized weights.  Nevertheless, convergence for one fixed
integrable, nonconstant mark distribution determines the asymptotic
behavior of the whole ranked weight sequence.

This question originates in Breiman's work on self-normalized sums
\cite{Breiman}.  He proved the direct limit theorem and a converse when
convergence is assumed for every centered integrable mark, and asked
whether one fixed mark suffices.  Mason and Zinn obtained the converse
under a moment assumption of order greater than two \cite{MZ}.  Kevei and
Mason analyzed subsequential limits \cite{KM12} and proved the converse for
a broad class of marks whose characteristic functions have prescribed
regular variation at the origin \cite{KM16}.  Our first theorem gives the
nondegenerate converse for one fixed mark under the first moment alone.
For completeness, it also includes the constant-limit criterion already
established by Breiman \cite[Theorems~1--2]{Breiman}.

Mason also developed a limit theory for self-normalized triangular arrays
\cite{MasonTriangular,MasonTriangularErratum}.
That theory concerns the asymptotic distributions generated by a triangular
array; the inverse problem here asks what convergence for one fixed mark
forces on the common law of the iid nonnegative variables being normalized.

\subsection{Self-normalized iid weights}

Let \(G\) be the law of \(Y\), assume \(\Pp\{Y>0\}>0\), and write
\[
 \bar G(x)=\Pp\{Y>x\},\qquad
 A_G(x)=\E\!\left[Y\1_{\{Y\leq x\}}\right],\qquad
 I_G(x)=\E(Y\wedge x).
\]
On \(\{S_n>0\}\), put \(W_{n,i}=Y_i/S_n\).  On \(\{S_n=0\}\), put
\(W_{n,i}=1/n\); this changes asymptotic statements only on an event of
probability \(\Pp\{Y=0\}^n\).  Let \(W_n^\downarrow\) be the decreasing
rearrangement of the weights.

We view ranked weights as elements of the compact Kingman
simplex~\cite{Kingman1975}
\[
 \K
 =\left\{q_1\geq q_2\geq\cdots\geq0:
          \sum_{j\geq1}q_j\leq1\right\},
\]
equipped with coordinatewise convergence.  Its mass-one part is
\[
 \K_1
 =\left\{q\in\K:\sum_jq_j=1\right\},
\]
on which we also use the $\ell^1$ topology.  The missing mass
$1-\sum_jq_j$ is the dust mass; the zero partition represents pure dust.

Write \(\mathcal P_1(\R)\) for the probability laws with finite first
absolute moment, \(W_1\) for the $1$-Wasserstein distance, and
\(\RV_\infty(\rho)\) for regular variation at infinity with index \(\rho\).

For \(F\in\mathcal P_1(\R)\), define
\[
 \mathcal B_n^G(F)
 =\Law\!\left(\sum_{i=1}^nW_{n,i}X_i\right),
 \qquad X_i\stackrel{\mathrm{iid}}{\sim}F.
\]
For \(0<\beta<1\), let
\(Q^{(\beta)}\sim\PD(\beta,0)\), independently of the marks, and set
\[
 \mathcal B_\beta(F)
 =\Law\!\left(\sum_{j\geq1}Q_j^{(\beta)}X_j\right),\qquad
 \mathcal B_0(F)=F,\qquad
 \mathcal B_1(F)=\delta_{\int xF(\dd x)}.
\]
The value \(\beta=1\) is a label for the dust endpoint; it is not a tail
index.

\begin{theorem}[A single fixed mark determines the iid regime]
\label{thm:intro-iid-regime}
Fix a non-Dirac law \(F_*\in\mathcal P_1(\R)\).  If
\(\mathcal B_n^G(F_*)\) converges weakly as \(n\to\infty\), then there is a
unique \(\beta\in[0,1]\), and exactly one of the following alternatives
holds.
\begin{enumerate}[label=\textup{(\roman*)}]
\item \(\beta=0\) and \(\bar G\in\RV_\infty(0)\).  In this case
\[
 W_n^\downarrow\longrightarrow e_1=(1,0,\ldots)
 \quad\text{in probability in }\ell^1.
\]
\item \(0<\beta<1\) and \(\bar G\in\RV_\infty(-\beta)\).  In this case
\[
 W_n^\downarrow\Rightarrow Q^{(\beta)}
 \quad\text{in }\ell^1.
\]
\item \(\beta=1\) and
\[
 \frac{\max_{i\leq n}Y_i}{S_n}\longrightarrow0
 \quad\text{in probability}.
\]
The ratio is set to zero on \(\{S_n=0\}\).
In this case \(W_n^\downarrow\to0\) in probability in the product topology
of the Kingman simplex, but not in \(\ell^1\).  The maximal-weight
condition is equivalent to each of the following conditions:
\[
 \begin{gathered}
 \dfrac{S_n}{a_n}\longrightarrow1
 \quad\text{in probability for some }a_n>0,\\
 \dfrac{x\bar G(x)}{A_G(x)}\longrightarrow0,\qquad
 \dfrac{x\bar G(x)}{I_G(x)}\longrightarrow0,\qquad
 A_G,I_G\in\RV_\infty(0).
 \end{gathered}
\]
\end{enumerate}
For every \(F\in\mathcal P_1(\R)\),
\begin{equation}
 W_1\bigl(\mathcal B_n^G(F),\mathcal B_\beta(F)\bigr)
 \longrightarrow0.
 \label{eq:intro-all-marks-W1}
\end{equation}
For every compact \(K\subset(\mathcal P_1(\R),W_1)\), the convergence is
uniform:
\[
 \sup_{F\in K}
 W_1\bigl(\mathcal B_n^G(F),\mathcal B_\beta(F)\bigr)
 \longrightarrow0.
\]
\end{theorem}

The first two alternatives are the one-big-jump and stable regimes.
In these two regimes, the marked sums with mark law $F_*$ have nondegenerate limits,
settling Breiman's question.  The third
alternative is relative stability: each ranked coordinate disappears,
although the weights have total mass one.  It includes finite-mean laws but
is not confined to them.

Theorem~\ref{thm:intro-iid-regime} combines the classification in
Theorem~\ref{thm:iid-complete-classification} with the transfer results in
Theorem~\ref{thm:iid-transfer}.
The proof first recovers the weight regime from convergence for $F_*$.  Once
that regime is known, independent marking gives the weak limit for every
integrable mark.  Conditional Jensen's inequality gives uniform
integrability, which upgrades this convergence to $W_1$ convergence.
The elementary contraction
\[
 W_1\bigl(\mathcal B_n^G(F),\mathcal B_n^G(F')\bigr)
 \leq W_1(F,F')
\]
and the corresponding bound for $\mathcal B_\beta$ then yield uniform
convergence on compact subsets of $(\mathcal P_1(\R),W_1)$.  The same
partition limits also give convergence of independently marked empirical
measures.  If iid locations \(Z_i\) have law \(H\) on a Polish
space $E$, then
\[
 \sum_iW_{n,i}\delta_{Z_i}
\]
converges in distribution, for the weak topology on $\mathcal P(E)$,
to a single atom, a normalized stable random measure, or \(H\),
according to the three alternatives.  Joint convergence holds for every
fixed finite family of \(L^1(H)\) test functions.

This classification uses the iid structure of the variables being
normalized.  For general random mass partitions,
a fixed two-point mark need not determine the partition law; this limitation
already appears in the discussion of Rademacher marks in
\cite[p.~72]{MZ}.  By contrast, for proper mass partitions the collection
of all finitely supported marks determines the partition structure
\cite[Corollary~9]{PitmanWeighted}.

For a fixed non-Dirac \(F\in\mathcal P_1(\R)\), the map
\(\beta\mapsto\mathcal B_\beta(F)\) is continuous and injective on
\([0,1]\).  In fact,
\[
 h_F(\beta)=\int_{\R}\sqrt{1+x^2}\,\mathcal B_\beta(F)(\dd x)
\]
is continuous and strictly decreasing.  Poisson--Dirichlet fragmentation
and strict Jensen's inequality give this conclusion under the first
moment alone; see Theorem~\ref{thm:iid-PD-identification}.
If \(F\) has finite, positive variance, the variance identity
\[
 \Var\bigl(\mathcal B_\beta(F)\bigr)
 =(1-\beta)\Var_F(X)
\]
also identifies the parameter.

Integrability cannot be removed from a theorem valid for every fixed mark.
We construct a positive lacunary law \(G\) for which
\(W_n^\downarrow\) has different subsequential limits, while every
self-normalized sum with standard symmetric Cauchy marks has exactly the
same Cauchy law.

\subsection{The inverse argument}

The main analytic step is formulated for general random subprobability
weights \(P(t)=(P_j(t))_{j\geq1}\).  Let \((X_j)\) be iid copies of a
centered, integrable, nonconstant mark, independent of the weights, and write
\[
 R_t=\sum_{j\geq1}P_j(t)X_j,\qquad
 m_t(s)=\E\sum_{j\geq1}P_j(t)^s.
\]
If \(R_t\) has a nondegenerate limit and \(m_t(r)\) is bounded for some
\(r\in(0,1)\), then \(m_t(p)\) converges to a strictly positive limit for
some
\(1<p\leq2\).  When the mark has a moment of order greater than one, no
power-sum bound at an exponent below one is needed.

Write \(a_X=\sup\{q>0:\E|X|^q<\infty\}\) for the moment abscissa
of the mark.  The proof compares holomorphic subsequential limits of the power sums.
The conditional characteristic product
\(\prod_j\varphi_X(uP_j(t))\), evaluated at positive and negative Fourier
arguments, gives two Mellin identities that separate the moment transforms
of the positive and negative parts of the mark.  If two power-sum limits
disagreed at a moment abscissa below two, these identities would continue
one of the moment transforms through its convergence boundary, which is
a singular point by Landau's theorem
(Lemma~\ref{lem:countable-landau}).  At the
first-moment boundary, integrability also makes any meromorphic singularity
of the resulting quotient removable, forcing the whole subsequential
power-sum transforms to agree.  A uniform remainder estimate handles
exponent two.

If the moment at the finite abscissa $a_X$ is finite, the same
removability argument determines the full expected power-sum profile.
Theorem~\ref{thm:countable-above-one} proves this for $1<a_X<4$, without
a power-sum bound at an exponent below one.  For $2\leq a_X<4$, finite variance improves
the nonlinear remainder from $o(u^2)$ to $O(u^4)$ and gives the larger
Mellin strip needed to reach the moment abscissa.

For the iid weights, write \(\Phi(q)=\E(1-e^{-qY})\).  Convergence to a
nondegenerate marked limit forces \(q\Phi'(q)/\Phi(q)\) to stay uniformly
below one as \(q\downarrow0\).  This supplies the required power-sum bound.
A Poissonized ratio--Tauberian theorem then turns the positive limit of
\(m_t(p)\) into regular variation of \(\bar G\).  The constant limit is
handled separately: for mass-one weights and a nonconstant integrable
mark, convergence to a constant is equivalent to disappearance of the
largest weight.

\subsection{Normalized jumps of a subordinator}

The same argument applies to a nonzero, unkilled, driftless subordinator
\(V\) with L\'evy measure \(\nu\).  Attach iid copies of an integrable
nonconstant mark \(X\) to its jumps, independently of \(V\), and put
\[
 V_t=\sum_{0<s\leq t}\Delta V_s,\qquad
 U_t=\sum_{0<s\leq t}X_s\Delta V_s,\qquad
 R_t=\frac{U_t}{V_t}
\]
on \(\{V_t>0\}\), with any fixed convention otherwise.  At \(t\downarrow
0\) we assume infinite activity.  Write
\[
 \bar\nu(x)=\nu((x,\infty)),\qquad
 A(x)=\int_{(0,x]}y\,\nu(\dd y),\qquad
 I(x)=\int_0^x\bar\nu(y)\,\dd y.
\]
The notation \(t\to\star\) and \(x\to\star\) refers to the corresponding
time and jump-size endpoints, either zero or infinity.

\begin{theorem}[Marked subordinator ratios at zero and infinity]
\label{thm:intro-subordinator}
Fix either \(t\downarrow0\), under infinite activity, or \(t\to\infty\).
Then \(R_t\) converges in distribution if and only if exactly one of the
following alternatives holds.
\begin{enumerate}[label=\textup{(\roman*)}]
\item For a unique \(\alpha\in[0,1)\),
\(\bar\nu\in\RV_\star(-\alpha)\).  The limit is \(X\) when
\(\alpha=0\), and has law
\[
 \Law\!\left(\sum_{j\geq1}Q_j^{(\alpha)}X_j\right),
 \qquad Q^{(\alpha)}\sim\PD(\alpha,0),
\]
when \(0<\alpha<1\).  The ranked normalized jumps converge in \(\ell^1\):
in probability to \(e_1\) when \(\alpha=0\), and in distribution to
\(Q^{(\alpha)}\) when \(0<\alpha<1\).
\item
\[
 \frac{x\bar\nu(x)}{A(x)}\longrightarrow0.
\]
Equivalently, the largest normalized jump tends to zero in probability, the expected power
sum of some (and hence every) order greater than one tends to zero, and
\(I\in\RV_\star(0)\).  In this case \(R_t\to\E X\) in probability, and
the ranked normalized jumps converge to zero in probability in the product
topology, but not in \(\ell^1\).
\end{enumerate}
\end{theorem}

Let \(\Phi\) now denote the Laplace exponent of \(V\).  For a nondegenerate
small-time limit, the same reasoning forces \(q\Phi'(q)/\Phi(q)\) to stay
uniformly below one as \(q\to\infty\), supplying a power-sum bound for an exponent in
\((0,1)\); the tail is then recovered by a fractional
ratio--Tauberian theorem.  At large time, bounded jumps are negligible in
the infinite-mean case, and the remaining compound-Poisson ratio reduces
to the iid theorem by depoissonization.  The finite-mean case belongs to
the constant branch.  The ranked-jump limits imply, by independent
marking, the corresponding limits for normalized L\'evy random measures
and their integrable mean functionals.

Theorem~\ref{thm:intro-subordinator} also resolves the first-moment question
raised by Kevei and Mason \cite[Remark~1]{KM13}.  For nondegenerate limits,
regular variation of the L\'evy tail follows from convergence for one
nonconstant integrable mark.  At the constant boundary, the exact
condition is disappearance of the largest normalized jump, as characterized
in \cite[Theorem~1.1(iii), \(k=0\)]{KM14}.  Equivalently, the integrated
tail \(I\) is slowly varying.  Regular variation of \(\bar\nu\) with index
\(-1\) implies this condition, whereas the converse need not hold.

The distinction is an index-zero phenomenon.  Slow variation of a
primitive does not in general imply regular variation with index \(-1\) of
its monotone density; related boundary behavior for truncated moments is
discussed in \cite[Theorem~1.1 and Remark~1.4]{Kevei2021}.  We give two direct
examples.  At infinity, the gamma subordinator has a constant marked-ratio
limit although its L\'evy tail is exponentially decreasing.  At zero, a
bounded-support monotone L\'evy tail with a logarithmically oscillating
factor has slowly varying integrated tail but is not regularly varying.

The forward theory of normalized jumps and Poisson--Dirichlet partitions
is developed in
\cite{PermanPitmanYor,Perman,PitmanYor,PitmanPK}.  Related
self-normalization results for L\'evy processes include
\cite{MallerMason,MallerMason2010,KM13,KM14}.  For maximal-jump dominance,
stable limits of trimmed L\'evy processes and small-time convergence of
subordinators under regularly or slowly varying L\'evy tails, see
\cite{BuchmannFanMaller2016,IpsenKeveiMaller2018,MallerSchindler2019}.
The Tauberian arguments used
below draw on the Drasin--Shea theory \cite{DS,BGT}, in the fractional
formulations of \cite{BinghamInoue,Rehak}.  The random-measure consequences
belong to the theory of completely random and normalized random measures
\cite{Kingman1967,RegazziniLijoiPrunster,JamesLijoiPrunster,
LijoiPrunsterSurvey}.

\subsection{Organization}

Section~\ref{sec:countable} proves the countable series and power-sum
theorems.  Sections~\ref{sec:poisson-calculus} and
\ref{sec:poisson-limits} develop the Poisson power-sum identities, ranked
normalization, and independent marking.  Section~\ref{sec:breiman} proves
the converse for one fixed mark in iid self-normalized sums.
Sections~\ref{sec:iid-classification} and \ref{sec:iid-operators} give the
complete iid classification and its consequences for other marks and
marked measures; Section~\ref{sec:iid-cauchy-sharpness} gives the Cauchy
counterexample.  Sections~\ref{sec:smalltime}--\ref{sec:classification}
develop the small-time and ratio--Tauberian arguments and complete the
subordinator classification at both endpoints.  Section~\ref{sec:measure}
records the normalized-random-measure consequences.

\section{A countable power-sum principle}
\label{sec:countable}

Throughout this section, $t$ ranges over $(0,\infty)$ and $t\to\star$
means either $t\downarrow0$ or $t\to\infty$.  All bounds in $t$ are local
at the specified endpoint.  Let
\[
 P(t)=(P_j(t))_{j\geq1},\qquad P_j(t)\geq0,\qquad
 \sum_{j\geq1}P_j(t)\leq1
\]
be a random subprobability sequence.  Let $X,X_1,X_2,\ldots$ be iid real
random variables, independent of the weights.  We write
\[
 R_t=\sum_{j\geq1}P_j(t)X_j,
 \qquad
 m_t(s)=\E\sum_{j\geq1}P_j(t)^s.
\]
The first result extracts a positive limiting expected power sum from a
nondegenerate limit of $R_t$.  The ratio--Tauberian arguments in
Sections~\ref{sec:breiman-tauberian} and~\ref{sec:tauberian} then use this
limit to recover the tail index for iid weights and L\'evy jumps,
respectively.  The first result does not require the expected total mass
$m_t(1)$ to converge.

\Needspace{12\baselineskip}
\begin{theorem}[Countable power-sum convergence principle]
\label{thm:countable-main}
Assume that
\[
 \E X=0,\qquad 0<\E|X|<\infty,
\]
and that, for some $r\in(0,1)$,
\begin{equation}
 \limsup_{t\to\star}m_t(r)<\infty.
 \label{eq:countable-r-bound}
\end{equation}
If $R_t\Rightarrow R$ as $t\to\star$, where $R$ is nondegenerate, then
there are $p\in(1,2]$ and $c\in(0,1]$ such that
\begin{equation}
 m_t(p)\longrightarrow c.
 \label{eq:countable-positive-power}
\end{equation}

More precisely, let $X_+=\max(X,0)$, $X_-=\max(-X,0)$, and put
\begin{align*}
 a_+&=\sup\{q>0:\E X_+^q<\infty\},\\
 a_-&=\sup\{q>0:\E X_-^q<\infty\},
 \qquad a_X=a_+\wedge a_-.
\end{align*}
If $a_X=1$, then for every fixed $p\in(1,2)$ there is a possibly
$p$-dependent $c_p\in(0,1]$ such that $m_t(p)\to c_p$.  If
$1<a_X<2$, \eqref{eq:countable-positive-power} holds with $p=a_X$, and if
$a_X\geq2$, it holds with $p=2$.
\end{theorem}

We shall also use the immediate sequential specialization of this theorem:
take \(t=n\), let \(n\to\infty\), and allow the law of \(P(n)\) to depend
on \(n\).  All hypotheses and conclusions are then read along that
sequence.

The proof at the $L^1$ boundary uses a power-sum bound at an exponent
below one.  If the mark has a moment of some order greater than one,
the Mellin transforms can instead be considered on a strip contained in
$\{\Re s>1\}$, and that bound is unnecessary.

\Needspace{8\baselineskip}
\begin{theorem}[Power-sum convergence with a moment of order greater than one]
\label{thm:countable-above-one}
Assume that
\[
 \E X=0,\qquad a_X>1,
\]
where $a_X$ is the moment abscissa defined in
Theorem~\ref{thm:countable-main}.  No convergence of $m_t(1)$ and no bound
on $m_t(r)$ for $r<1$ are imposed.  If $R_t\Rightarrow R$ as
$t\to\star$, where $R$ is nondegenerate, then
\[
 \begin{cases}
  m_t(a_X)\longrightarrow c\in(0,1],&1<a_X<2,\\
  m_t(2)\longrightarrow c\in(0,1],&a_X\geq2.
 \end{cases}
\]
If $1<a_X<4$ and $\E|X|^{a_X}<\infty$, then $m_t$ converges locally
uniformly on $\{\Re s>1\}$ to a holomorphic function $m$.  In that case
$m(q)>0$ for every real $q>1$.
\end{theorem}

Vanishing expected power sums of order greater than one have an exact
probabilistic description when the weights have total mass one.  We record
it at the outset because it will later separate the dust regime from the
nondegenerate one.

\begin{theorem}[Constant limits and vanishing maximal weight]
\label{thm:countable-dust}
Suppose that $X$ is integrable and nonconstant, and that
\[
 P_j(t)\geq0,\qquad \sum_{j\geq1}P_j(t)=1
 \quad\text{almost surely}.
\]
Set
\[
 S_t=\sum_{j\geq1}P_j(t)X_j,\qquad
 P_*(t)=\sup_{j\geq1}P_j(t),\qquad \mu=\E X.
\]
The following are equivalent as $t\to\star$:
\begin{enumerate}
\item $S_t$ converges in distribution to a constant;
\item $P_*(t)\to0$ in probability;
\item for one, and hence for every, $p>1$,
\[
 \E\sum_{j\geq1}P_j(t)^p\longrightarrow0.
\]
\end{enumerate}
Whenever these conditions hold, the constant is $\mu$ and
$S_t\to\mu$ in probability.
\end{theorem}

The proofs of the power-sum principles compare holomorphic subsequential
limits of $m_t$.  Equality of
the limiting marked laws gives two signed identities involving the
difference of these power-sum limits and the moment transforms of the
positive and negative parts of $X$.
Landau's theorem at the moment abscissa selects the exponent where these
limits agree; a separate boundary estimate handles exponent two.  We first
justify the countable products and uniform remainders used in those
identities.

\subsection{Countable series, products, and Dirichlet series}
\label{sec:countable-series}

Zero weights are omitted when complex powers are formed, and
$x^s=\exp(s\log x)$ for $x>0$.

\begin{lemma}[Weighted series and conditional products]
\label{lem:countable-series-products}
Let $(p_j)_{j\geq1}$ be nonnegative with $\sum_jp_j\leq1$, and let
$X,X_1,X_2,\ldots$ be iid with $\E|X|<\infty$.  Then
$\sum_jp_jX_j$ converges absolutely almost surely.  If
$\chi(u)=\E e^{iuX}$, then, for every $u\in\R$,
\[
 \prod_{j\geq1}\chi(up_j)
\]
is an order-independent convergent product and is the characteristic
function of $\sum_jp_jX_j$.
\end{lemma}

\begin{proof}
Tonelli's theorem gives
\[
 \E\sum_{j\geq1}p_j|X_j|
 =\E|X|\sum_{j\geq1}p_j\leq\E|X|,
\]
so the nonnegative random sum on the left is finite almost surely.  Moreover,
\[
 \sum_{j\geq1}|1-\chi(up_j)|
 \leq |u|\E|X|\sum_{j\geq1}p_j<\infty.
\]
The infinite product therefore converges independently of the ordering,
with the usual interpretation if one factor vanishes.  The characteristic
functions of the finite partial sums converge to this product, while the
partial sums converge almost surely.  Bounded convergence completes the
proof.
\end{proof}

\begin{lemma}[Dirichlet series of the weights]
\label{lem:countable-dirichlet}
Fix $r>0$.  If $\sum_jP_j(t)^r<\infty$ almost surely, then
\[
 D_t(s)=\sum_{j\geq1}P_j(t)^s
\]
is almost surely holomorphic on $\Re s>r$.  If $m_t(r)<\infty$, then
$m_t(s)=\E D_t(s)$ is holomorphic on the same half-plane.  Under a
uniform bound on $m_t(r)$, the family $(m_t)$ is locally bounded, and hence
normal, there.
\end{lemma}

\begin{proof}
For a compact set $K\subset\{\Re s>r\}$, let
$\sigma_K=\inf_{s\in K}\Re s$.  Since all weights lie in $[0,1]$,
\[
 \sup_{s\in K}\sum_{j\geq1}|P_j(t)^s|
 \leq\sum_{j\geq1}P_j(t)^{\sigma_K}
 \leq\sum_{j\geq1}P_j(t)^r.
\]
The Weierstrass theorem gives pathwise holomorphy.  Taking expectations and
using the same bound gives holomorphy by Morera's theorem and Tonelli's
theorem.  If $\sup_t m_t(r)<\infty$, the displayed estimate also gives
local boundedness uniformly in $t$, so Montel's theorem applies.
\end{proof}

We use the following countable version of the weighted weak law.

\begin{lemma}[Weighted weak law under a first moment]
\label{lem:countable-wlln}
Let $(P_j(t))_{j\geq1}$ be random subprobability weights independent of
iid centered marks with finite first absolute moment.  If
$P_*(t)=\sup_jP_j(t)\to0$ in probability, then
\[
 \sum_{j\geq1}P_j(t)X_j\longrightarrow0
 \quad\text{in probability}.
\]
\end{lemma}

\begin{proof}
For $K>0$, write
\[
 Z_K=X\1_{\{|X|\leq K\}}
      -\E[X\1_{\{|X|\leq K\}}],
 \qquad W_K=X-Z_K.
\]
Then $\E Z_K=\E W_K=0$, $\E Z_K^2<\infty$, and
\[
 \E|W_K|
 \leq2\E\bigl(|X|\1_{\{|X|>K\}}\bigr)\longrightarrow0.
\]
Since $\sum_jP_j(t)^2\leq P_*(t)\sum_jP_j(t)<\infty$, the truncated
series converges in conditional $L^2$ and its conditional variances pass to
the limit.  Thus, conditionally on the weights,
\[
 \Var\!\left(\sum_jP_j(t)Z_{K,j}
       \,\middle|\,P(t)\right)
 =\E Z_K^2\sum_jP_j(t)^2
 \leq\E Z_K^2 P_*(t).
\]
Chebyshev's inequality controls the truncated sum after $t\to\star$.
The remainder is uniform in $t$, since
\[
 \E\left|\sum_jP_j(t)W_{K,j}\right|\leq\E|W_K|.
\]
Letting $K\to\infty$ proves the assertion.
\end{proof}

\begin{proof}[Proof of Theorem~\ref{thm:countable-dust}]
If $P_*(t)\to0$, Lemma~\ref{lem:countable-wlln}, applied to
$X_j-\mu$, gives $S_t\to\mu$ in probability.

Conversely, suppose that $S_t\Rightarrow c$.  Enlarge the probability
space by an independent iid copy $(X_j')$ of the marks, using the same
weights, and set $S_t'=\sum_jP_j(t)X_j'$.  Both $S_t$ and $S_t'$ converge
in probability to $c$, and hence $S_t-S_t'\to0$ in probability.  By
Lemma~\ref{lem:countable-series-products}, conditioning on the weights gives
\begin{equation}
 \E e^{iu(S_t-S_t')}
 =\E\prod_{j\geq1}|\chi(uP_j(t))|^2\longrightarrow1.
 \label{eq:countable-symmetrized-product}
\end{equation}
Since $X$ is nonconstant, there is $a>0$ such that
\begin{equation}
 |\chi(v)|<1\qquad(0<|v|\leq a).
 \label{eq:countable-strict-cf}
\end{equation}
Indeed, otherwise there would be nonzero $v_n\to0$ such that
$|\chi(v_n)|=1$.  For an independent copy $X'$, this would imply
$e^{iv_n(X-X')}=1$ almost surely for every $n$, and hence $X=X'$ almost
surely, contradicting nonconstancy.

Fix $\varepsilon\in(0,1]$ and $0<u\leq a$.  By compactness and
\eqref{eq:countable-strict-cf},
\[
 \rho_\varepsilon
 =\sup_{v\in[u\varepsilon/2,u]}|\chi(v)|^2<1.
\]
On $\{P_*(t)\geq\varepsilon\}$, some weight is larger than
$\varepsilon/2$, so the product in
\eqref{eq:countable-symmetrized-product} is at most $\rho_\varepsilon$.
It follows that
\[
 \E\prod_j|\chi(uP_j(t))|^2
 \leq1-(1-\rho_\varepsilon)
       \Pp\{P_*(t)\geq\varepsilon\}.
\]
Thus $P_*(t)\to0$ in probability, and the first part of the proof gives
$c=\mu$.

Finally, for every $p>1$,
\begin{equation}
 P_*(t)^p\leq\sum_jP_j(t)^p\leq P_*(t)^{p-1}.
 \label{eq:countable-max-power}
\end{equation}
If $P_*(t)\to0$ in probability, then the middle term tends to zero in
probability and is bounded by one, so it converges to zero in $L^1$ and
hence in expectation.  Conversely, convergence of that
expectation and the left inequality imply $P_*(t)\to0$ by Markov's
inequality.
\end{proof}

\subsection{The nonlinear remainder}
\label{sec:countable-remainder}

We now return to the hypotheses of Theorem~\ref{thm:countable-main},
including $R_t\Rightarrow R$, and fix $r$ as in
\eqref{eq:countable-r-bound}.
For a deterministic countable subprobability sequence $p=(p_j)$ and
$\varepsilon\in\{-1,1\}$, define
\begin{align*}
 A_p^\varepsilon(u)
   &=\sum_{j\geq1}\{\chi(\varepsilon up_j)-1\},\\
 N_p^\varepsilon(u)
   &=\prod_{j\geq1}\chi(\varepsilon up_j)-1-A_p^\varepsilon(u).
\end{align*}
Both quantities are well defined by
Lemma~\ref{lem:countable-series-products}.

\begin{lemma}[Uniform remainder estimates]
\label{lem:countable-remainder}
There are a constant $C_r<\infty$ and a deterministic function $\delta$ with
$\delta(u)\to0$ as $u\downarrow0$ such that, for either sign and every
countable subprobability sequence $p$,
\begin{align}
 |N_p^\varepsilon(u)|
   &\leq u^2\delta(u), &&0<u\leq1,
 \label{eq:countable-small-remainder}\\
 |N_p^\varepsilon(u)|
   &\leq2+C_ru^r\sum_{j\geq1}p_j^r, &&u\geq1.
 \label{eq:countable-large-remainder}
\end{align}
\end{lemma}

\begin{proof}
We first truncate the sequence after $n$ terms.  Since $X$ is centered and
integrable,
\[
 \omega(u)=\sup_{0<|v|\leq u}
 \frac{|\chi(v)-1|}{|v|}\longrightarrow0
 \qquad(u\downarrow0).
\]
For $z_j=\chi(\varepsilon up_j)-1$ and $0<u\leq1$, set
\[
 \delta(u)=\frac12e^{u\omega(u)}\omega(u)^2.
\]
Then $\delta(u)\to0$ as $u\downarrow0$, and
$\sum_{j\leq n}|z_j|\leq u\omega(u)$.  Expanding the finite product,
\[
 \left|\prod_{j\leq n}(1+z_j)-1-\sum_{j\leq n}z_j\right|
 \leq e^{\sum_{j\leq n}|z_j|}-1-\sum_{j\leq n}|z_j|
 \leq u^2\delta(u).
\]
This proves \eqref{eq:countable-small-remainder}.  For $u\geq1$, the
inequality $|e^{iv}-1|\leq2^{1-r}|v|^r$ and the finiteness of
$\E|X|^r$ give
\[
 |A_{p,n}^\varepsilon(u)|
 \leq C_ru^r\sum_{j\leq n}p_j^r,
 \qquad
 \left|\prod_{j\leq n}\chi(\varepsilon up_j)-1\right|\leq2.
\]
Letting $n\to\infty$ proves both bounds, since the sums and products
converge absolutely and the constants do not depend on $n$.
\end{proof}

After restricting the parameter set to a terminal neighborhood of
$\star$, we may replace \eqref{eq:countable-r-bound} by
\begin{equation}
 \sup_t m_t(r)\leq M_r<\infty.
 \label{eq:countable-uniform-r}
\end{equation}
For $r<\Re s<2$, set
\[
 J_{t,\varepsilon}(s)
 =\int_0^\infty u^{-s-1}
   \E N_{P(t)}^\varepsilon(u)\dd u.
\]
Lemma~\ref{lem:countable-remainder} and
\eqref{eq:countable-uniform-r} show that $J_{t,\varepsilon}$ is
holomorphic on the strip
\[
 \mathcal S_r=\{s\in\mathbb C:r<\Re s<2\},
\]
and that the family $(J_{t,\varepsilon})_t$ is normal there.  The same
majorants give the boundary estimate
\begin{equation}
 \lim_{s\uparrow2}(2-s)
 \sup_t|J_{t,\varepsilon}(s)|=0.
 \label{eq:countable-abelian-boundary}
\end{equation}
To see this, split the defining integral at $1$.  The part over
$[1,\infty)$ vanishes after multiplication by $2-s$, since its bound stays
bounded as $s\uparrow2$.  Given $\eta>0$, choose $a\in(0,1)$ such that
$\delta(u)\leq\eta$ on $(0,a)$.  Then
\[
 (2-s)\int_0^a u^{1-s}\delta(u)\dd u
 \leq\eta a^{2-s},
\]
whereas the same expression integrated over $(a,1)$ tends to zero because
$\delta$ is bounded there.  This proves
\eqref{eq:countable-abelian-boundary}.

\subsection{Signed Mellin identities}
\label{sec:countable-mellin}

Put
\[
 K_{t,\varepsilon}(u)=\E e^{i\varepsilon uR_t}
\]
and, for $0<\Re s<1$,
\[
 L_{t,\varepsilon}(s)
 =\int_0^\infty u^{-s-1}
   \{K_{t,\varepsilon}(u)-1\}\dd u.
\]
The estimate
\[
 |K_{t,\varepsilon}(u)-1|
 \leq\min\{2,u\E|X|\}
\]
is uniform in $t$.  If $\varphi_R$ is the characteristic function of $R$,
define
\[
 L_\varepsilon(s)=\int_0^\infty u^{-s-1}
       \{\varphi_R(\varepsilon u)-1\}\dd u,
 \qquad 0<\Re s<1.
\]
For a compact set $K\subset\{0<\Re s<1\}$, put
$\sigma_0=\inf_{s\in K}\Re s$ and
$\sigma_1=\sup_{s\in K}\Re s$.  Up to a constant depending only on
$K$ and $\E|X|$, the integrands are bounded by
\[
 u^{-\sigma_1}\1_{(0,1]}(u)
 +u^{-\sigma_0-1}\1_{(1,\infty)}(u).
\]
Weak convergence gives uniform convergence of the characteristic
functions on compact subsets of $\R$.  Dominated convergence therefore
shows that $L_{t,\varepsilon}\to L_\varepsilon$ locally uniformly on
$0<\Re s<1$.

For $0<\Re s<1$, define
\[
 \kappa_\varepsilon(s)
 =\int_0^\infty u^{-s-1}\{\chi(\varepsilon u)-1\}\dd u.
\]
If $r<\sigma=\Re s<1$, Fubini's theorem and the substitution
$v=uP_j(t)$ are justified by
\begin{align*}
 &\int_0^\infty u^{-\sigma-1}
 \E\sum_j
 |\chi(\varepsilon uP_j(t))-1|\dd u\\
 &\hspace{35mm}
 =m_t(\sigma)
   \int_0^\infty v^{-\sigma-1}|\chi(\varepsilon v)-1|\dd v<\infty.
\end{align*}
Consequently,
\begin{equation}
 L_{t,\varepsilon}(s)
 =m_t(s)\kappa_\varepsilon(s)+J_{t,\varepsilon}(s),
 \qquad r<\Re s<1.
 \label{eq:countable-mellin-decomposition}
\end{equation}

Let $m$ be an arbitrary locally uniform subsequential limit of $(m_t)$,
obtained along a sequence $t_n\to\star$ on
which
\[
 m_{t_n}\longrightarrow m
 \quad\text{locally uniformly on }\{\Re s>r\}.
\]
Normality of the two families $(J_{t,+})$ and $(J_{t,-})$ permits a
further diagonal extraction such that
\[
 J_{t_n,\varepsilon}\longrightarrow J_\varepsilon
 \quad\text{locally uniformly on }\mathcal S_r,
 \qquad \varepsilon\in\{-1,1\}.
\]
The further extraction does not change $m$.  Hence every subsequential limit
of $(m_t)$ can be completed to a triple $(m,J_+,J_-)$.  Complete two
arbitrary subsequential limits in this way and write
\[
 d=m^{(1)}-m^{(2)},
 \qquad
 Q_\varepsilon=J_\varepsilon^{(2)}-J_\varepsilon^{(1)}.
\]
The common limit of $L_{t,\varepsilon}$ and
\eqref{eq:countable-mellin-decomposition} give
\begin{equation}
 d(s)\kappa_\varepsilon(s)=Q_\varepsilon(s),
 \qquad r<\Re s<1.
 \label{eq:countable-cluster-identity}
\end{equation}
For every real $s\in(r,2)$,
\[
 |Q_\varepsilon(s)|
 \leq |J_\varepsilon^{(1)}(s)|+|J_\varepsilon^{(2)}(s)|
 \leq2\sup_t|J_{t,\varepsilon}(s)|.
\]
Thus \eqref{eq:countable-abelian-boundary} yields
\begin{equation}
 (2-s)Q_\varepsilon(s)\longrightarrow0
 \qquad(s\uparrow2).
 \label{eq:countable-Q-boundary}
\end{equation}

The scalar integral underlying the two signed identities is recorded next,
including the branch and integrability conventions.

\begin{lemma}[Regularized Fourier--Mellin integrals]
\label{lem:countable-scalar-mellin}
Let $0<\Re s<1$, $\varepsilon\in\{-1,1\}$, and $x\in\R$.  With
$0^s=0$ and, for $x\ne0$,
\[
 (-i\varepsilon x)^s
 =\exp\left\{s\left(\log|x|
       -\frac{i\pi}{2}\varepsilon\operatorname{sgn}x\right)\right\},
\]
one has
\begin{equation}
 \int_0^\infty u^{-s-1}(e^{i\varepsilon ux}-1)\dd u
 =\Gamma(-s)(-i\varepsilon x)^s.
 \label{eq:countable-scalar-mellin}
\end{equation}
If $\sigma=\Re s\in(0,1)$, then
\begin{equation}
 \int_0^\infty u^{-\sigma-1}
       |e^{i\varepsilon ux}-1|\dd u
 =C_\sigma|x|^\sigma,
 \qquad
 C_\sigma=\int_0^\infty v^{-\sigma-1}|e^{iv}-1|\dd v<\infty.
 \label{eq:countable-scalar-domination}
\end{equation}
For $1<\Re s<2$, the corresponding first-order regularization satisfies
\begin{equation}
 \int_0^\infty u^{-s-1}
       (e^{i\varepsilon ux}-1-i\varepsilon ux)\dd u
 =\Gamma(-s)(-i\varepsilon x)^s,
 \label{eq:countable-scalar-mellin-above-one}
\end{equation}
with the same branch convention.  If $1<\sigma<2$, then
\begin{equation}
 \int_0^\infty u^{-\sigma-1}
       |e^{i\varepsilon ux}-1-i\varepsilon ux|\dd u
 =C_\sigma^{(2)}|x|^\sigma<\infty,
 \label{eq:countable-scalar-domination-above-one}
\end{equation}
where
$C_\sigma^{(2)}=\int_0^\infty
 v^{-\sigma-1}|e^{iv}-1-iv|\dd v$.
\end{lemma}

\begin{proof}
For $\Re z>0$, integration by parts, with the principal branch of $z^s$,
gives
\[
 \int_0^\infty u^{-s-1}(e^{-zu}-1)\dd u=\Gamma(-s)z^s.
\]
Let $z$ tend to $-i\varepsilon x$ through the right half-plane.  Dominated
convergence is justified near zero by a constant multiple of
$u^{-\Re s}$ and at infinity by a constant multiple of
$u^{-\Re s-1}$.  This proves \eqref{eq:countable-scalar-mellin} and its
stated branch.  The substitution $v=u|x|$ proves
\eqref{eq:countable-scalar-domination}; the constant is finite because
$|e^{iv}-1|\leq\min\{2,v\}$.

For the second strip, let $1<\Re s<2$ and $\Re z>0$.  Two integrations by
parts, or one integration by parts followed by the first identity with
parameter $s-1$, give
\[
 \int_0^\infty u^{-s-1}(e^{-zu}-1+zu)\dd u
 =\Gamma(-s)z^s.
\]
Letting $z$ tend to $-i\varepsilon x$ through the right half-plane proves
\eqref{eq:countable-scalar-mellin-above-one}.  Uniform domination follows
from a constant multiple of $u^{1-\Re s}$ near zero and of
$u^{-\Re s}$ at infinity.  Finally, scaling by $|x|$ proves
\eqref{eq:countable-scalar-domination-above-one}; its constant is finite
because $|e^{iv}-1-iv|=O(v^2)$ at zero and is $O(v)$ at infinity.
\end{proof}

Let
\[
 A(s)=\E X_+^s,\qquad B(s)=\E X_-^s.
\]
Here and below, $0^s=0$ in the moment transforms.  For
$0<\sigma<1$, the first moment gives $\E|X|^\sigma<\infty$, and
Lemma~\ref{lem:countable-scalar-mellin} yields
\[
 \int_0^\infty u^{-\sigma-1}
 \E|e^{i\varepsilon uX}-1|\dd u
 =C_\sigma\E|X|^\sigma<\infty.
\]
Fubini's theorem is therefore available before the positive and negative
parts of $X$ are separated.  Initially on $0<\Re s<1$, we obtain
\begin{align*}
 \kappa_+(s)
 &=\Gamma(-s)\{e^{-i\pi s/2}A(s)+e^{i\pi s/2}B(s)\},\\
 \kappa_-(s)
 &=\Gamma(-s)\{e^{i\pi s/2}A(s)+e^{-i\pi s/2}B(s)\}.
\end{align*}
Solving the two signed identities removes the apparent singularity at
$s=1$:
\begin{align}
 d(s)A(s)
 &=F_A(s):=\frac{\Gamma(1+s)}{2\pi i}
 \{e^{-i\pi s/2}Q_+(s)-e^{i\pi s/2}Q_-(s)\},
 \label{eq:countable-FA}\\
 d(s)B(s)
 &=F_B(s):=\frac{\Gamma(1+s)}{2\pi i}
 \{-e^{i\pi s/2}Q_+(s)+e^{-i\pi s/2}Q_-(s)\}.
 \label{eq:countable-FB}
\end{align}
Both right-hand sides are holomorphic throughout $\mathcal S_r$.

For reference, the domains used in the continuation argument are
\[
\begin{array}{c|c}
\text{function}&\text{domain of holomorphy}\\ \hline
m^{(i)},\ d&\Re s>r\\
J_\varepsilon^{(i)},\ Q_\varepsilon,\ F_A,\ F_B
  &r<\Re s<2\\
A&0<\Re s<a_+\\
B&0<\Re s<a_-
\end{array}
\]
The identities obtained on $r<\Re s<1$ extend by the identity theorem to
\begin{align*}
 d(s)A(s)&=F_A(s),
   &&r<\Re s<\min\{a_+,2\},\\
 d(s)B(s)&=F_B(s),
   &&r<\Re s<\min\{a_-,2\}.
\end{align*}
If $a_X<\infty$, the transform $A$ has abscissa $a_X$ when
$a_+=a_X$, and $B$ has abscissa $a_X$ when $a_-=a_X$.  At least one
of these alternatives holds, and the corresponding transform is singular
there by Lemma~\ref{lem:countable-landau} below.

\subsection{Selection of the exponent}
\label{sec:countable-selection}

We first record the precise form of Landau's theorem used below.

\begin{lemma}[Landau at the abscissa of convergence]
\label{lem:countable-landau}
Let $Z\geq0$ and
\[
 0<a_Z=\sup\{q>0:\E Z^q<\infty\}<\infty.
\]
Then $s=a_Z$ is a singular point of the moment transform
$s\mapsto\E Z^s$.  This remains true when
$\E Z^{a_Z}<\infty$.
\end{lemma}

\begin{proof}
The contribution from $0<Z\leq1$ is holomorphic in a neighborhood of every
positive real point.  For the contribution from $Z>1$, define the finite
measure $\nu$ on $(0,\infty)$ by
\[
 \nu(B)=\Pp\{Z>1,\ \log Z\in B\}.
\]
The corresponding moment contribution is the Laplace transform
\[
 G(s)=\int_{(0,\infty)}e^{sx}\,\nu(\dd x),
 \qquad \Re s<a_Z,
\]
of a positive measure.  Suppose that there were a holomorphic function
$\widetilde G$ on $D(a_Z,\epsilon)$ which agrees with $G$ on the nonempty,
connected overlap $D(a_Z,\epsilon)\cap\{\Re s<a_Z\}$.  The identity theorem
makes this agreement unique, so $G$ and $\widetilde G$ glue to one
holomorphic function on the union of the original half-plane and the
continuation disk.  After reducing $\epsilon$ so that
$\epsilon<2a_Z$, take
$a=a_Z-\epsilon/4$ and $h=\epsilon/3$.  Then
\[
 D(a,\epsilon/2)\subset D(a_Z,3\epsilon/4)
 \subset D(a_Z,\epsilon),
\]
and $a+h>a_Z$.  The Taylor series of the glued function
about $a$ therefore
converges at $a+h$.  For every $n\geq0$,
differentiation below the abscissa is justified by choosing
$a<b<a_Z$ and using $x^ne^{ax}\leq C_{n,a,b}e^{bx}$.  It gives
\[
 G^{(n)}(a)=\int_0^\infty x^ne^{ax}\,\nu(\dd x)\geq0.
\]
Tonelli's theorem and convergence of the Taylor series then yield
\[
 \int_0^\infty e^{(a+h)x}\,\nu(\dd x)
 =\sum_{n=0}^\infty\frac{h^n}{n!}G^{(n)}(a)<\infty,
\]
contradicting $a+h>a_Z$.  Thus the Laplace transform, and hence the moment
transform, is singular at its abscissa.
\end{proof}

We shall also use the following removability criterion.

\begin{lemma}[Removability at a real boundary point]
\label{lem:countable-removability}
Let $f$ and $g$ be holomorphic in a neighborhood of a real point $a$, with
$g\not\equiv0$.  Suppose that $h$ is holomorphic on the portion of that
neighborhood lying to the left of $a$, that $f=gh$ there, and that
$h(s)$ remains bounded for real $s\uparrow a$.  Then $f/g$ has a removable
singularity at $a$.
\end{lemma}

\begin{proof}
If $g(a)\neq0$, there is nothing to prove.  Otherwise, let $k$ and $\ell$
be the orders of the zeros of $g$ and $f$ at $a$; the case $f\equiv0$ is
immediate.  Along the real axis,
$|f(s)/g(s)|$ is asymptotic to a nonzero constant times
$|s-a|^{\ell-k}$.  Its boundedness forces $\ell\geq k$, which is exactly
the removability assertion.
\end{proof}

The following boundary estimate applies whether or not the second moment
is finite.

\begin{lemma}[Second-moment boundary]
\label{lem:countable-second-moment-boundary}
Assume that $\E|X|^s<\infty$ for every $1<s<2$, that $X$ is not almost surely
zero, and that $d$ is continuous at $2$.  Suppose, for $1<s<2$, that
\[
 2d(s)\Gamma(-s)\cos(\pi s/2)\E|X|^s=Q_+(s)+Q_-(s)
\]
and
\[
 (2-s)Q_\varepsilon(s)\longrightarrow0,
 \qquad \varepsilon\in\{-1,1\}.
\]
Then $d(2)=0$.
\end{lemma}

\begin{proof}
Put $c_s=2(2-s)\Gamma(-s)\cos(\pi s/2)$, so that $c_s\to-1$ as
$s\uparrow2$.  Since $X$ is not almost surely zero, there are
$0<a<b<\infty$ with $\Pp\{a\leq|X|\leq b\}>0$.  For $1<s<2$,
\[
 \E|X|^s\geq\min\{a,a^2\}\Pp\{a\leq|X|\leq b\}=:c>0.
\]
Multiplying the assumed identity by $2-s$ therefore gives, for $s$
sufficiently close to $2$,
\[
 |d(s)|\leq
 \frac{|(2-s)(Q_+(s)+Q_-(s))|}{c|c_s|}\longrightarrow0.
\]
Continuity of $d$ at $2$ proves the assertion.
\end{proof}

We now show that any two subsequential limits agree at the exponent stated in
Theorem~\ref{thm:countable-main}.

Suppose first that $a_X=1$, and suppose for contradiction that
$d\not\equiv0$.  If $d(1)\neq0$, the quotients $F_A/d$ and $F_B/d$ continue
$A$ and $B$ through $1$.  If $d(1)=0$, then, for real $s\uparrow1$,
\[
 A(s)\leq1+\E X_+,
 \qquad B(s)\leq1+\E X_-,
\]
and Lemma~\ref{lem:countable-removability} again makes both quotients
removable at $1$.  At least one of $A$ and $B$ has abscissa $1$, so either
alternative continues that transform through its Landau singularity.  Hence
\begin{equation}
 d\equiv0\qquad(a_X=1).
 \label{eq:countable-case-one}
\end{equation}

Suppose next that $1<a_X<2$.  The identities
\eqref{eq:countable-FA}--\eqref{eq:countable-FB} extend by the identity
theorem as far as the corresponding moment transforms are holomorphic.  If
$d(a_X)\neq0$, the appropriate quotient continues the transform whose
abscissa is $a_X$ through that point, in conflict with
Lemma~\ref{lem:countable-landau}.  Therefore
\begin{equation}
 d(a_X)=0\qquad(1<a_X<2).
 \label{eq:countable-case-interior}
\end{equation}

Finally, suppose that $a_X\geq2$.  For real $1<s<2$, adding the two
continued signed identities gives
\begin{equation}
 2d(s)\Gamma(-s)\cos(\pi s/2)\E|X|^s
 =Q_+(s)+Q_-(s).
 \label{eq:countable-boundary-two}
\end{equation}
Since
\[
 2(2-s)\Gamma(-s)\cos(\pi s/2)\longrightarrow-1,
\]
equations \eqref{eq:countable-Q-boundary} and
\eqref{eq:countable-boundary-two}, together with
Lemma~\ref{lem:countable-second-moment-boundary}, imply $d(2)=0$.

\begin{proof}[Proof of Theorem~\ref{thm:countable-main}]
The preceding argument shows that all subsequential limits have the same
value at the exponent $p$ specified in the theorem; when $a_X=1$, it shows
this separately for every fixed $p\in(1,2)$.  Denote the common value at
the selected exponent by $c$ (or by $c_p$ in the latter case).

Every sequence $t_n\to\star$ has a subsequence along which the normal
families above converge locally uniformly.  Along that subsequence,
$m_{t_n}(p)\to c$.  The usual subsequence criterion therefore gives
$m_t(p)\to c$ as $t\to\star$.  Since the weights are
subprobabilities, $0\leq c\leq1$.

If $c=0$, then
\[
 \E P_*(t)^p
 \leq\E\sum_jP_j(t)^p=m_t(p)\longrightarrow0.
\]
Thus $P_*(t)\to0$ in probability, and
Lemma~\ref{lem:countable-wlln} would give $R_t\to0$ in probability,
contrary to the nondegeneracy of $R$.  Hence $c>0$.
\end{proof}

\begin{proof}[Proof of Theorem~\ref{thm:countable-above-one}]
For $\sigma>1$, the subprobability property gives
\begin{equation}
 |m_t(s)|\leq m_t(\sigma)\leq m_t(1)\leq1,
 \qquad \Re s=\sigma.
 \label{eq:above-one-normality}
\end{equation}
Lemma~\ref{lem:countable-dirichlet} therefore shows that $(m_t)$ is a
normal family on $\{\Re s>1\}$.

Fix a compact set $K\subset\{1<\Re s<\min(a_X,2)\}$.  Choose
\[
 \sup_{s\in K}\Re s<q<\min(a_X,2).
\]
Conditionally on the weights, add an atom of mass
$1-\sum_jP_j(t)$ carrying the value zero.  Jensen's inequality then gives
\begin{equation}
 \E\bigl(|R_t|^q\mid P(t)\bigr)
 \leq \E|X|^q\sum_jP_j(t)
 \leq \E|X|^q.
 \label{eq:above-one-jensen}
\end{equation}
For $\varepsilon\in\{-1,1\}$ put
\[
 K_{t,\varepsilon}(u)=\E e^{i\varepsilon uR_t},\qquad
 \widehat L_{t,\varepsilon}(s)
 =\int_0^\infty u^{-s-1}\{K_{t,\varepsilon}(u)-1\}\dd u.
\]
Because $\E R_t=0$, the inequality
$|e^{ix}-1-ix|\leq C_q|x|^q$, valid for $1<q\leq2$, and
\eqref{eq:above-one-jensen} control the integral at zero; the bound by two
controls it at infinity.  More explicitly,
\[
 |K_{t,\varepsilon}(u)-1|
 \leq C_qu^q\E|R_t|^q
 \leq C_qu^q\E|X|^q.
\]
Weak convergence of $R_t$, together with these majorants, shows that
$\widehat L_{t,\varepsilon}$ has a unique local limit on
$1<\Re s<\min(a_X,2)$.  The compact set was arbitrary.

Reuse $N_{P(t)}^\varepsilon$ from Section~\ref{sec:countable-remainder} and
set
\[
 \widehat J_{t,\varepsilon}(s)
 =\int_0^\infty u^{-s-1}
   \E N_{P(t)}^\varepsilon(u)\dd u.
\]
The small-$u$ part of Lemma~\ref{lem:countable-remainder} requires no
power-sum bound at an exponent below one and gives
\[
 |\E N_{P(t)}^\varepsilon(u)|\leq u^2\delta(u).
\]
At infinity,
\[
 |N_{P(t)}^\varepsilon(u)|
 \leq2+\sum_j|\chi(\varepsilon uP_j(t))-1|
 \leq2+u\E|X|.
\]
Consequently $(\widehat J_{t,\varepsilon})_t$ is normal on
$1<\Re s<2$, and the proof of
\eqref{eq:countable-abelian-boundary} applies verbatim to give
\begin{equation}
 (2-s)\sup_t|\widehat J_{t,\varepsilon}(s)|\longrightarrow0
 \qquad(s\uparrow2).
 \label{eq:above-one-abelian-boundary}
\end{equation}

When $\E X^2<\infty$, the same remainder has a stronger bound.  Put
$C=\E X^2/2$.  Centering gives
\[
 |\chi(v)-1|
 =\bigl|\E(e^{ivX}-1-ivX)\bigr|\leq Cv^2.
\]
For a deterministic subprobability sequence $p$, set
$z_j=\chi(\varepsilon up_j)-1$.  Then
$\sum_j|z_j|\leq Cu^2\sum_jp_j^2\leq Cu^2$.  Expanding finite products
as in Lemma~\ref{lem:countable-remainder}, and then passing to the
countable product, yields
\begin{equation}
 |N_p^\varepsilon(u)|
 \leq e^{Cu^2}-1-Cu^2
 \leq \tfrac12 C^2e^C u^4,
 \qquad 0<u\leq1.
 \label{eq:above-one-fourth-order-remainder}
\end{equation}
Together with the bound $2+u\E|X|$ at infinity, this shows that
$(\widehat J_{t,\varepsilon})_t$ is normal on $1<\Re s<4$ whenever
the variance is finite.

For $1<\Re s<\min(a_X,2)$, absolute Fubini and the substitution
$v=uP_j(t)$ yield
\begin{equation}
 \widehat L_{t,\varepsilon}(s)
 =m_t(s)\widehat\kappa_\varepsilon(s)
  +\widehat J_{t,\varepsilon}(s),
 \label{eq:above-one-decomposition}
\end{equation}
where
\[
 \widehat\kappa_\varepsilon(s)
 =\int_0^\infty u^{-s-1}
   \E(e^{i\varepsilon uX}-1-i\varepsilon uX)\dd u.
\]
Indeed, for every real $\sigma$ in this strip,
Lemma~\ref{lem:countable-scalar-mellin} gives
\[
 \int_0^\infty u^{-\sigma-1}
 \E|e^{i\varepsilon uX}-1-i\varepsilon uX|\dd u
 =C_\sigma^{(2)}\E|X|^\sigma<\infty.
\]
After the substitution $v=uP_j(t)$, summing the same bound over $j$ costs
only the factor $m_t(\sigma)\leq1$.  This justifies all interchanges above,
and the second identity in the lemma gives
\begin{align*}
 \widehat\kappa_+(s)
 &=\Gamma(-s)\{e^{-i\pi s/2}A(s)+e^{i\pi s/2}B(s)\},\\
 \widehat\kappa_-(s)
 &=\Gamma(-s)\{e^{i\pi s/2}A(s)+e^{-i\pi s/2}B(s)\}.
\end{align*}

Compare two arbitrary subsequential limits $m^{(1)}$ and $m^{(2)}$ of
$(m_t)$, and after further extraction complete each with subsequential limits
of the two normal families $\widehat J_{t,\varepsilon}$; when the variance
is finite, take these limits on $1<\Re s<4$.  The unique limit
of $\widehat L_{t,\varepsilon}$ and
\eqref{eq:above-one-decomposition} imply, with
\[
 d=m^{(1)}-m^{(2)},\qquad
 Q_\varepsilon=\widehat J_\varepsilon^{(2)}
               -\widehat J_\varepsilon^{(1)},
\]
that $d(s)\widehat\kappa_\varepsilon(s)=Q_\varepsilon(s)$.  Inverting the
two signed equations exactly as in
\eqref{eq:countable-FA}--\eqref{eq:countable-FB} gives
\begin{align}
 d(s)A(s)
 &=\frac{\Gamma(1+s)}{2\pi i}
   \{e^{-i\pi s/2}Q_+(s)-e^{i\pi s/2}Q_-(s)\},
 \label{eq:above-one-positive}\\
 d(s)B(s)
 &=\frac{\Gamma(1+s)}{2\pi i}
   \{-e^{i\pi s/2}Q_+(s)+e^{-i\pi s/2}Q_-(s)\}.
 \label{eq:above-one-negative}
\end{align}
The right-hand sides are holomorphic on $1<\Re s<2$, and on
$1<\Re s<4$ when the variance is finite.

If $1<a_X<2$ and $d(a_X)\neq0$, the quotient corresponding to a
one-sided moment transform whose abscissa is $a_X$ continues that transform
through its Landau singularity.  Hence $d(a_X)=0$.  If in addition
$\E|X|^{a_X}<\infty$ and $d\not\equiv0$, that transform stays bounded on
the real axis as $s\uparrow a_X$; Lemma~\ref{lem:countable-removability}
makes the relevant quotient in \eqref{eq:above-one-positive} or
\eqref{eq:above-one-negative} removable, again contradicting
Lemma~\ref{lem:countable-landau}.  Thus in the boundary-moment case every
two subsequential limits coincide on $\{\Re s>1\}$.

If $a_X\geq2$, adding the two signed equations gives, for $1<s<2$,
\[
 2d(s)\Gamma(-s)\cos(\pi s/2)\E|X|^s
 =Q_+(s)+Q_-(s).
\]
Since
$2(2-s)\Gamma(-s)\cos(\pi s/2)\to-1$,
\eqref{eq:above-one-abelian-boundary} implies $d(2)=0$, whether
$\E X^2$ is finite or infinite, by
Lemma~\ref{lem:countable-second-moment-boundary}.

It remains to prove convergence of the full profile when
$2\leq a_X<4$ and $\E|X|^{a_X}<\infty$.  The variance is then finite,
so $Q_+$ and $Q_-$ are holomorphic on $1<\Re s<4$.  Adding the signed
identities and using the reflection formula gives, initially for
$1<\Re s<2$,
\begin{equation}
 d(s)\E|X|^s=H(s),\qquad
 H(s)=-\frac{\Gamma(1+s)\sin(\pi s/2)}{\pi}
       \{Q_+(s)+Q_-(s)\}.
 \label{eq:above-one-absolute-moment}
\end{equation}
The function $H$ is holomorphic on $1<\Re s<4$, including $s=2$.
Since $s\mapsto\E|X|^s$ is holomorphic on $1<\Re s<a_X$, the
identity theorem extends \eqref{eq:above-one-absolute-moment} throughout
that strip.  If $d\not\equiv0$, the quotient $H/d$ is meromorphic in
a neighborhood of $a_X$ and agrees there, to the left of $a_X$, with
$\E|X|^s$.  This moment transform stays bounded on the real axis as
$s\uparrow a_X$, by the finite critical moment assumption.
Lemma~\ref{lem:countable-removability} therefore makes $H/d$ holomorphic
at $a_X$, contradicting Lemma~\ref{lem:countable-landau} applied to
$|X|$.  Thus $d\equiv0$.  The argument includes $a_X=2$: it continues
the identity \eqref{eq:above-one-absolute-moment}, without extending
the defining integral for $\widehat\kappa_\varepsilon$ beyond its
strip of absolute convergence.

Normality and the subsequence criterion now give the asserted endpoint
convergence.  The limits lie in $[0,1]$ by
\eqref{eq:above-one-normality}.  A zero limit would force
$P_*(t)\to0$ in probability and then $R_t\to0$ in probability by
Lemma~\ref{lem:countable-wlln}; hence each asserted limit is positive.
When $1<a_X<4$ and $\E|X|^{a_X}<\infty$, the unique subsequential
limit is the local uniform limit $m$, and the same argument shows
$m(q)>0$ for every real $q>1$.
\end{proof}

The restriction $a_X<4$ comes from the fourth-order bound on the nonlinear
remainder used above.  We do not claim that $4$ is an optimal threshold
for convergence of the full power-sum profile.

\begin{remark}[The power-sum bound for exponents below one can fail]
\label{rem:countable-above-one-sharpness}
For $n\geq1$, take the deterministic probability vector
\[
 P^{(n)}=\left(\frac12,
   \underbrace{\frac1{2n},\ldots,\frac1{2n}}_{n\ \mathrm{times}},
   0,\ldots\right).
\]
If $X$ is centered, nonconstant, and $a_X>1$, then the weighted weak law
gives
\[
 \sum_jP_j^{(n)}X_j\Rightarrow\frac X2,
\]
whereas, for every $r<1$,
\[
 \sum_j(P_j^{(n)})^r
 =2^{-r}\{1+n^{1-r}\}\longrightarrow\infty.
\]
Thus Theorem~\ref{thm:countable-above-one} covers nondegenerate convergence
without the bound \eqref{eq:countable-r-bound} required by
Theorem~\ref{thm:countable-main}.
\end{remark}

\begin{remark}[The full family when $a_X=1$]
\label{rem:countable-full-family-at-one}
When $a_X=1$, the proof of Theorem~\ref{thm:countable-main} yields more
than its real-power conclusion.  With $r$ as in
\eqref{eq:countable-r-bound}, any two subsequential limits of $(m_t)$
agree identically on $\{\Re s>r\}$.  Normality and the subsequence criterion therefore imply
that $m_t$ converges locally uniformly on this half-plane to a holomorphic
limit $m$.  For every real $q>1$, one has $m(q)>0$: otherwise
$\E P_*(t)^q\leq m_t(q)\to0$, and the weighted weak law would make the
limiting marked mean degenerate.  Below we need only one exponent $p>1$ with $m(p)>0$.
\end{remark}

\begin{remark}[Expected power sums do not determine a partition law]
\label{rem:countable-not-law-determining}
Even the full expected power-sum profile in
Theorem~\ref{thm:countable-above-one} does not identify the law of a
random mass partition.  Indeed, let
\begin{align*}
 a&=(1/2,1/4,1/4,0,\ldots),\\
 b&=(1/2,1/2,0,\ldots),\\
 c&=(1/4,1/4,1/4,1/4,0,\ldots).
\end{align*}
For every $q>1$,
\[
 \sum_i a_i^q
 =\frac12\sum_i b_i^q+\frac12\sum_i c_i^q,
\]
although $\delta_a$ and $\frac12\delta_b+\frac12\delta_c$ are different
laws.  In the L\'evy setting below, the expected power sum identifies a
regular-variation regime through a Tauberian theorem.  The limiting law of
the normalized jumps is then obtained separately by point-process
convergence and control of the small jumps.
\end{remark}

\begin{remark}[Convergence of the total mass is not implicit]
\label{rem:countable-mass-sharpness}
Let $X$ be symmetric strictly $\alpha$-stable, $1<\alpha\leq2$, with
$\E e^{iuX}=e^{-|u|^\alpha}$.  For $t\geq1$, let
$P(t)=(1/2,1/2,0,\ldots)$ when $\lfloor t\rfloor$ is even and
$P(t)=(2^{1/\alpha-1},0,\ldots)$ otherwise.
The weighted sums have the same nondegenerate law for every $t\geq1$, and the
$r$th power sums are uniformly bounded for every $r\in(0,1)$.  However,
the total masses alternate between $1$ and $2^{1/\alpha-1}$.  Thus
convergence of the marked mean and the bound
\eqref{eq:countable-r-bound} do not force $m_t(1)$ to converge.
\end{remark}

\section{Poissonian power-sum calculus}
\label{sec:poisson-calculus}

We now introduce the Poissonian identities shared by iid weights and L\'evy
jumps.  Let $\nu$ be a nonzero measure on $(0,\infty)$,
possibly finite, satisfying
$\int_0^\infty(1\wedge x)\,\nu(\dd x)<\infty$, and put
\[
 \lambda=\nu(0,\infty)\in(0,\infty],
 \qquad
 \Phi(q)=\int_0^\infty(1-e^{-qx})\,\nu(\dd x).
\]
Let $(J_i(t))$ enumerate the atoms of a Poisson random measure with
intensity $t\nu$, set $V_t=\sum_iJ_i(t)$, and define
\[
 Q_i(t)=\frac{J_i(t)}{V_t}\1_{\{V_t>0\}},
 \qquad
 m_t(p)=\E\sum_iQ_i(t)^p.
\]
The vector $Q(t)=(Q_i(t))_{i\geq1}$ is the Poisson-jump instance of the
abstract weights $P(t)$ in Section~\ref{sec:countable}; $m_t$ denotes the
same expected power sum.  Thus $Q(t)=0$ when there is no atom.  If
$\lambda<\infty$, this event has
probability $e^{-t\lambda}$; if $\lambda=\infty$, it is null.  Scalar
Poissonized ratios below use the same zero convention.  When a probability
mass partition is required for the constant-limit theorem, the empty
configuration is replaced by $e_1=(1,0,\ldots)$ and a fresh independent mark; when a
normalized random measure is required, it is replaced by the base measure
specified in that construction.  These modifications are absent under
infinite activity and have
vanishing probability in every finite-activity large-time application.

Write
\[
 A_p(q)=\int_0^\infty x^pe^{-qx}\,\nu(\dd x),
 \qquad
 e(q)=\frac{q\Phi'(q)}{\Phi(q)},
 \qquad q>0.
\]
Thus $e(q)$ is the derivative of $\log\Phi(q)$ with respect to $\log q$.
On $(0,\lambda)$, $\Phi^{-1}$ denotes the inverse of $\Phi$.  We use the
conventions $e^{-\infty}=0$, $\infty e^{-\infty}=0$, and
$\int_0^{t\lambda}=\int_0^\infty$ when $\lambda=\infty$.
Values assigned to $e(\Phi^{-1}(v/t))$ at the integration endpoints are
immaterial and may be chosen arbitrarily.

\begin{lemma}[Campbell--Mecke power-sum identity]
\label{lem:poisson-campbell}
For $p,t>0$,
\begin{equation}
 m_t(p)=\frac{t}{\Gamma(p)}
 \int_0^\infty q^{p-1}A_p(q)e^{-t\Phi(q)}\dd q.
 \label{eq:poisson-campbell}
\end{equation}
For $0<p<1$, either side may be infinite.  Moreover,
\[
 m_t(1)=1-e^{-t\lambda}.
\]
\end{lemma}

\begin{proof}
Campbell--Mecke gives
\[
 m_t(p)=t\int_0^\infty x^p
  \E\left[(x+V_t')^{-p}\right]\nu(\dd x),
\]
where $V_t'$ is an independent copy of $V_t$.  Insert
\[
 y^{-p}=\frac1{\Gamma(p)}\int_0^\infty q^{p-1}e^{-qy}\dd q,
 \qquad y>0,
\]
and use $\E e^{-qV_t'}=e^{-t\Phi(q)}$.  Tonelli's theorem proves
\eqref{eq:poisson-campbell}.  The formula at $p=1$ also follows directly
from $\sum_iQ_i(t)=\1_{\{V_t>0\}}$.
\end{proof}

The case $p=2$ expresses the power sum directly in terms of $e$.

\begin{lemma}[Exact quadratic power sum]
\label{lem:poisson-quadratic}
For every $t>0$,
\begin{equation}
 m_t(2)=t\lambda e^{-t\lambda}
 +\int_0^{t\lambda}ve^{-v}
 \left[1-e\!\left(\Phi^{-1}(v/t)\right)\right]\dd v.
 \label{eq:poisson-quadratic}
\end{equation}
\end{lemma}

\begin{proof}
Put $D(q)=q\Phi'(q)$.  Since $A_2(q)=-\Phi''(q)$,
Lemma~\ref{lem:poisson-campbell} gives
\[
 m_t(2)=-t\int_0^\infty q\Phi''(q)e^{-t\Phi(q)}\dd q.
\]
Integrate first over $[\varepsilon,M]$.  Since
$D'=\Phi'+q\Phi''$, integration by parts on that interval yields
\begin{align*}
 -t\int_\varepsilon^M q\Phi''(q)e^{-t\Phi(q)}\dd q
 &=t\int_\varepsilon^M\Phi'(q)e^{-t\Phi(q)}\dd q
   -t[D(q)e^{-t\Phi(q)}]_{\varepsilon}^{M}\\
 &\quad
   -t^2\int_\varepsilon^M q\Phi'(q)^2e^{-t\Phi(q)}\dd q.
\end{align*}
The boundary term vanishes as $\varepsilon\downarrow0$ because
$0<D(q)<\Phi(q)\to0$.  If $\lambda=\infty$, the same inequality gives
$D(M)e^{-t\Phi(M)}\to0$.  If $\lambda<\infty$, then $\nu$ is finite,
$Mx e^{-Mx}\leq e^{-1}$, and dominated convergence gives $D(M)\to0$.
Letting $\varepsilon\downarrow0$ and $M\to\infty$ therefore gives
\[
 m_t(2)=1-e^{-t\lambda}
 -t^2\int_0^\infty q\Phi'(q)^2e^{-t\Phi(q)}\dd q.
\]
Under $v=t\Phi(q)$, the last integral becomes
\[
 \int_0^{t\lambda}ve^{-v}e\!\left(\Phi^{-1}(v/t)\right)\dd v.
\]
Finally,
$1-e^{-a}=ae^{-a}+\int_0^a ve^{-v}\dd v$, also for $a=\infty$ under
the stated conventions, which proves \eqref{eq:poisson-quadratic}.
\end{proof}

The broader comparison of maxima and sums goes back at least to
Darling~\cite{Darling}.  Cohn and Hall studied the asymptotic equivalence
of deterministic weighted and unweighted sums; their
Theorems~1--2 characterize relative stability and its subsequential form
through the corresponding quadratic ratios \cite{CohnHall}.  For the
expected ratio of the sum of squares to the square of the sum, see
McLeish and O'Brien \cite{MO} and the exact formula of Fuchs and Joffe
\cite{FJ}.  Fuchs, Joffe, and Teugels subsequently gave an exact Laplace
formula and direct and converse asymptotic characterizations
\cite[equation~(2.1) and Theorems~5.3--5.5]{FJT}.
Mason and Zinn \cite{MZPriority} acknowledge that their Proposition~3 for
$0<\alpha<1$ had already been proved in \cite[Theorem~5.3]{FJT}.

For the converse arguments, the following differential inequality controls
where $e$ can remain close to one at either endpoint.

\begin{lemma}[A differential inequality for the Laplace exponent]
\label{lem:poisson-slope-dynamics}
For $x<\log\lambda$ when $\lambda<\infty$, and for $x\in\R$ when
$\lambda=\infty$, define
\[
 g(x)=e\!\left(\Phi^{-1}(e^x)\right).
\]
If $q=\Phi^{-1}(e^x)$, then
\begin{equation}
 g'(x)=1-g(x)+\frac{q\Phi''(q)}{\Phi'(q)}\leq1-g(x).
 \label{eq:poisson-slope-dynamics}
\end{equation}
Consequently, with $u=1-g$, the function $x\mapsto e^xu(x)$ is
nondecreasing and
\begin{equation}
 u(y)\leq e^{x-y}u(x),\qquad y\leq x.
 \label{eq:poisson-left-window}
\end{equation}
\end{lemma}

\begin{proof}
Differentiate $e(q)=q\Phi'(q)/\Phi(q)$ and use
$\dd x/\dd q=\Phi'(q)/\Phi(q)$.  The last term in
\eqref{eq:poisson-slope-dynamics} is nonpositive by concavity.  Thus
$u'+u\geq0$, which proves the final assertion.
\end{proof}

The corresponding fractional identity generates bounds on expected power sums
for exponents below one.

\begin{lemma}[Fractional power-sum representation]
\label{lem:poisson-fractional}
Let $0<r<1$ and $D(q)=q\Phi'(q)$.  For every $q>0$,
\begin{equation}
 q^rA_r(q)=\frac1{\Gamma(1-r)}
 \int_0^\infty(e^h-1)^{-r}D(qe^h)\dd h.
 \label{eq:poisson-fractional}
\end{equation}
Both sides may be infinite.  Moreover,
\begin{equation}
 m_t(r)=\frac{t}{\Gamma(r)\Gamma(1-r)}
 \int_0^\infty(e^h-1)^{-r}I_t(h)\dd h,
 \label{eq:poisson-fractional-representation}
\end{equation}
where
\begin{equation}
 I_t(h)=\int_0^\lambda
 \exp\{-t\Phi(e^{-h}\Phi^{-1}(v))\}\dd v.
 \label{eq:poisson-I-def}
\end{equation}
\end{lemma}

\begin{proof}
For $x,q>0$,
\[
 x^re^{-qx}=\frac{x}{\Gamma(1-r)}
 \int_0^\infty z^{-r}e^{-(q+z)x}\dd z.
\]
Integrate with respect to $\nu(\dd x)$, multiply by $q^r$, and set
$z=q(e^h-1)$.  Tonelli's theorem proves
\eqref{eq:poisson-fractional}.  Insert that identity in
\eqref{eq:poisson-campbell}, apply Tonelli again, and set $s=qe^h$.
The remaining inner integral is
\[
 \int_0^\infty\Phi'(s)e^{-t\Phi(e^{-h}s)}\dd s.
\]
The substitution $v=\Phi(s)$ gives \eqref{eq:poisson-I-def} and completes
the proof.
\end{proof}

The following form of the Drasin--Shea theorem will turn the power-sum
limits into tail asymptotics in both ratio--Tauberian arguments.  For
measurable functions on \((0,\infty)\), write
\[
 (h*_{M}K)(x)=\int_0^\infty h(x/u)K(u)\,\frac{\dd u}{u},
 \qquad
 \mathcal M K(z)=\int_0^\infty u^{-z}K(u)\,\frac{\dd u}{u}.
\]

\begin{proposition}[Drasin--Shea ratio theorem]
\label{prop:poisson-mellin-ratio}
Let \(K\) be a nonnegative, measurable, nonzero kernel whose Mellin
transform has maximal interval of convergence
\[
 c<z<d,\qquad -\infty\leq c<0<d\leq\infty.
\]
Assume that \(\mathcal M K(z)\to\infty\) as \(z\downarrow c\) when
\(c>-\infty\), and as \(z\uparrow d\) when \(d<\infty\).  Let \(h\) be
nonnegative and locally bounded on \([0,\infty)\).  Suppose that, for some
\(\zeta>0\) and \(x_0\geq0\), either
\[
 h(\lambda x)\geq\zeta h(x)
 \quad\text{for all }x>x_0,\ 1\leq\lambda\leq2,
\]
or the same condition holds with the reverse inequality.  If the
convolution is finite and
\[
 \frac{(h*_{M}K)(x)}{h(x)}\longrightarrow\gamma\in(0,\infty),
 \qquad x\to\infty,
\]
then, for some \(\rho\in(c,d)\),
\[
 h\in\RV_\infty(\rho),
 \qquad
 \gamma=\mathcal M K(\rho).
\]
\end{proposition}

\begin{proof}
This is \cite[Proposition~1, alternative~(2.4)]{Rehak}, a maximal-interval
form of the Drasin--Shea theorem; see also
\cite[Theorem~5.2.3 and p.~274]{BGT} and \cite{DS,BinghamInoue}.
\end{proof}

\section{Normalized Poisson jumps and independent marking}
\label{sec:poisson-limits}

We next isolate the Poissonian limit results used in both scalar
applications and the final random-measure theorem.  They do not depend on
either converse classification.

For a measure $M$ and a measurable function $h$ integrable with respect to
$M$, write $Mh=\int h\,\dd M$.

Let $V$ be a nonzero, unkilled, driftless subordinator with L\'evy measure
$\nu$.  Write
\[
 \bar\nu(x)=\nu((x,\infty)),
 \qquad
 A(x)=\int_{(0,x]}y\,\nu(\dd y),
 \qquad
 I(x)=\int_0^x\bar\nu(y)\dd y.
\]
For $t>0$, let
\[
 N_t=\sum_{0<s\leq t}\delta_{\Delta V_s},
 \qquad
 V_t=\int x\,N_t(\dd x),
\]
and list the atoms of $N_t$ in nonincreasing order as
$J_{(1)}(t),J_{(2)}(t),\ldots$, padding with zeros.  Ties are broken by
auxiliary iid continuous variables, independent of all marks used below.
On $\{V_t>0\}$, set
\[
 Q_i(t)=\frac{J_{(i)}(t)}{V_t},
\]
write $Q^\downarrow(t)=(Q_i(t))_{i\geq1}$, and put
$Q^\downarrow(t)=0$ on $\{V_t=0\}$.  At the zero endpoint, we impose
infinite activity; at infinity the zero-mass probability tends to zero.

The ranked weights take values in $\K$, with the
topologies and dust convention introduced in
Section~\ref{sec:introduction}.

We first record the deterministic integrated-tail fact used in the dust
branch.  Its endpoint assumptions are part of the statement.

\begin{lemma}[Integrated-tail criterion]
\label{lem:poisson-integrated-tail}
Fix either $x\downarrow0$ under infinite activity, or $x\to\infty$ for a
nonzero L\'evy measure.  Then
\begin{equation}
 \frac{x\bar\nu(x)}{A(x)}\longrightarrow0
 \quad\Longleftrightarrow\quad
 I\in\RV_\star(0),
 \label{eq:poisson-integrated-tail}
\end{equation}
where $\star=0$ or $\infty$ is the specified endpoint.
\end{lemma}

\begin{proof}
Tonelli's theorem gives
\begin{equation}
 I(x)=A(x)+x\bar\nu(x).
 \label{eq:poisson-I-decomposition}
\end{equation}
Thus the ratio in \eqref{eq:poisson-integrated-tail} tends to zero if and
only if $x\bar\nu(x)/I(x)\to0$.  If the latter holds, then for
$\lambda>1$,
\[
 0\leq I(\lambda x)-I(x)
 \leq(\lambda-1)x\bar\nu(x)=o(I(x)).
\]
This gives $I(\lambda x)/I(x)\to1$; the case $0<\lambda<1$ follows by
applying the same comparison at $\lambda x$ with multiplier $1/\lambda$.
Conversely, for fixed $0<\lambda<1$, monotonicity of the tail gives
\[
 I(x)-I(\lambda x)
 \geq(1-\lambda)x\bar\nu(x).
\]
Slow variation and \eqref{eq:poisson-I-decomposition} now imply the ratio
limit.
\end{proof}

\subsection{Point processes and ranked normalization}

Write $\mathcal M_p((0,\infty])$ for the Radon point measures on
$(0,\infty]$, endowed with the vague topology; equivalently, such measures
have finitely many atoms in $(a,\infty]$ for every $a>0$.

\begin{lemma}[Ranked atoms under vague convergence]
\label{lem:measure:ranked-atoms}
Let $N_n,N$ be point measures on $(0,\infty]$ such that $N_n\to N$
vaguely.  Assume that $N$ is simple, has no atom at infinity, has finitely
many atoms in $(a,\infty]$ for every $a>0$, and has infinitely many atoms
in $(0,\infty)$.  If $x_i^{(n)}$ and $x_i$ are their atoms listed in
nonincreasing order, with zeros appended when necessary, then
$x_i^{(n)}\to x_i$ for every fixed $i$.
\end{lemma}

\begin{proof}
Fix $k$.  Choose $a\in(x_{k+1},x_k)$ that is not an atom of $N$, and
choose disjoint relatively compact intervals $I_1,\ldots,I_k$ around
$x_1,\ldots,x_k$, all contained in $(a,\infty]$, with
$N(\partial I_i)=0$.  Vague convergence gives, for all sufficiently large
$n$,
\[
 N_n(I_i)=1\quad(1\leq i\leq k),
 \qquad N_n((a,\infty])=k.
\]
Shrinking the intervals proves the coordinatewise convergence.
\end{proof}

\begin{lemma}[Discrete Scheff\'e lemma]
\label{lem:measure:discrete-scheffe}
Let $a^{(n)},a\in\ell^1_+$.  If $a_i^{(n)}\to a_i$ for every $i$ and
$\sum_i a_i^{(n)}\to\sum_i a_i$, then
$\lVert a^{(n)}-a\rVert_1\to0$.
\end{lemma}

\begin{proof}
For fixed $K$, the first $K$ coordinates converge in $\ell^1$, while
convergence of the total masses gives
\[
 \sum_{i>K}a_i^{(n)}\longrightarrow\sum_{i>K}a_i.
\]
Hence
\[
 \limsup_{n\to\infty}\lVert a^{(n)}-a\rVert_1
 \leq2\sum_{i>K}a_i,
\]
and the right-hand side tends to zero as $K\to\infty$.
\end{proof}

\begin{proposition}[Poisson point process and total mass]
\label{prop:poisson-point-total}
Fix an endpoint $\star\in\{0,\infty\}$, with infinite activity if
$\star=0$.  Suppose
\[
 \bar\nu\in\RV_\star(-\alpha),
 \qquad 0<\alpha<1.
\]
Let $b_t$ be an asymptotic generalized inverse of
$x\mapsto1/\bar\nu(x)$ at the relevant endpoint, so that
$t\bar\nu(b_t)\to1$ as $t\to\star$; its existence follows from asymptotic inversion for
monotone regularly varying functions~\cite{BGT}.  Then, on
$\mathcal M_p((0,\infty])\times\R_+$,
\[
 \left(\sum_{0<s\leq t}\delta_{\Delta V_s/b_t},\frac{V_t}{b_t}\right)
 \Longrightarrow(N_\alpha,S_\alpha)
 \qquad(t\to\star),
\]
where $N_\alpha$ is a Poisson random measure with intensity
$\nu_\alpha(\dd x)=\alpha x^{-\alpha-1}\dd x$ and
$S_\alpha=\int x\,N_\alpha(\dd x)\in(0,\infty)$ almost surely.
Necessarily, $b_t\downarrow0$ at the zero endpoint and $b_t\to\infty$ at
infinity.
\end{proposition}

\begin{proof}
The asserted finiteness and positivity follow directly from the stable
intensity:
\[
 \int_0^\infty(1\wedge x)\,\nu_\alpha(\dd x)
 =\frac{\alpha}{1-\alpha}+1<\infty,
 \qquad
 \nu_\alpha(0,\infty)=\infty.
\]
The Poisson-integral criterion therefore gives $S_\alpha<\infty$ almost
surely, while the point measure is nonempty almost surely; thus
$S_\alpha>0$ almost surely.

Fix an arbitrary sequence $t_n\to\star$, set $b_n=b_{t_n}$, and write
\[
 \mu_n(B)=t_n\nu(b_nB),
 \qquad
 \mathcal N_n=\sum_{0<s\leq t_n}\delta_{\Delta V_s/b_n}.
\]
For every $x>0$,
\[
 \mu_n((x,\infty])=t_n\bar\nu(b_nx)\longrightarrow x^{-\alpha}.
\]
Thus $\mu_n\to\nu_\alpha$ vaguely.  Equivalently, for every
$f\in C_c^+((0,\infty])$,
\[
 \E e^{-\mathcal N_nf}
 =\exp\left\{-\int(1-e^{-f})\dd\mu_n\right\}
 \longrightarrow
 \exp\left\{-\int(1-e^{-f})\dd\nu_\alpha\right\}.
\]
Hence $\mathcal N_n\Rightarrow N_\alpha$; see also
Resnick~\cite{Resnick1986} for this point-process form of regular
variation.

For $0<\varepsilon<M<\infty$, put
\[
 T_{\varepsilon,M}(N)=\int_{(\varepsilon,M]}x\,N(\dd x).
\]
At point measures with no atom at $\varepsilon$ or $M$, the map
$N\mapsto(N,T_{\varepsilon,M}(N))$ is continuous.  Therefore
\[
 (\mathcal N_n,T_{\varepsilon,M}(\mathcal N_n))
 \Longrightarrow
 (N_\alpha,T_{\varepsilon,M}(N_\alpha)).
\]
The lower truncation is controlled on both sides.  Karamata's theorem gives
\begin{align*}
 \E\int_{(0,\varepsilon]}x\,\mathcal N_n(\dd x)
 &=\frac{t_n}{b_n}A(\varepsilon b_n)
 \longrightarrow
 \frac{\alpha}{1-\alpha}\varepsilon^{1-\alpha},\\
 \E\int_{(0,\varepsilon]}x\,N_\alpha(\dd x)
 &=\frac{\alpha}{1-\alpha}\varepsilon^{1-\alpha}.
\end{align*}
Consequently, for every $\eta>0$,
\begin{align*}
 \limsup_{n\to\infty}
 \Pp\left\{\int_{(0,\varepsilon]}x\,\mathcal N_n(\dd x)>\eta\right\}
 &\leq\frac{\alpha}{(1-\alpha)\eta}\varepsilon^{1-\alpha},\\
 \Pp\left\{\int_{(0,\varepsilon]}x\,N_\alpha(\dd x)>\eta\right\}
 &\leq\frac{\alpha}{(1-\alpha)\eta}\varepsilon^{1-\alpha}.
\end{align*}
For the upper truncation, use the largest atom:
\[
 \Pp\{\mathcal N_n((M,\infty])>0\}
 =1-e^{-t_n\bar\nu(Mb_n)}
 \longrightarrow1-e^{-M^{-\alpha}},
\]
and the same formula holds for $N_\alpha$.  In particular, for every
$\eta>0$,
\[
 \begin{aligned}
 &\Pp\!\left\{
   \left|\int x\,\mathcal N_n(\dd x)
          -T_{\varepsilon,M}(\mathcal N_n)\right|>\eta
 \right\}\\
 &\quad\leq
 \Pp\!\left\{\int_{(0,\varepsilon]}x\,\mathcal N_n(\dd x)>\eta\right\}\\
 &\qquad+\Pp\{\mathcal N_n((M,\infty])>0\}.
 \end{aligned}
\]
The analogous inequality holds for $N_\alpha$.  Markov's inequality for
the lower truncation and the displayed upper error, followed in order by
$n\to\infty$, $\varepsilon\downarrow0$, and $M\to\infty$, prove
\[
 \left(\mathcal N_n,\int x\,\mathcal N_n(\dd x)\right)
 \Longrightarrow(N_\alpha,S_\alpha)
\]
by the converging-together lemma.  Since the endpoint sequence was
arbitrary and the state space is metrizable, the asserted endpoint
convergence follows.
\end{proof}

\begin{proposition}[Ranked normalized jumps]
\label{prop:measure:ranked-jumps}
If
\[
 \bar\nu\in\RV_\star(-\alpha),
 \qquad0\leq\alpha<1,
\]
with infinite activity when $\star=0$, then
\[
 Q^\downarrow(t)\Longrightarrow Q^{(\alpha)}\quad\text{in }\ell^1
 \qquad(t\to\star),
\]
where $Q^{(\alpha)}\sim\PD(\alpha,0)$ for $0<\alpha<1$, and
$Q^{(0)}=e_1=(1,0,\ldots)$.
\end{proposition}

\begin{proof}
For $0<\alpha<1$, fix an arbitrary endpoint sequence and apply
Proposition~\ref{prop:poisson-point-total}.  On a Skorokhod
representation, the point processes and their total masses converge almost
surely.  The limit is simple almost surely, so
Lemma~\ref{lem:measure:ranked-atoms} gives coordinatewise convergence of
the ranked atoms.  Since $S_\alpha>0$, division by the total mass gives
coordinatewise convergence of the normalized sequences.  Their total
masses converge to one; Lemma~\ref{lem:measure:discrete-scheffe} therefore
upgrades the convergence to $\ell^1$.  The normalized jumps of a stable
subordinator have law $\PD(\alpha,0)$~\cite{PitmanYor}.

For $\alpha=0$, the maximal-jump theorem
\cite[Theorem~1.1(ii), $k=0$]{KM14} gives
\[
 \frac{J_{(1)}(t)}{V_t}\longrightarrow1
 \quad\text{in probability}.
\]
On $\{V_t>0\}$,
\[
 \lVert Q^\downarrow(t)-e_1\rVert_1
 =2\left(1-\frac{J_{(1)}(t)}{V_t}\right).
\]
The zero-mass event is absent at zero and has probability tending to zero
at infinity.  This proves the claim.
\end{proof}

\subsection{Independent marking and dust}

Let $E$ be a Polish space and $H\in\cP(E)$.  For
$q\in\K$ and $z\in E^{\mathbb N}$, define
\begin{equation}
 \mathcal M_H(q,z)
 =\sum_{i\geq1}q_i\delta_{z_i}
  +\left(1-\sum_{i\geq1}q_i\right)H.
 \label{eq:measure:marking-map}
\end{equation}
The second term is part of the map: it records coordinate mass lost in the
product topology.  When $H$ has atoms, independently sampled locations can
coincide, and the corresponding weights then add in the resulting measure.

Choose a bounded compatible metric on $E$, a countable
convergence-determining family $(h_k)_{k\geq1}$ in the unit ball of the
bounded-Lipschitz functions, and set
\begin{equation}
 d_w(\mu_1,\mu_2)=\sum_{k\geq1}2^{-k}
   (|\mu_1 h_k-\mu_2 h_k|\wedge1).
 \label{eq:measure:weak-metric}
\end{equation}
Then $d_w$ metrizes weak convergence on $\cP(E)$.

\begin{lemma}[Measurability of independent marking]
\label{lem:measure:marking-borel}
The map $(q,z)\mapsto\mathcal M_H(q,z)$ from
$\K\times E^{\mathbb N}$ to $\cP(E)$ is Borel
measurable.
\end{lemma}

\begin{proof}
For each $k$,
\[
 \mathcal M_H(q,z)h_k
 =Hh_k+\sum_{i\geq1}q_i\{h_k(z_i)-Hh_k\}.
\]
The series is absolutely convergent, and its partial sums are Borel
functions of $(q,z)$.  Hence every coordinate in
\eqref{eq:measure:weak-metric} is Borel.  These coordinates generate the
Borel $\sigma$-field of $\cP(E)$.
\end{proof}

\begin{proposition}[Independent marking with dust]
\label{prop:measure:marking-dust}
Let $Q^{(n)}\Rightarrow Q$ in the product topology of
$\K$.  Independently of all weights, let
$Z_1,Z_2,\ldots$ be iid with law $H$.  Then
\begin{equation}
 \bigl(Q^{(n)},\mathcal M_H(Q^{(n)},Z)\bigr)
 \Longrightarrow
 \bigl(Q,\mathcal M_H(Q,Z)\bigr)
 \label{eq:measure:marking-convergence}
\end{equation}
in the product topology times the weak topology.  The convergence holds
jointly after adjoining
\[
 \mathcal M_H(Q^{(n)},Z)g_j,
 \qquad1\leq j\leq m,
\]
for every finite family $g_1,\ldots,g_m\in L^1(H)$.  If, in addition,
$Q^{(n)}\Rightarrow Q$ in the $\ell^1$ topology, then
\eqref{eq:measure:marking-convergence} holds with the $\ell^1$ topology on
its first coordinate.

Under the product-topology hypothesis, if $Q=0$ almost surely, then
\begin{equation}
 \sum_{j=1}^m
 \E\left|\mathcal M_H(Q^{(n)},Z)g_j-Hg_j\right|\longrightarrow0.
 \label{eq:measure:marking-dust-L1}
\end{equation}
\end{proposition}

\begin{proof}
Use Skorokhod's theorem to couple the weights so that
$Q^{(n)}\to Q$ coordinatewise almost surely, and attach an iid sequence
$Z$ with law $H$, independent of all coupled weights.  For bounded $h$
and $K\geq1$, write
\[
 \mathcal M_H(Q^{(n)},Z)h
 =Hh+\sum_{i\leq K}Q_i^{(n)}\{h(Z_i)-Hh\}+R_{n,K},
\]
where
\[
 R_{n,K}=\sum_{i>K}Q_i^{(n)}\{h(Z_i)-Hh\}.
\]
Conditionally on the weights,
\[
 \E(R_{n,K}^2\mid Q^{(n)})
 =\Var_H(h)\sum_{i>K}(Q_i^{(n)})^2
 \leq\Var_H(h)Q_{K+1}^{(n)}.
\]
The head converges almost surely for fixed $K$.  Bounded convergence and
$Q_{K+1}\leq1/(K+1)$ give
\[
 \limsup_{n\to\infty}\E R_{n,K}^2
 \leq\frac{\Var_H(h)}{K+1};
\]
the same bound holds for the limiting tail.  Thus the integrals converge
in probability for every $h_k$.  For fixed $r$, the first $r$ summands
in $d_w$ converge in probability, while the remaining summands contribute
at most $\sum_{k>r}2^{-k}$.  Therefore
\[
 d_w\bigl(\mathcal M_H(Q^{(n)},Z),\mathcal M_H(Q,Z)\bigr)
 \longrightarrow0
 \quad\text{in probability}.
\]
Together with the coupled convergence of the weights, this proves the
joint random-measure assertion.

The same head--tail proof works for bounded Borel functions because the
marks are shared under the coupling.  For $g\in L^1(H)$, the marked
series is absolutely convergent almost surely, since, conditionally on
$Q=q$,
\[
 \E\sum_iq_i|g(Z_i)|=\left(\sum_iq_i\right)H|g|<\infty.
\]
If $g_L=(-L)\vee(g\wedge L)$, then, uniformly in $n$,
\[
 \E\left|\mathcal M_H(Q^{(n)},Z)(g-g_L)\right|
 \leq H|g-g_L|\longrightarrow0,
\]
and the same estimate holds for the limit.  Applying the
converging-together lemma to a finite family proves the assertion for
$L^1(H)$ test functions.  Under $\ell^1$ convergence, use an $\ell^1$ Skorokhod
coupling and the bound
\[
 \left|\sum_i(Q_i^{(n)}-Q_i)h(Z_i)\right|
 +\left|\sum_i(Q_i^{(n)}-Q_i)\right||Hh|
 \leq2\lVert h\rVert_\infty\lVert Q^{(n)}-Q\rVert_1.
\]

If $Q=0$ almost surely, then $Q_1^{(n)}\to0$ in probability and
$\E Q_1^{(n)}\to0$, since $0\leq Q_1^{(n)}\leq1$.  For bounded $h$,
conditional centering and independence give
\[
 \E\left|\mathcal M_H(Q^{(n)},Z)h-Hh\right|^2
 =\Var_H(h)\E\sum_i(Q_i^{(n)})^2
 \leq\Var_H(h)\E Q_1^{(n)}\longrightarrow0.
\]
For $g\in L^1(H)$ and $g_L=(-L)\vee(g\wedge L)$, the preceding truncation
bound yields
\[
 \limsup_{n\to\infty}
 \E\left|\mathcal M_H(Q^{(n)},Z)g-Hg\right|
 \leq2H|g-g_L|.
\]
Letting $L\to\infty$ and summing over the finite family proves
\eqref{eq:measure:marking-dust-L1}.
\end{proof}

\begin{lemma}[Nondegeneracy of an independently marked partition]
\label{lem:measure:marked-nondegenerate}
Let $Q\in\K$ be independent of iid $H$-marks
$(Z_i)$, and suppose $Q_1>0$ almost surely.  If $f\in L^1(H)$ is
not constant $H$-almost surely, then $\mathcal M_H(Q,Z)f$ is
nondegenerate.
\end{lemma}

\begin{proof}
If independent real random variables $Y$ and $R$ have an almost surely
constant sum, then both are almost surely constant: from $Y+R=c$ one has
$\sigma(Y)\subset\sigma(R)$ modulo null sets, while the two $\sigma$-fields
are independent.

If $\mathcal M_H(Q,Z)f$ were constant, its conditional law given $Q=q$
would be the same point mass for $\Law(Q)$-almost every $q$.  For such a
$q$, split the absolutely convergent sum into the independent terms
\[
 q_1f(Z_1)
 \quad\text{and}\quad
 \sum_{i\geq2}q_if(Z_i)+\left(1-\sum_iq_i\right)Hf.
\]
Since $q_1>0$, the preceding fact would make $f(Z_1)$ constant, a
contradiction.
\end{proof}

\section{Poissonized iid weights and Breiman's conjecture}
\label{sec:breiman}

We first prove the nondegenerate part of
Theorem~\ref{thm:intro-iid-regime}, using the power-sum principle and the
Poissonian limits established above.

Let $(X_j)_{j\geq1}$ and $(Y_j)_{j\geq1}$ be independent iid
sequences, with generic elements $X$ and $Y$.  Throughout this section,
\begin{equation}
 Y\geq0,\qquad \Pp\{Y>0\}>0,\qquad \E|X|<\infty,
 \label{eq:breiman-standing}
\end{equation}
and
\[
 S_n=\sum_{j=1}^nY_j,
 \qquad
 T_n=\frac{\sum_{j=1}^nX_jY_j}{S_n}\1_{\{S_n>0\}}.
\]
Set $S_0=0$; thus $T_n=0$ when $S_n=0$.  Write
\[
 \bar G(x)=\Pp\{Y>x\},\qquad x\geq0.
\]

The constant-limit case will be included in
Section~\ref{sec:iid-classification}.

\begin{theorem}[Breiman's conjecture]
\label{thm:breiman}
Assume \eqref{eq:breiman-standing}.  Then $T_n$ converges in distribution,
along the full sequence of positive integers, to a nondegenerate random
variable if and only if $X$ is nonconstant and
\begin{equation}
 \bar G\in\RV_\infty(-\beta)
 \quad\text{for some }\beta\in[0,1).
 \label{eq:breiman-tail-rv}
\end{equation}
The index $\beta$ is unique.
In that case,
\begin{equation}
 T_n\Rightarrow W_\beta\overset{d}{=}
 \begin{cases}
  \displaystyle\sum_{i\geq1}Q_iX_i,
    &(Q_i)^\downarrow\sim\PD(\beta,0),\quad0<\beta<1,\\[2mm]
  X,&\beta=0,
 \end{cases}
 \label{eq:breiman-limit-law}
\end{equation}
where the weights and marks on the right are independent.
\end{theorem}

Kevei and Mason's converse assumes that
$-\log\operatorname{Re}\varphi_{X-\E X}(u)$ is regularly varying at zero
with index in $(1,2]$ \cite[Theorem~1 and Corollary~1]{KM16}.
Theorem~\ref{thm:breiman} requires only integrability and nonconstancy of
the mark.

Centering entails no loss of generality.  If $\mu=\E X$ and
$\widehat X_j=X_j-\mu$, then the corresponding ratio satisfies
\begin{equation}
 \widehat T_n=T_n-\mu\1_{\{S_n>0\}}.
 \label{eq:breiman-centering}
\end{equation}
If $\lambda=\Pp\{Y>0\}$, then
$\Pp\{S_n=0\}=(1-\lambda)^n$.  Hence weak convergence and
nondegeneracy are preserved by \eqref{eq:breiman-centering}.  A
nondegenerate limit therefore forces $X$ to be nonconstant: if $X$ were
constant, the ratio would converge to that constant because
$\Pp\{S_n=0\}\to0$.  Accordingly, after centering we assume until the end
of the necessity proof that
\begin{equation}
 \E X=0,\qquad 0<\E|X|<\infty.
 \label{eq:breiman-centered}
\end{equation}

\subsection{Poissonization and a bound on the Laplace exponent}

We first record a transfer estimate that requires no moment assumption on
$Y$.  Besides justifying depoissonization, it will also be useful in the
large-time argument for subordinators.

\begin{lemma}[Poissonization and depoissonization]
\label{lem:breiman-depoissonization}
Couple all the ratios $T_m$ through the same iid sequences.  For integers
$m,n\geq0$, with $T_0=0$,
\begin{equation}
 \E|T_m-T_n|
 \leq
 2\E|X|\frac{|m-n|}{m\vee n},
 \label{eq:breiman-depoisson-bound}
\end{equation}
where the quotient is interpreted as zero when $m=n=0$.  Consequently, if
$N_t$ is Poisson with mean $t$ and independent of the sequences, then
\begin{equation}
 T_{N_t}-T_{\lfloor t\rfloor}\longrightarrow0
 \quad\text{in }L^1\qquad(t\to\infty).
 \label{eq:breiman-depoisson-limit}
\end{equation}
In particular, $(T_n)$ converges weakly along the full sequence if and only
if its Poissonized version does, and the limits agree.
\end{lemma}

\begin{proof}
If $m=n=0$, the bound is immediate.  Otherwise, by symmetry assume
$m\geq n$, so that $m\geq1$.
Let $W^{(m)}$ be the sequence of normalized weights
\[
 W_j^{(m)}=\frac{Y_j}{S_m}\1_{\{j\leq m,\,S_m>0\}}.
\]
Then
\[
 \|W^{(m)}-W^{(n)}\|_1
 \leq
 2\left(1-\frac{S_n}{S_m}\right)\1_{\{S_m>0\}}.
\]
Exchangeability gives
\[
 \E\left[\frac{S_n}{S_m}\1_{\{S_m>0\}}\right]
 =\frac nm\Pp\{S_m>0\}.
\]
Thus
\[
 \E\|W^{(m)}-W^{(n)}\|_1\leq2\frac{m-n}{m}.
\]
Conditioning on the weights and using the triangle inequality proves
\eqref{eq:breiman-depoisson-bound}; the case $n\geq m$ is symmetric.
Finally, for $t\geq1$,
\[
 \E|T_{N_t}-T_{\lfloor t\rfloor}|
 \leq
 \frac{2\E|X|}{\lfloor t\rfloor}
 \E|N_t-\lfloor t\rfloor|
 =O(t^{-1/2}),
\]
which proves \eqref{eq:breiman-depoisson-limit}.
\end{proof}

Let $N_t$ be Poisson with mean $t$, independently of all other variables,
and put
\[
 S_{N_t}=\sum_{j=1}^{N_t}Y_j,
 \qquad
 P_j(t)=\frac{Y_j}{S_{N_t}}\1_{\{j\leq N_t,\,S_{N_t}>0\}},
 \qquad
 R_t=\sum_{j\geq1}P_j(t)X_j=T_{N_t}.
\]
If $T_n\Rightarrow T$, then Lemma~\ref{lem:breiman-depoissonization}
gives
\begin{equation}
 R_t\Rightarrow T\qquad(t\to\infty).
 \label{eq:breiman-poissonized-limit}
\end{equation}
The number of positive observations is Poisson with mean $t\lambda$, so
\begin{equation}
 \Pp\{S_{N_t}=0\}=e^{-t\lambda}.
 \label{eq:breiman-zero-mass}
\end{equation}

Write
\begin{equation}
 \Phi(q)=\E(1-e^{-qY}),
 \qquad
 e(q)=\frac{q\Phi'(q)}{\Phi(q)},
 \qquad q>0.
 \label{eq:breiman-laplace-slope}
\end{equation}
Then $\Phi$ is a strictly increasing bijection from $(0,\infty)$ onto
$(0,\lambda)$,
\[
 \Phi'(q)=\E(Ye^{-qY})>0,
 \qquad
 \Phi''(q)=-\E(Y^2e^{-qY})<0,
\]
and
\begin{equation}
 0<e(q)<1.
 \label{eq:breiman-slope-range}
\end{equation}
The last inequality follows from $xe^{-x}<1-e^{-x}$ for $x>0$.

For $p>0$, set
\[
 m_t(p)=\E\sum_{j\geq1}P_j(t)^p.
\]
The positive observations form a Poisson random measure with finite L\'evy
measure
\[
 \nu_Y(B)=\Pp\{Y\in B\},\qquad B\subset(0,\infty),
\]
whose total mass is $\lambda$.  After zero observations are omitted,
$P(t)$ is the normalized atom sequence $Q(t)$ of
Section~\ref{sec:poisson-calculus} for $\nu=\nu_Y$, so their power sums
coincide.  Lemma~\ref{lem:poisson-campbell} gives,
for every $p,t>0$,
\begin{equation}
 m_t(p)=\frac{t}{\Gamma(p)}
 \int_0^\infty q^{p-1}\E(Y^pe^{-qY})e^{-t\Phi(q)}\dd q.
 \label{eq:breiman-ratio-moment}
\end{equation}
It also gives $m_t(1)=1-e^{-t\lambda}$.  The common quadratic identity,
Lemma~\ref{lem:poisson-quadratic}, specializes to
\begin{equation}
 m_t(2)=t\lambda e^{-t\lambda}
 +\int_0^{t\lambda}ve^{-v}
 \left[1-e\!\left(\Phi^{-1}(v/t)\right)\right]\dd v.
 \label{eq:breiman-exact-quadratic}
\end{equation}

\begin{proposition}[A strict bound for $q\Phi'(q)/\Phi(q)$ at zero]
\label{prop:breiman-slope-gap}
If the limit in \eqref{eq:breiman-poissonized-limit} is nondegenerate, then
\begin{equation}
 \limsup_{q\downarrow0}\frac{q\Phi'(q)}{\Phi(q)}<1.
 \label{eq:breiman-slope-gap}
\end{equation}
\end{proposition}

\begin{proof}
Suppose instead that the limsup equals one.  Choose $q_k\downarrow0$ such
that
\[
 \delta_k=1-e(q_k)\downarrow0,
 \qquad x_k=\log\Phi(q_k)\downarrow-\infty,
\]
and set
\[
 L_k=\frac12\log\frac1{\delta_k},
 \qquad
 t_k=\frac{e^{L_k}}{\Phi(q_k)}.
\]
Then $t_k\to\infty$.  In logarithmic coordinates, put
\[
 g(x)=e\!\left(\Phi^{-1}(e^x)\right),
 \qquad x<\log\lambda,
 \qquad u(x)=1-g(x).
\]
Lemma~\ref{lem:poisson-slope-dynamics} shows that, for $y\leq x_k$,
\[
 u(y)\leq e^{x_k-y}\delta_k.
\]
If $v\leq e^{L_k}$, then
$y=\log(v/t_k)=x_k-L_k+\log v\leq x_k$, and consequently
\begin{equation}
 1-e\!\left(\Phi^{-1}(v/t_k)\right)
 \leq\frac{\sqrt{\delta_k}}v.
 \label{eq:breiman-slope-window}
\end{equation}
Fix $0<\varepsilon<M<\infty$.  For large $k$, both
$M<t_k\lambda$ and $M<e^{L_k}$.  By
\eqref{eq:breiman-exact-quadratic}, \eqref{eq:breiman-slope-range}, and
\eqref{eq:breiman-slope-window},
\[
 m_{t_k}(2)
 \leq t_k\lambda e^{-t_k\lambda}
 +\int_0^\varepsilon v\dd v
 +\frac{\sqrt{\delta_k}}\varepsilon\int_\varepsilon^Mve^{-v}\dd v
 +\int_M^\infty ve^{-v}\dd v.
\]
Let first $k\to\infty$, then $\varepsilon\downarrow0$, and finally
$M\to\infty$.  It follows that $m_{t_k}(2)\to0$.  Since
\[
 \max_jP_j(t_k)^2\leq\sum_jP_j(t_k)^2,
\]
the largest weight tends to zero in probability.
Lemma~\ref{lem:countable-wlln} then gives $R_{t_k}\to0$ in probability,
contradicting
\eqref{eq:breiman-poissonized-limit}.
\end{proof}

\subsection{A fractional power-sum bound}

The strict bound in Proposition~\ref{prop:breiman-slope-gap} gives a bound
on $m_t(r)$ for some $r\in(0,1)$, as required by
Theorem~\ref{thm:countable-main}.

\begin{proposition}[Fractional power-sum bound]
\label{prop:breiman-subunit}
Suppose that
\begin{equation}
 \limsup_{q\downarrow0}\frac{q\Phi'(q)}{\Phi(q)}<1.
 \label{eq:breiman-subunit-slope-hypothesis}
\end{equation}
Then there is an $r\in(0,1)$ such that
\begin{equation}
 \sup_{t>0}m_t(r)<\infty.
 \label{eq:breiman-subunit-bound}
\end{equation}
\end{proposition}

\begin{proof}
Choose
\[
 \limsup_{q\downarrow0}e(q)<a<r<1
\]
and $q_0>0$ so that $e(q)\leq a$ for $0<q\leq q_0$.  Put
$v_0=\Phi(q_0)$.  Lemma~\ref{lem:poisson-fractional}, applied to
$\nu_Y$, gives
\begin{equation}
 m_t(r)=\frac{t}{\Gamma(r)\Gamma(1-r)}
 \int_0^\infty(e^h-1)^{-r}I_t(h)\dd h,
 \label{eq:breiman-subunit-representation}
\end{equation}
where
\[
 I_t(h)=\int_0^\lambda
 \exp\{-t\Phi(e^{-h}\Phi^{-1}(v))\}\dd v.
\]
For $0<v\leq v_0$, integrating $e(q)\leq a$ on a logarithmic scale gives
\[
 \log\frac{\Phi(\Phi^{-1}(v))}
 {\Phi(e^{-h}\Phi^{-1}(v))}
 \leq ah,
\]
and hence
\[
 \Phi(e^{-h}\Phi^{-1}(v))\geq e^{-ah}v.
\]
For $v_0<v<\lambda$, monotonicity gives the same lower bound with $v_0$
in place of $v$.  It follows that
\[
 I_t(h)
 \leq\int_0^{v_0}e^{-te^{-ah}v}\dd v
 +(\lambda-v_0)e^{-te^{-ah}v_0}
 \leq\left(1+\frac{\lambda}{ev_0}\right)\frac{e^{ah}}t.
\]
Finally,
\[
 \int_0^\infty(e^h-1)^{-r}e^{ah}\dd h
 =\mathrm B(r-a,1-r)<\infty.
\]
Combining the last two displays with
\eqref{eq:breiman-subunit-representation} proves
\eqref{eq:breiman-subunit-bound}.
\end{proof}

No smoothness or nonlattice hypothesis on the law of $Y$ has entered the
argument.  In particular, the law may be atomic, may have lacunary support,
and may have an atom at zero.

\subsection{A Poissonized ratio--Tauberian theorem}
\label{sec:breiman-tauberian}

We now show that a positive limit of an expected power sum determines the
regular-variation index of $\bar G$.  The corresponding deterministic-sample
criterion is stated for $1<p\leq2$ in \cite[Proposition~2]{KM16}, whose
proof also applies to every $p>1$; the quadratic case goes back to the
earlier ratio results in \cite{MO,FJT}.  We give the Poissonized form using
Proposition~\ref{prop:poisson-mellin-ratio}.  This is the large-time
counterpart of the small-time theorem in Section~\ref{sec:tauberian}.

\begin{proposition}[Poissonized ratio--Tauberian theorem]
\label{prop:breiman-tauberian}
Fix $p>1$.  The following are equivalent:
\begin{enumerate}
\item $m_t(p)\to c$ as $t\to\infty$ for some $c>0$;
\item $\bar G\in\RV_\infty(-\beta)$ for some
      $\beta\in[0,1)$.
\end{enumerate}
When these conditions hold, $\beta$ is unique and
\begin{equation}
 c=\frac{\Gamma(p-\beta)}{\Gamma(p)\Gamma(1-\beta)}.
 \label{eq:breiman-power-limit}
\end{equation}
\end{proposition}

\begin{proof}
Define
\begin{equation}
 H_p(q)=\int_0^q s^{p-1}\E(Y^pe^{-sY})\dd s.
 \label{eq:breiman-Hp}
\end{equation}
Push the measure $\dd H_p$ forward under $q\mapsto\Phi(q)$ and denote its
cumulative function by $U_p$.  Thus
\[
 U_p(v)=H_p(\Phi^{-1}(v)),\qquad0<v<\lambda.
\]
The measure is finite, since $H_p(\infty)=\lambda\Gamma(p)$, and has no
atom at zero.  Formula \eqref{eq:breiman-ratio-moment} becomes
\begin{equation}
 \Gamma(p)m_t(p)=t\int_{[0,\lambda)}e^{-tv}\dd U_p(v).
 \label{eq:breiman-pushforward}
\end{equation}
Karamata's Laplace--Stieltjes theorem at the origin therefore gives
\begin{equation}
 m_t(p)\longrightarrow c>0
 \quad\Longleftrightarrow\quad
 \frac{H_p(q)}{\Phi(q)}\longrightarrow c\Gamma(p)
 \qquad(q\downarrow0).
 \label{eq:breiman-pushforward-equivalence}
\end{equation}

Put
\[
 F_p(q)=\int_0^\infty\bar G(y)y^{p-1}e^{-qy}\dd y,
 \qquad
 F_1(q)=\int_0^\infty\bar G(y)e^{-qy}\dd y.
\]
Tail integration gives the exact identities
\begin{equation}
 H_p(q)=q^pF_p(q),
 \qquad
 \Phi(q)=qF_1(q).
 \label{eq:breiman-tail-identities}
\end{equation}
Indeed, $\Phi(q)=qF_1(q)$ follows directly from Tonelli's theorem.  Moreover,
integration by parts against the tail gives
\[
 \frac{\dd}{\dd q}\{q^pF_p(q)\}
 =q^{p-1}\E(Y^pe^{-qY}),
\]
while $q^pF_p(q)\to0$ as $q\downarrow0$ by dominated convergence after
the substitution $v=qy$.  Integration in $q$ proves the first identity.
Consequently, \eqref{eq:breiman-pushforward-equivalence} is equivalent to
\begin{equation}
 q^{p-1}\frac{F_p(q)}{F_1(q)}\longrightarrow c\Gamma(p).
 \label{eq:breiman-tail-ratio}
\end{equation}
The Weyl fractional-integral identity
\begin{equation}
 F_1(q)=\frac1{\Gamma(p-1)}
 \int_0^\infty v^{p-2}F_p(q+v)\dd v
 \label{eq:breiman-weyl}
\end{equation}
follows from Tonelli's theorem.  Substituting $v=q(u-1)$ in
\eqref{eq:breiman-weyl} and using
\eqref{eq:breiman-tail-ratio} yields
\begin{equation}
 \int_1^\infty(u-1)^{p-2}\frac{F_p(qu)}{F_p(q)}\dd u
 \longrightarrow\frac1{c(p-1)}.
 \label{eq:breiman-mercer-ratio}
\end{equation}

Set
\[
 h(x)=x^{-p}F_p(1/x),
 \qquad
 k_p(u)=(u-1)^{p-2}u^{-p+1}\1_{\{u>1\}}.
\]
Then \eqref{eq:breiman-mercer-ratio} reads
\[
 \frac{(h*_Mk_p)(x)}{h(x)}\longrightarrow\frac1{c(p-1)},
 \qquad x\to\infty,
\]
with the Mellin-convolution convention of
Section~\ref{sec:poisson-calculus}.
The Mellin transform of the kernel is
\begin{equation}
 \widehat k_p(z)
 =\int_0^\infty u^{-z}k_p(u)\,\frac{\dd u}{u}
 =\mathrm B(p-1,1+z),
 \qquad\Re z>-1.
 \label{eq:breiman-kernel-transform}
\end{equation}
Its maximal Mellin strip is
\[
 -1<\Re z<\infty.
\]
The only finite endpoint is $-1$, and
$\widehat k_p(z)\to\infty$ as $z\downarrow-1$ along the real axis.
Moreover, $h$ is positive and nonincreasing, and, for $\theta\geq1$,
\begin{equation}
 \theta^{-p}\leq\frac{h(\theta x)}{h(x)}\leq1.
 \label{eq:breiman-bounded-decrease}
\end{equation}
In particular, for $1\leq\theta\leq2$,
\[
 h(\theta x)\geq2^{-p}h(x),
\]
which is the one-sided condition in
Proposition~\ref{prop:poisson-mellin-ratio}, with \(\zeta=2^{-p}\).
Indeed,
$q^pF_p(q)=\int_0^\infty\bar G(v/q)v^{p-1}e^{-v}\dd v$ is
nondecreasing in $q$.  Monotone convergence also gives
$h(x)\to\lambda\Gamma(p)$ as $x\downarrow0$, so $h$ has a positive,
locally bounded extension to the origin; the convolution is finite by
\eqref{eq:breiman-weyl}.  The kernel is nonnegative and nonzero, its
maximal Mellin interval is \((-1,\infty)\), and its only finite boundary
is the divergent boundary at \(-1\).  Hence
Proposition~\ref{prop:poisson-mellin-ratio} gives
\begin{equation}
 h\in\RV_\infty(\rho),
 \qquad
 \mathrm B(p-1,1+\rho)=\frac1{c(p-1)}
 \label{eq:breiman-drasin-shea}
\end{equation}
for some $\rho>-1$.  Since $m_t(p)\leq1$, one has $0<c\leq1$.  The beta
function in \eqref{eq:breiman-drasin-shea} is strictly decreasing in
$\rho$, so $-1<\rho\leq0$.  Put $\beta=-\rho$.  Then
$\beta\in[0,1)$, and
\eqref{eq:breiman-drasin-shea} is equivalent to
\eqref{eq:breiman-power-limit}.

It remains to recover the tail.  From
$h\in\RV_\infty(\rho)$,
\[
 F_p\in\RV_0(-p-\rho).
\]
Let
\[
 V(x)=\int_0^x y^{p-1}\bar G(y)\dd y.
\]
Karamata's Laplace--Stieltjes theorem gives
\[
 V(x)\sim\frac{F_p(1/x)}{\Gamma(p+\rho+1)},
 \qquad
 V\in\RV_\infty(p+\rho).
\]
Now
\[
 W(z)=pV(z^{1/p})=\int_0^z\bar G(w^{1/p})\dd w
\]
is regularly varying with positive index $(p+\rho)/p$.  Its density is
nonincreasing.  The monotone density theorem therefore yields
\[
 \bar G(w^{1/p})
 \sim\frac{p+\rho}{p}\frac{W(w)}w
 \in\RV_\infty(\rho/p).
\]
Replacing $w$ by $x^p$ proves
$\bar G\in\RV_\infty(\rho)
=\RV_\infty(-\beta)$.

Conversely, if
$\bar G(x)=x^{-\beta}L(x)$ with $0\leq\beta<1$, Karamata's theorem
gives
\[
 F_1(q)\sim\Gamma(1-\beta)q^{\beta-1}L(1/q),
 \qquad
 F_p(q)\sim\Gamma(p-\beta)q^{\beta-p}L(1/q).
\]
Equations \eqref{eq:breiman-tail-identities} and
\eqref{eq:breiman-pushforward-equivalence} then give
$m_t(p)\to c$, with $c$ as in
\eqref{eq:breiman-power-limit}.
\end{proof}

\begin{proof}[Proof of Theorem~\ref{thm:breiman}]
Assume first that $T_n$ has a nondegenerate weak limit.  Center the marks as
in \eqref{eq:breiman-centering}.  By
\eqref{eq:breiman-poissonized-limit},
Proposition~\ref{prop:breiman-slope-gap},
Proposition~\ref{prop:breiman-subunit}, and
Theorem~\ref{thm:countable-main}, there are $p\in(1,2]$ and $c>0$ such that
\[
 m_t(p)\longrightarrow c.
\]
Proposition~\ref{prop:breiman-tauberian} gives a unique
$\beta\in[0,1)$ for which
$\bar G\in\RV_\infty(-\beta)$.

Conversely, suppose that $X$ is nonconstant and
\eqref{eq:breiman-tail-rv} holds.  Define the finite L\'evy measure on
$(0,\infty)$ by
\[
 \nu_Y(B)=\Pp\{Y\in B\},
 \qquad B\subset(0,\infty).
\]
Its tail is $\bar G$.  The normalized jumps of the corresponding
compound-Poisson subordinator at time $t$ are exactly the weights in
$T_{N_t}$.  Arrange the positive observations in decreasing order of their
$Y$-values,
breaking ties with auxiliary iid continuous variables independent of the
marks.  The resulting permutation is measurable with respect to the
$\sigma$-field generated by $N_t$, the $Y$-values and the auxiliary tie-breakers,
which is independent of the marks.  Hence the permuted marks remain iid with
law $\Law(X)$ and independent of the ranked weights.
Proposition~\ref{prop:measure:ranked-jumps}, applied at
infinity, gives their $\ell^1$ limit: $\PD(\beta,0)$ for $0<\beta<1$ and
$e_1$ for $\beta=0$.  Proposition~\ref{prop:measure:marking-dust}, with
base law $\Law(X)$ and the identity function, then yields
$T_{N_t}\Rightarrow W_\beta$ with $W_\beta$ as in
\eqref{eq:breiman-limit-law}.  On the zero-jump event the convention in that
proposition returns $\E X$, whereas $T_{N_t}=0$; this event has
probability $e^{-t\Pp\{Y>0\}}$ and is asymptotically irrelevant.
Lemma~\ref{lem:measure:marked-nondegenerate} shows that $W_\beta$ is
nondegenerate.  Finally,
Lemma~\ref{lem:breiman-depoissonization} transfers this convergence to
$T_n$.
\end{proof}

\section{Complete classification for iid self-normalized weights}
\label{sec:iid-classification}

Let $G$ be a probability law on $[0,\infty)$ such that
$\Pp\{Y>0\}>0$ for $Y\sim G$, and let $Y_1,Y_2,\ldots$ be iid with
law $G$.  Put
\[
 S_n=\sum_{i=1}^nY_i,
 \qquad
 W_{n,i}=\begin{cases}
 Y_i/S_n,&S_n>0,\\[1mm]
 1/n,&S_n=0,
 \end{cases}
 \qquad 1\leq i\leq n,
\]
and append zero coordinates.  This convention makes the weights a mass
partition for every outcome.  If $\rho=\Pp\{Y=0\}$, it differs from any
other convention only on an event of probability $\rho^n$.  Write
$W_n^\downarrow$ for the decreasing rearrangement and
\begin{equation}
 \begin{aligned}
 \bar G(x)&=\Pp\{Y>x\},
 &A_G(x)&=\E\bigl[Y\1_{\{Y\leq x\}}\bigr],\\
 I_G(x)&=\E(Y\wedge x)=A_G(x)+x\bar G(x).
 \end{aligned}
 \label{eq:iid-truncated-means}
\end{equation}
For $F\in\cP_1(\R)$, take iid $F$-distributed variables
$X_1,X_2,\ldots$, independent of the weights, and write
\begin{equation}
 \mathcal B_n^G(F)
 =\Law\left(\sum_{i=1}^nW_{n,i}X_i\right).
 \label{eq:iid-output-operator}
\end{equation}

The constant branch of the classification is most conveniently expressed
through relative stability.

\begin{lemma}[Relative stability and the truncated mean]
\label{lem:iid-relative-stability}
The following conditions are equivalent:
\begin{align}
 &\frac{Y_{n:n}}{S_n}\longrightarrow0
       \quad\text{in probability},
 \label{eq:iid-dust-max}\\
 &\text{there are }a_n>0\text{ such that }
       \frac{S_n}{a_n}\longrightarrow1
       \quad\text{in probability},
 \label{eq:iid-relative-stability}\\
 &\frac{x\bar G(x)}{A_G(x)}\longrightarrow0,
 \label{eq:iid-dust-ratio-A}\\
 &\frac{x\bar G(x)}{I_G(x)}\longrightarrow0,
 \label{eq:iid-dust-ratio-I}\\
 &A_G\in\RV_\infty(0),
 \label{eq:iid-AG-slow}\\
 &I_G\in\RV_\infty(0).
 \label{eq:iid-IG-slow}
\end{align}
Here $Y_{n:n}=\max_{i\leq n}Y_i$; on the all-zero event the ratio in
\eqref{eq:iid-dust-max} is set equal to zero.
\end{lemma}

\begin{proof}
Breiman~\cite[Proposition~1]{Breiman} proves the equivalence of
\eqref{eq:iid-dust-max} and \eqref{eq:iid-relative-stability}, without a
continuity assumption on the law of $Y$.  For nonnegative $Y$, the truncated
first moment $v$ in the relative-stability criterion is $A_G$; the criterion
therefore gives \eqref{eq:iid-dust-ratio-A}; see
\cite[pp.~499 and 502]{Maller1979}.  If $Y$ is bounded, this equivalence
also follows directly from the weak law of large numbers.  An atom at zero
causes no change, since $\Pp\{S_n=0\}=\Pp\{Y=0\}^n\to0$.
We verify the remaining equivalences, including both conventions for the
truncated mean.

The identity in \eqref{eq:iid-truncated-means} shows directly that
\eqref{eq:iid-dust-ratio-A} and \eqref{eq:iid-dust-ratio-I} are
equivalent.  If \eqref{eq:iid-dust-ratio-I} holds, then, for $\lambda>1$,
monotonicity of $\bar G$ gives
\[
 0\leq I_G(\lambda x)-I_G(x)
 \leq(\lambda-1)x\bar G(x)=o(I_G(x)).
\]
The case $0<\lambda<1$ follows by applying the same estimate at $\lambda x$.
Thus $I_G$ is slowly varying.  Conversely, for $0<\lambda<1$,
\[
 I_G(x)-I_G(\lambda x)
 =\int_{\lambda x}^x\bar G(y)\,\dd y
 \geq(1-\lambda)x\bar G(x).
\]
Slow variation of $I_G$ therefore implies
\eqref{eq:iid-dust-ratio-I}.  Under either ratio condition,
$A_G(x)\sim I_G(x)$, so \eqref{eq:iid-AG-slow} follows from
\eqref{eq:iid-IG-slow}.

It remains to start from \eqref{eq:iid-AG-slow}.  Fix $\lambda>1$.
Partitioning $(x,\infty)$ into the intervals
$(\lambda^kx,\lambda^{k+1}x]$ yields
\begin{equation}
 x\bar G(x)
 \leq\sum_{k\geq0}\lambda^{-k}
 \{A_G(\lambda^{k+1}x)-A_G(\lambda^kx)\}.
 \label{eq:iid-potter-sum}
\end{equation}
After division by $A_G(x)$, every fixed finite initial part tends to zero.
For $0<\varepsilon<1$, Potter's bound gives, for all sufficiently large
$x$ and every $k\geq0$,
\[
 \frac{A_G(\lambda^{k+1}x)}{A_G(x)}
 \leq C\lambda^{\varepsilon(k+1)}.
\]
The tail of \eqref{eq:iid-potter-sum}, starting at $k=K$, is consequently
bounded by a constant times
$\sum_{k\geq K}\lambda^{-(1-\varepsilon)k}$, uniformly for large $x$.
Letting first $x\to\infty$ and then $K\to\infty$ proves
$x\bar G(x)=o(A_G(x))$.
\end{proof}

We next use the uniform convention on $\{S_n=0\}$ to transfer the whole
mass partition.  These weights have total mass one on every outcome and
give the exact identity below; the zero convention in
Lemma~\ref{lem:breiman-depoissonization} gives an upper bound instead.
For $n=0$, let $W^{(0)}$ be any fixed mass-one
sequence.  For $n\geq1$, write $W^{(n)}=(W_{n,i})_{i\geq1}$ in the original
sample order.

\begin{lemma}[Sample-size coupling]
\label{lem:iid-partition-depoissonization}
Couple all sample sizes through the same sequence $(Y_i)$.  For
$m,n\geq1$,
\begin{equation}
 \E\lVert W^{(m)}-W^{(n)}\rVert_1
 =2\frac{|m-n|}{m\vee n}.
 \label{eq:iid-exact-size-coupling}
\end{equation}
The same left-hand side with decreasing rearrangements is no larger.  If
$N_n$ is Poisson with mean $n$, independent of the sequence, then
\begin{equation}
 \E\lVert W^{(N_n)}-W^{(n)}\rVert_1
 \leq\frac{2\E|N_n-n|}{n}\leq\frac2{\sqrt n}.
 \label{eq:iid-partition-poisson-bound}
\end{equation}

Let $E$ be Polish, let $H\in\cP(E)$, and attach one iid $H$-sequence
$(Z_i)$, independent of the weights.  If
$M^{(n)}=\sum_iW_i^{(n)}\delta_{Z_i}$, then
\begin{align}
 d_{\mathrm{BL}}(M^{(m)},M^{(n)})
 &\leq\lVert W^{(m)}-W^{(n)}\rVert_1,
 \label{eq:iid-measure-size-coupling}\\
 \E\left|M^{(m)}h-M^{(n)}h\right|
 &\leq H|h|\,
       \E\lVert W^{(m)}-W^{(n)}\rVert_1,
 \qquad h\in L^1(H).
 \label{eq:iid-test-size-coupling}
\end{align}
Here $d_{\mathrm{BL}}$ is any bounded-Lipschitz metric formed from test
functions bounded by one.
\end{lemma}

\begin{proof}
By symmetry, suppose $m\geq n\geq1$.  On $\{S_m>0\}$, including the case
$S_n=0$ under the uniform convention,
\[
 \lVert W^{(m)}-W^{(n)}\rVert_1
 =2\left(1-\frac{S_n}{S_m}\right).
\]
Exchangeability gives
\[
 \E\left[\frac{S_n}{S_m};S_m>0\right]
 =\frac nm\Pp\{S_m>0\}.
\]
On $\{S_m=0\}$ the two uniform vectors have distance $2(1-n/m)$.
Adding the two contributions proves
\eqref{eq:iid-exact-size-coupling}.  Decreasing rearrangement is
nonexpansive in $\ell^1$ for finite vectors, by the optimal matching
inequality, and the assertion with appended zeros follows.

Conditioning on $N_n$ and using \eqref{eq:iid-exact-size-coupling} when
$N_n\geq1$, and the elementary bound $\lVert p-p'\rVert_1\leq2$ when
$N_n=0$, gives
\[
 \E\lVert W^{(N_n)}-W^{(n)}\rVert_1
 \leq\frac2n\E|N_n-n|.
\]
The last inequality in \eqref{eq:iid-partition-poisson-bound} follows from
Cauchy--Schwarz and $\Var(N_n)=n$.

For a bounded-Lipschitz function $f$ with $\lVert f\rVert_\infty\leq1$,
\[
 |M^{(m)}f-M^{(n)}f|
 \leq\sum_i|W_i^{(m)}-W_i^{(n)}|,
\]
which proves \eqref{eq:iid-measure-size-coupling}.  For an integrable $h$,
the same inequality with $|h(Z_i)|$ on the right, followed by conditioning
on the weights, proves \eqref{eq:iid-test-size-coupling}.
\end{proof}

\begin{theorem}[Classification from a single fixed mark]
\label{thm:iid-complete-classification}
Fix a non-Dirac law $F_*\in\cP_1(\R)$.  If
$\mathcal B_n^G(F_*)$ converges weakly along the full sequence, then exactly
one of the following alternatives holds.
\begin{enumerate}[label=\textup{(\roman*)}]
\item \emph{One big jump.}  The tail $\bar G$ is slowly varying at
infinity.  In this case
\[
 W_n^\downarrow\longrightarrow e_1=(1,0,\ldots)
 \quad\text{in probability in }\ell^1.
\]
\item \emph{Stable partition.}  For a unique $\beta\in(0,1)$,
$\bar G\in\RV_\infty(-\beta)$.  In this case
\[
 W_n^\downarrow\Longrightarrow Q^{(\beta)}
 \quad\text{in }\ell^1,
 \qquad Q^{(\beta)}\sim\PD(\beta,0).
\]
\item \emph{Dust.}  The equivalent conditions of
Lemma~\ref{lem:iid-relative-stability} hold.  In this case
\[
 W_n^\downarrow\longrightarrow0
 \quad\text{in probability in the product topology of }
 \K.
\]
There is no convergence to zero in $\ell^1$.
\end{enumerate}
In cases \textup{(i)} and \textup{(ii)} the limiting law of the marked sum is
nondegenerate.  In case \textup{(iii)} it is the point mass at
$\int xF_*(\dd x)$.
\end{theorem}

\begin{proof}
Let $L$ be the weak limit.  Replacing the uniform weights on $\{S_n=0\}$
by zero weights changes the law of the weighted mean by total variation at
most $\rho^n$.
Thus Theorem~\ref{thm:breiman} applies without changing $L$.

Suppose first that $L$ is nondegenerate.  Theorem~\ref{thm:breiman} gives
a unique $\beta\in[0,1)$ such that
$\bar G\in\RV_\infty(-\beta)$.  If $\beta=0$, poissonize the sample and
regard its positive observations as the jumps of the compound-Poisson
process with L\'evy measure
\[
 \nu_G(B)=\Pp\{Y\in B\},\qquad B\subset(0,\infty).
\]
Proposition~\ref{prop:measure:ranked-jumps}, at infinity, gives convergence
of the poissonized ranked weights to $e_1$ in $\ell^1$.  Its zero-jump
convention is immaterial because the empty probability is
$\exp\{-n\Pp(Y>0)\}$.  Lemma~\ref{lem:iid-partition-depoissonization}
then gives the deterministic-sample assertion.  The same argument for
$0<\beta<1$ gives the $\PD(\beta,0)$ limit.  Nondegeneracy of the resulting
weighted means follows from Lemma~\ref{lem:measure:marked-nondegenerate}.

Suppose instead that $L$ is a point mass.  The weights have total mass one,
so Theorem~\ref{thm:countable-dust}, applied to the nonconstant integrable
mark, yields
\[
 \max_{i\leq n}W_{n,i}\longrightarrow0
 \quad\text{in probability}
\]
and identifies the limiting constant as the mean of $F_*$.  The displayed
condition is \eqref{eq:iid-dust-max}, up to the exponentially negligible
all-zero event, and hence Lemma~\ref{lem:iid-relative-stability} gives all
the criteria in the dust alternative.  Since every fixed ranked coordinate
is bounded by the largest one, the product-topology convergence follows.
On the other hand, $\lVert W_n^\downarrow\rVert_1=1$ for every $n$, so
$\ell^1$ convergence to zero is impossible.

The alternatives are disjoint.  Indeed, the two regularly varying cases
have different indices and produce nondegenerate limits, whereas the dust
case produces a constant limit.  This also completes the proof that the
three cases exhaust every full-sequence limit.
\end{proof}

\begin{remark}[The endpoint label]
\label{rem:iid-beta-one-label}
We label the one-big-jump, stable, and dust limits by
$\beta=0$, $0<\beta<1$, and $\beta=1$, respectively.  The last value is a
label for the dust endpoint.  It does not assert that
$\bar G\in\RV_\infty(-1)$; finite-mean laws already belong to the dust
branch without satisfying such a tail condition.
\end{remark}

\section{Transfer to other marks and independently marked measures}
\label{sec:iid-operators}

The weight classification determines the limits for every integrable mark.
The consequences below follow from independent marking, uniform
integrability and contraction in $W_1$.  They give convergence in $W_1$,
uniformly over compact families of mark laws, together with the limits of
independently marked measures.

For $0<\beta<1$, let $Q^{(\beta)}\sim\PD(\beta,0)$ and take an iid
$F$-sequence $(X_j)$ independent of $Q^{(\beta)}$.  For
$F\in\cP_1(\R)$ define
\begin{equation}
 \mathcal B_\beta(F)
 =\Law\left(\sum_{j\geq1}Q_j^{(\beta)}X_j\right),
 \qquad
 \mathcal B_0(F)=F,
 \qquad
 \mathcal B_1(F)=\delta_{\int xF(\dd x)}.
 \label{eq:iid-limit-operators}
\end{equation}
The series in \eqref{eq:iid-limit-operators} is absolutely convergent almost
surely by Lemma~\ref{lem:countable-series-products}.

For $E$ Polish, $H\in\cP(E)$ and iid $H$-locations $Z=(Z_j)$, write
\begin{equation}
 M_{q,H}=\mathcal M_H(q,Z),\qquad q\in\K,
 \label{eq:iid-marked-limit-measure}
\end{equation}
using the marking map \eqref{eq:measure:marking-map}.
Thus $M_{e_1,H}=\delta_{Z_1}$ and $M_{0,H}=H$.

\begin{theorem}[Transfer from a single fixed mark to all integrable marks]
\label{thm:iid-transfer}
Assume the hypothesis of Theorem~\ref{thm:iid-complete-classification}, and
let $\beta\in[0,1]$ be the label of the resulting branch.  Then, for every
fixed $F\in\cP_1(\R)$,
\begin{equation}
 W_1\bigl(\mathcal B_n^G(F),\mathcal B_\beta(F)\bigr)
 \longrightarrow0.
 \label{eq:iid-pointwise-W1}
\end{equation}
For every compact set
$K\subset(\cP_1(\R),W_1)$, this convergence is uniform:
\begin{equation}
 \sup_{F\in K}
 W_1\bigl(\mathcal B_n^G(F),\mathcal B_\beta(F)\bigr)
 \longrightarrow0.
 \label{eq:iid-compact-uniform-W1}
\end{equation}

More generally, let $H\in\cP(E)$, take fresh iid $H$-locations $(Z_i)$
independent of the weights, and set
\[
 M_n=\sum_{i=1}^nW_{n,i}\delta_{Z_i}.
\]
For any fixed finite family $h_1,\ldots,h_m\in L^1(H)$, one has, in the
one-big-jump and stable branches,
\begin{equation}
 \left(W_n^\downarrow,M_n,(M_nh_k)_{k\leq m}\right)
 \Longrightarrow
 \left(Q^{(\beta)},M_{Q^{(\beta)},H},
       (M_{Q^{(\beta)},H}h_k)_{k\leq m}\right)
 \label{eq:iid-marking-regular}
\end{equation}
in $\ell^1\times\cP(E)\times\R^m$, where $Q^{(0)}=e_1$ and
$\cP(E)$ has its weak topology.  In the dust branch,
\begin{equation}
 \begin{gathered}
 W_n^\downarrow\longrightarrow0
 \quad\text{in probability in the product topology},\\
 M_n\longrightarrow H
 \quad\text{weakly in probability},
 \end{gathered}
 \label{eq:iid-marking-dust}
\end{equation}
and
\begin{equation}
 (M_nh_1,\ldots,M_nh_m)\longrightarrow
 (Hh_1,\ldots,Hh_m)
 \quad\text{in }L^1(\Omega;\R^m).
 \label{eq:iid-marking-dust-L1}
\end{equation}
Here and below, \(L^1(\Omega;\R^m)\) is formed using the
\(\ell^1\)-norm on \(\R^m\).
These statements are joint in the indicated product spaces.
\end{theorem}

\begin{proof}
Arrange the $Y_i$'s in decreasing order, breaking ties with auxiliary iid
continuous variables independent of the locations.  The correspondingly permuted
locations remain iid with law $H$ and are independent of
$W_n^\downarrow$.  Hence the joint law of $(W_n^\downarrow,M_n)$ is the
joint law of
\[
 \left(W_n^\downarrow,
       M_{W_n^\downarrow,H}\right)
\]
in the construction \eqref{eq:iid-marked-limit-measure}.  The partition
limits in Theorem~\ref{thm:iid-complete-classification}, followed by
Proposition~\ref{prop:measure:marking-dust}, prove
\eqref{eq:iid-marking-regular}, \eqref{eq:iid-marking-dust} and the
$L^1$ conclusion \eqref{eq:iid-marking-dust-L1}.  This use of independent marking is essential:
the deterministic map in \eqref{eq:iid-marked-limit-measure} need not be
continuous in the product topology when the locations are held fixed.

Taking $E=\R$, $H=F$, and $h(x)=x$ in the preceding marking limits gives
weak convergence to the laws in \eqref{eq:iid-limit-operators}.
We now upgrade this convergence to $W_1$.  By the de la Vall\'ee--Poussin
criterion, there is an
increasing convex function $\Psi$ with
$\Psi(x)/x\to\infty$ and $\int\Psi(|x|)F(\dd x)<\infty$.  Conditional
Jensen gives
\begin{equation}
 \sup_n\E\Psi\left(\left|\sum_iW_{n,i}X_i\right|\right)
 \leq\E\Psi(|X|).
 \label{eq:iid-output-UI}
\end{equation}
The same estimate holds for the random variable defining
$\mathcal B_\beta(F)$, including both endpoints.  Thus these weighted means,
together with their limit, form a uniformly integrable family.  Weak
convergence and uniform integrability give \eqref{eq:iid-pointwise-W1}.

It remains to prove uniformity.  For any coupling $(X,X')$ of $F,F'$, use
iid copies of the pair, independently of the weights.  Then
\[
 \E\left|\sum_iW_{n,i}X_i-\sum_iW_{n,i}X_i'\right|
 \leq\E|X-X'|.
\]
Taking the infimum over couplings proves
\begin{equation}
 W_1\bigl(\mathcal B_n^G(F),\mathcal B_n^G(F')\bigr)
 \leq W_1(F,F').
 \label{eq:iid-output-contraction}
\end{equation}
The identical argument gives this contraction for $\mathcal B_\beta$;
when $\beta=1$, use
$|\int xF(\dd x)-\int xF'(\dd x)|\leq W_1(F,F')$.
If $F_1,\ldots,F_N$ is a finite $\varepsilon$-net for $K$, the two
contractions imply
\[
 \sup_{F\in K}W_1(\mathcal B_n^G(F),\mathcal B_\beta(F))
 \leq2\varepsilon+
 \max_{j\leq N}W_1(\mathcal B_n^G(F_j),\mathcal B_\beta(F_j)).
\]
Letting $n\to\infty$ and then $\varepsilon\downarrow0$ proves
\eqref{eq:iid-compact-uniform-W1}.
\end{proof}

The branch parameter is determined by the expectation of a strictly convex
function of the marked mean, even when the mark has no second moment.

\begin{lemma}[Expected power sums of a stable partition]
\label{lem:iid-PD-power-sums}
For $0<\beta<1$ and $p>\beta$,
\begin{equation}
 \E\sum_{j\geq1}(Q_j^{(\beta)})^p
 =\frac{\Gamma(p-\beta)}{\Gamma(p)\Gamma(1-\beta)}.
 \label{eq:iid-PD-power-sum}
\end{equation}
\end{lemma}

\begin{proof}
Formula \eqref{eq:iid-PD-power-sum} is the specialization
\(\theta=0\), \(f(u)=u^p\) of
\cite[equation~(6), p.~858]{PitmanYor}.
\end{proof}

For Dirichlet means, Letac and Piccioni
\cite[Theorem~1.2 and Corollary~5.5]{LetacPiccioni2018} establish
monotonicity in convex order and, for integrable non-Dirac base laws,
identification of the concentration parameter.  Their weights have law
$\PD(0,\theta)$, with $\theta>0$.  The following theorem concerns instead
the parameter $\beta$ in $\PD(\beta,0)$, together with its one-atom and
dust endpoints.

\Needspace{18\baselineskip}
\begin{theorem}[Identification of the Poisson--Dirichlet parameter]
\label{thm:iid-PD-identification}
For every non-Dirac $F\in\cP_1(\R)$, the map
\[
 [0,1]\ni\beta\longmapsto\mathcal B_\beta(F)
 \in(\cP_1(\R),W_1)
\]
is continuous and injective.  More precisely, the function
\begin{equation}
 h_F(\beta)=\int_{\R}\sqrt{1+x^2}\,\mathcal B_\beta(F)(\dd x)
 \label{eq:iid-PD-convex-moment}
\end{equation}
is continuous and strictly decreasing on $[0,1]$.

If $0<\Var_F(X)<\infty$, the parameter can be read explicitly from
\begin{equation}
 \Var(\mathcal B_\beta(F))
 =(1-\beta)\Var_F(X),
 \qquad 0\leq\beta\leq1.
 \label{eq:iid-PD-output-variance}
\end{equation}
\end{theorem}

\begin{proof}
We first prove continuity using the Griffiths--Engen--McCloskey (GEM)
representation of $\PD(\beta,0)$~\cite{PitmanYor}.  Its size-biased
sequence is
\[
 \widetilde Q_j^{(\beta)}
 =V_j^{(\beta)}\prod_{i<j}(1-V_i^{(\beta)}),
 \qquad
 V_j^{(\beta)}\sim\operatorname{Beta}(1-\beta,j\beta),
\]
  with independent factors.  Using a common sequence of uniform random
variables, couple the beta factors by their quantile functions.  Then every
fixed head converges almost surely when $\beta_n\to\beta\in(0,1)$.  The
residual mass after $K$ terms
satisfies
\begin{equation}
 \E R_K^{(\beta)}
 =\prod_{j=1}^K\frac{j\beta}{1+(j-1)\beta}.
 \label{eq:iid-GEM-residual}
\end{equation}
This product tends to zero uniformly for $\beta$ in a compact subinterval
of $(0,1)$: every factor is increasing in $\beta$, so it is bounded above
by the product at the right endpoint, which tends to zero.  Head convergence
and \eqref{eq:iid-GEM-residual} therefore give $\ell^1$ convergence in
probability of the size-biased sequences.  Decreasing rearrangement is
nonexpansive in $\ell^1$, and hence the ranked partitions are locally
continuous in distribution in $\ell^1$.

At $\beta=0$, Lemma~\ref{lem:iid-PD-power-sums} with $p=2$ gives
\[
 \E\sum_j(Q_j^{(\beta)})^2=1-\beta.
\]
Since $\sum_jq_j^2\leq q_1$ for every mass-one ranked sequence,
\begin{equation}
 \E\lVert Q^{(\beta)}-e_1\rVert_1
 =2\E(1-Q_1^{(\beta)})\leq2\beta.
 \label{eq:iid-PD-zero-endpoint}
\end{equation}
As $\beta\uparrow1$,
$\E[(Q_1^{(\beta)})^2]\leq1-\beta$, so
$Q^{(\beta)}\to0$ in probability in the product topology as
$\beta\uparrow1$.  The independent-marking result,
Proposition~\ref{prop:measure:marking-dust}, now gives weak continuity of
$\mathcal B_\beta(F)$ at $\beta=1$.  Local $\ell^1$ coupling gives weak
continuity in the interior and at $\beta=0$.  In all cases, conditional Jensen as in
\eqref{eq:iid-output-UI} gives uniform integrability along the parameter
sequence.  The weak continuity therefore upgrades to $W_1$ continuity.

For strict monotonicity, we use the Poisson--Dirichlet fragmentation
identity \cite[Corollary~13]{PitmanCoalescents1999} and the convex-order
coupling of \cite[Proposition~16]{PitmanWeighted}.  Fix $0<a<b<1$.
Here $\PD(\sigma,\theta)$ denotes the two-parameter Poisson--Dirichlet
law, with $0<\sigma<1$ and $\theta>-\sigma$
\cite[Definition~1]{PitmanYor}; hence $(b,-a)$ is admissible.
Let $P\sim\PD(a,0)$ and, independently, let
$W_i=(W_{ij})_{j\geq1}$ be iid $\PD(b,-a)$ partitions.  The ranked
products $(P_iW_{ij})_{i,j\geq1}$ have law $\PD(b,0)$.
Take iid marks $X_{ij}$ with law $F$, independently of the partitions.
Conditionally on $(P,W)$, choose indices $J_i$ independently of one
another and of the marks, with
$\Pp\{J_i=j\mid P,W\}=W_{ij}$.  Put
\[
 Y=\sum_iP_iX_{iJ_i},\qquad
 Z=\sum_{i,j}P_iW_{ij}X_{ij},\qquad
 \mathcal G=\sigma(P,W,(X_{ij})_{i,j\geq1}).
\]
Tonelli's theorem gives
\[
 \E\sum_iP_i|X_{iJ_i}|
 =\E\sum_{i,j}P_iW_{ij}|X_{ij}|
 =\int_{\R}|x|F(\dd x).
\]
Both series therefore converge absolutely almost surely and in $L^1$.
Conditionally on $(P,W)$, the selected marks $X_{iJ_i}$ are iid with
law $F$, so they are also independent of the partitions.  It follows that
$Y$ has law $\mathcal B_a(F)$, $Z$ has law $\mathcal B_b(F)$, and
$\E[Y\mid\mathcal G]=Z$.

Let $\varphi(x)=\sqrt{1+x^2}$, which is strictly convex and satisfies
$\varphi(x)\leq1+|x|$.  Set
$\mathcal H=\mathcal G\vee\sigma(J_i:i\geq2)$.
Conditionally on $\mathcal H$, only $J_1$ varies in $Y$.
Since $P_1,W_{11},W_{12}>0$ almost surely, this conditional law is
nonconstant on the event $\{X_{11}\ne X_{12}\}$, which has positive
probability.  Strict conditional Jensen, followed by conditional Jensen
with respect to $\mathcal G$, gives
\[
 \E\varphi(Y)
 >\E\varphi\bigl(\E[Y\mid\mathcal H]\bigr)
 \geq\E\varphi\bigl(\E[Y\mid\mathcal G]\bigr)
 =\E\varphi(Z).
\]
Thus $h_F(a)>h_F(b)$.

When $a=0<b<1$, use $P=e_1$ and $W_1\sim\PD(b,0)$ directly in the
same selector argument.  When $b=1$, the law $\mathcal B_a(F)$ is
nondegenerate for $a<1$ by Lemma~\ref{lem:measure:marked-nondegenerate},
and strict Jensen against its mean gives $h_F(a)>h_F(1)$.
Finally, $\varphi$ is $1$-Lipschitz, so the established $W_1$ continuity
implies continuity of $h_F$.  Its strict monotonicity proves injectivity
of $\beta\mapsto\mathcal B_\beta(F)$.

If $F$ has finite variance, conditioning on the weights and using
\eqref{eq:iid-PD-power-sum} at $p=2$ proves
\eqref{eq:iid-PD-output-variance} for $0<\beta<1$.  The two endpoint
identities follow directly from \eqref{eq:iid-limit-operators}.
\end{proof}

\begin{remark}[Why compactness is needed in \eqref{eq:iid-compact-uniform-W1}]
\label{rem:iid-no-global-W1}
Global uniformity over $\cP_1(\R)$ is false.  For equal weights and
$F_n=(1-1/n)\delta_0+n^{-1}\delta_n$, the weighted mean has law
$\operatorname{Bin}(n,1/n)$, whereas the dust limit is $\delta_1$; indeed
\[
 W_1\bigl(\operatorname{Bin}(n,1/n),\delta_1\bigr)
 \longrightarrow\E|N-1|=2e^{-1},
 \qquad N\sim\operatorname{Pois}(1).
\]
\end{remark}

\section{Why integrability of the mark is needed}
\label{sec:iid-cauchy-sharpness}

The conclusion of Theorem~\ref{thm:iid-complete-classification} cannot be
extended to arbitrary nonintegrable marks.  The following example
keeps the law of the weighted mean fixed while the underlying mass partition
has different subsequential limits.

For any mass partition of total mass one, independent of the iid Cauchy marks,
invariance of the Cauchy mean is classical; see
\cite[Proposition~15]{PitmanWeighted}.  The point here is that this
invariance can coexist with weights obtained by normalizing iid observations,
even when their ranked partition has no full-sequence limit.

\begin{proposition}[A Cauchy mark that does not detect the weight regime]
\label{prop:iid-cauchy-sharpness}
There is a probability law $G$ on $(0,\infty)$ for which
$W_n^\downarrow$ does not converge in distribution even in the product
topology, whereas, for standard symmetric Cauchy marks,
\[
 \sum_{i=1}^nW_{n,i}X_i\stackrel{d}{=}X_1
 \qquad\text{for every }n\geq1.
\]
\end{proposition}

\begin{proof}
For $k\geq1$, set
\[
 q_k=2^{-2^k},
 \qquad
 q_0=1-\sum_{k\geq1}q_k>0,
 \qquad a_0=1.
\]
Define recursively
\[
 A_{k-1}=\sum_{j<k}q_ja_j,
 \qquad
 a_k=2a_{k-1}+\frac{k^2A_{k-1}}{q_k},
 \qquad k\geq1,
\]
and let $G$ put mass $q_k$ at $a_k$, $k\geq0$.  The support is strictly
increasing and
\begin{equation}
 q_ka_k\geq k^2A_{k-1}.
 \label{eq:iid-lacunary-dominance}
\end{equation}

Fix $c>0$ and put $n_k(c)=\lfloor c/q_k\rfloor$.  In a sample of this
size, let $N_k$ be the number of observations equal to $a_k$, let $H_k$ be
the event that no observation exceeds $a_k$, and let $L_k$ be the sum of
the observations below $a_k$.  Then
\[
 N_k\Longrightarrow N_c\sim\operatorname{Pois}(c).
\]
Moreover, since $q_{k+1}=q_k^2$ and the remaining terms decrease still
faster,
\begin{equation}
 \Pp(H_k^c)
 \leq n_k(c)\sum_{j>k}q_j=O(q_k)\longrightarrow0.
 \label{eq:iid-lacunary-no-higher-level}
\end{equation}
By \eqref{eq:iid-lacunary-dominance},
\begin{equation}
 \E\frac{L_k}{a_k}
 \leq\frac{n_k(c)A_{k-1}}{a_k}
 \leq\frac{c}{k^2}\longrightarrow0.
 \label{eq:iid-lacunary-lower-sum}
\end{equation}
Set
\[
 Z_k=\frac{Y_{n_k(c):n_k(c)}}{S_{n_k(c)}},
 \qquad
 f(0)=0,
 \qquad
 f(j)=\frac1j\quad(j\geq1).
\]
On $H_k\cap\{N_k\geq1\}$,
\begin{equation}
 Z_k=\frac1{N_k+L_k/a_k}.
 \label{eq:iid-lacunary-max-ratio-positive}
\end{equation}
Consequently, on this event,
\[
 |Z_k-f(N_k)|
 =\frac{L_k/a_k}{N_k(N_k+L_k/a_k)}
 \leq\frac{L_k}{a_k}.
\]

On $H_k\cap\{N_k=0\}$ all observations are at most $a_{k-1}$.  The number
$C_{k-1}$ of observations equal to $a_{k-1}$ satisfies
\[
 \E C_{k-1}=n_k(c)q_{k-1}
 \sim c\frac{q_{k-1}}{q_k}\longrightarrow\infty,
 \qquad
 \frac{\Var(C_{k-1})}{(\E C_{k-1})^2}
 =\frac{1-q_{k-1}}{n_k(c)q_{k-1}}\longrightarrow0,
\]
so $C_{k-1}\to\infty$ in probability.  On
$H_k\cap\{N_k=0\}\cap\{C_{k-1}\geq1\}$,
\[
 0\leq Z_k\leq\frac1{C_{k-1}}.
\]

Let $\varepsilon>0$ and choose an integer $M>1/\varepsilon$.  Then
\begin{align*}
 \Pp\{|Z_k-f(N_k)|>\varepsilon\}
 &\leq \Pp(H_k^c)
      +\Pp\left\{\frac{L_k}{a_k}>\varepsilon\right\}
      +\Pp\{C_{k-1}<M\}\\
 &\longrightarrow0.
\end{align*}
Indeed, outside the three events on the right, the difference bound above
applies when $N_k\geq1$; when $N_k=0$, one has
$C_{k-1}\geq M$ and $|Z_k-f(N_k)|=Z_k\leq1/M<\varepsilon$.
The three terms tend
to zero by \eqref{eq:iid-lacunary-no-higher-level}, Markov's inequality
and \eqref{eq:iid-lacunary-lower-sum}, and
$C_{k-1}\to\infty$ in probability, respectively.

Thus $Z_k-f(N_k)\to0$ in probability.  Since
$N_k\Rightarrow N_c$ and $f$ is continuous on the discrete space
$\mathbb Z_{\geq0}$, Slutsky's theorem gives
\begin{equation}
 Z_k\Longrightarrow f(N_c)=R_c,
 \qquad
 R_c=\begin{cases}
 0,&N_c=0,\\
 1/N_c,&N_c\geq1.
 \end{cases}
 \label{eq:iid-lacunary-subsequence-limit}
\end{equation}
The laws for $c=1$ and $c=2$ differ, since
$\Pp\{R_c=1\}=ce^{-c}$.  Thus the first coordinate of
$W_n^\downarrow$ has different subsequential laws, and the partition
cannot converge in the product topology.

If $X$ is standard symmetric Cauchy, its characteristic function is
$e^{-|u|}$.  Conditional on any nonnegative weights of total mass one,
\[
 \E\left[\exp\left\{iu\sum_iW_{n,i}X_i\right\}
          \mathrel{\Big|}W_n\right]
 =\prod_i e^{-|u|W_{n,i}}=e^{-|u|}.
\]
The weighted mean is therefore standard Cauchy for every $n$, completing
the proof.
\end{proof}

\section{Normalized L\'evy jumps at small time}
\label{sec:smalltime}

We now consider the normalized jumps of a subordinator as $t\downarrow0$.
Nondegenerate convergence of the marked mean forces the fractional expected
power-sum bound required by Theorem~\ref{thm:countable-main}.  No regular
variation assumption is made.

Let $V=(V_t)_{t\geq0}$ be a nonzero, unkilled, driftless subordinator with
L\'evy measure $\nu$.  Thus
\begin{equation}
 \begin{gathered}
 \int_0^\infty(1\wedge x)\,\nu(\dd x)<\infty,
 \qquad \E e^{-qV_t}=e^{-t\Phi(q)},\\
 \Phi(q)=\int_0^\infty(1-e^{-qx})\,\nu(\dd x).
 \end{gathered}
 \label{eq:smalltime-laplace}
\end{equation}
Throughout this section,
\begin{equation}
 \nu(0,\infty)=\infty.
 \label{eq:smalltime-infinite-activity}
\end{equation}
Hence $V_t>0$ almost surely for every $t>0$.  If
$(J_i(t))_{i\geq1}$ enumerates the jumps on $(0,t]$, put
\begin{equation}
 Q_i(t)=\frac{J_i(t)}{V_t},\qquad
 m_t(p)=\E\sum_{i\geq1}Q_i(t)^p,\quad p>0.
 \label{eq:smalltime-power-sums}
\end{equation}
Then $\sum_iQ_i(t)=1$ almost surely.  Let $X,X_1,X_2,\ldots$ be iid,
independent of $V$, with
\begin{equation}
 \E|X|<\infty,\qquad X\ \text{nonconstant},
 \label{eq:smalltime-marks}
\end{equation}
and set
\begin{equation}
 R_t=\frac{U_t}{V_t}=\sum_{i\geq1}Q_i(t)X_i,
 \qquad U_t=\sum_{i\geq1}J_i(t)X_i.
 \label{eq:smalltime-marked-mean}
\end{equation}
Conditionally on the jumps, the expected absolute sums of the series
for $R_t$ and $U_t$ are $\E|X|$ and $V_t\E|X|$, respectively.  Since
$V_t<\infty$ almost surely, both series converge absolutely almost surely.
Write
\begin{equation}
 e(q)=\frac{q\Phi'(q)}{\Phi(q)},\qquad q>0.
 \label{eq:smalltime-log-slope}
\end{equation}
Infinite activity makes $\Phi$ a strictly increasing bijection from
$(0,\infty)$ onto itself.  Indeed, monotone convergence gives
$\Phi(q)\uparrow\nu(0,\infty)=\infty$ as $q\to\infty$.  Moreover,
\begin{equation}
 0<q\Phi'(q)<\Phi(q),\qquad 0<e(q)<1,
 \label{eq:smalltime-log-slope-range}
\end{equation}
because $ze^{-z}<1-e^{-z}$ for $z>0$.  The Campbell formula, quadratic and
fractional identities, and differential inequality are given in
Section~\ref{sec:poisson-calculus}.

\begin{theorem}[A power-sum bound at small time]
\label{thm:smalltime-admissibility}
Suppose that
\begin{equation}
 R_t\Rightarrow R\quad(t\downarrow0),\qquad
 R\ \text{nondegenerate}.
 \label{eq:smalltime-full-limit}
\end{equation}
Then
\begin{equation}
 a_\infty:=\limsup_{q\to\infty}e(q)<1.
 \label{eq:smalltime-slope-gap}
\end{equation}
For every $0<t_0<\infty$ and every
\begin{equation}
 a_\infty<b<r<1,
 \label{eq:smalltime-br-choice}
\end{equation}
there is $q_0>0$ such that, with $v_0=\Phi(q_0)$,
\begin{equation}
 \sup_{0<t\leq t_0}m_t(r)
 \leq
 \frac{e^{t_0v_0}}{\Gamma(r)\Gamma(1-r)}
 \mathrm B(r-b,1-r)<\infty.
 \label{eq:smalltime-subunit-bound}
\end{equation}
Moreover, $m_t(1)=1$ for every $t>0$.
\end{theorem}

\subsection{The Laplace-exponent estimate at infinity}
\label{sec:smalltime-slope}

No monotonicity of $q\mapsto e(q)$ is needed.  We use instead the
one-sided estimate in logarithmic $\Phi$-coordinates from
Lemma~\ref{lem:poisson-slope-dynamics}.

\begin{lemma}
\label{lem:smalltime-slope-gap}
Under \eqref{eq:smalltime-full-limit},
\[
 \limsup_{q\to\infty}e(q)<1.
\]
\end{lemma}

\begin{proof}
Suppose that the limsup is one.  Choose $q_k\to\infty$ such that
\[
 \delta_k=1-e(q_k)\downarrow0,
 \qquad x_k=\log\Phi(q_k)\to\infty,
\]
and put
\begin{equation}
 M_k=\min\{\delta_k^{-1/2},\Phi(q_k)^{1/2}\},
 \qquad t_k=\frac{M_k}{\Phi(q_k)}.
 \label{eq:smalltime-mk-tk}
\end{equation}
Then $M_k\to\infty$ and
$0<t_k\leq\Phi(q_k)^{-1/2}\to0$.  For $0<v\leq M_k$, set
\[
 y=\log(v/t_k)=x_k+\log(v/M_k)\leq x_k.
\]
Equation~\eqref{eq:poisson-left-window} gives
\begin{equation}
 1-e\!\left(\Phi^{-1}(v/t_k)\right)
 \leq\frac{\delta_kM_k}{v}.
 \label{eq:smalltime-window-bound}
\end{equation}
On $(M_k,\infty)$, use only \eqref{eq:smalltime-log-slope-range}.
Since $\lambda=\nu(0,\infty)=\infty$, the first term
$t_k\lambda e^{-t_k\lambda}$ in \eqref{eq:poisson-quadratic} is zero under
the stated convention, and the upper integration limit $t_k\lambda$ is
infinite.  Therefore
\begin{align*}
 m_{t_k}(2)
 &\leq\delta_kM_k\int_0^{M_k}e^{-v}\dd v
     +\int_{M_k}^\infty ve^{-v}\dd v\\
 &\leq\sqrt{\delta_k}+(M_k+1)e^{-M_k}
 \longrightarrow0.
\end{align*}
With $Q_*(t)=\sup_iQ_i(t)$,
\[
 \Pp\{Q_*(t_k)>\varepsilon\}
 \leq\varepsilon^{-2}m_{t_k}(2)\longrightarrow0.
\]
Lemma~\ref{lem:countable-wlln}, applied to $X_i-\E X$, now yields
$R_{t_k}\to\E X$ in probability.  This contradicts the nondegenerate
limit in \eqref{eq:smalltime-full-limit}.
\end{proof}

The use of convergence as $t\downarrow0$, rather than convergence along a
fixed subsequence, is essential: the times in
\eqref{eq:smalltime-mk-tk} are selected from the Laplace exponent.

\subsection{A fractional power-sum bound}
\label{sec:smalltime-subunit}

The representation \eqref{eq:poisson-fractional-representation} now reduces
the problem to an endpoint-specific comparison for its inner integral.

\begin{lemma}[Global comparison]
\label{lem:smalltime-global-comparison}
Suppose that $e(q)\leq b$ for $q\geq q_0$, where $0<b<1$, and put
$v_0=\Phi(q_0)$.  Then, for all $s,h>0$,
\begin{equation}
 \Phi(e^{-h}s)\geq e^{-bh}\Phi(s)-v_0.
 \label{eq:smalltime-global-comparison}
\end{equation}
\end{lemma}

\begin{proof}
If $e^{-h}s\geq q_0$, then
\[
 \log\frac{\Phi(s)}{\Phi(e^{-h}s)}
 =\int_{-h}^0e(se^w)\dd w\leq bh,
\]
which gives the stronger inequality
$\Phi(e^{-h}s)\geq e^{-bh}\Phi(s)$.  If
$e^{-h}s<q_0\leq s$, then
\[
 \frac{\Phi(s)}{\Phi(q_0)}
 \leq\left(\frac{s}{q_0}\right)^b<e^{bh},
\]
so the right-hand side of
\eqref{eq:smalltime-global-comparison} is negative.  The same is immediate
when $s<q_0$.
\end{proof}

\begin{lemma}[Fractional bound]
\label{lem:smalltime-subunit-bound}
Suppose that $\limsup_{q\to\infty}e(q)<1$.  Fix $0<t_0<\infty$ and
choose $b,r$ as in \eqref{eq:smalltime-br-choice}.  Then
\eqref{eq:smalltime-subunit-bound} holds.
\end{lemma}

\begin{proof}
Choose $q_0$ so that $e(q)\leq b$ for $q\geq q_0$, and put
$v_0=\Phi(q_0)$.  Lemma~\ref{lem:smalltime-global-comparison} and
\eqref{eq:poisson-I-def} imply, for $0<t\leq t_0$,
\[
 I_t(h)\leq e^{tv_0}\int_0^\infty e^{-te^{-bh}v}\dd v
 =\frac{e^{tv_0}e^{bh}}{t}
 \leq\frac{e^{t_0v_0}e^{bh}}{t}.
\]
Furthermore,
\begin{align*}
 \int_0^\infty(e^h-1)^{-r}e^{bh}\dd h
 &=\int_0^1x^{r-b-1}(1-x)^{-r}\dd x\\
 &=\mathrm B(r-b,1-r)<\infty.
\end{align*}
Here $r>b$ controls the integral at infinity, while $r<1$ controls it at
zero.  Combining this identity with
\eqref{eq:poisson-fractional-representation} proves the claim.
\end{proof}

\begin{proof}[Proof of Theorem~\ref{thm:smalltime-admissibility}]
Lemma~\ref{lem:smalltime-slope-gap} proves
\eqref{eq:smalltime-slope-gap}.  Choose $b,r$ as in
\eqref{eq:smalltime-br-choice}; if $a_\infty=0$, take $b>0$.
Lemma~\ref{lem:smalltime-subunit-bound} gives
\eqref{eq:smalltime-subunit-bound}, and
Lemma~\ref{lem:poisson-campbell} gives $m_t(1)=1$ because
$\nu(0,\infty)=\infty$.
\end{proof}

\begin{corollary}[Positive small-time limit of an expected power sum]
\label{cor:smalltime-positive-power}
Under the hypotheses of Theorem~\ref{thm:smalltime-admissibility}, there
exist $p\in(1,2]$ and $c\in(0,1]$ such that
\[
 m_t(p)\longrightarrow c\qquad(t\downarrow0).
\]
\end{corollary}

\begin{proof}
Replace the marks by $\widehat X=X-\E X$.  Since the weights have total
mass one, the corresponding weighted mean is $R_t-\E X$ and has a
nondegenerate limit.  Moreover,
$0<\E|\widehat X|<\infty$.  The fractional bound in
Theorem~\ref{thm:smalltime-admissibility} verifies
\eqref{eq:countable-r-bound}; Theorem~\ref{thm:countable-main} gives the
conclusion.
\end{proof}

The bound in Theorem~\ref{thm:smalltime-admissibility} is local to
$t\downarrow0$.  This cannot be replaced by a bound uniform over all
$t>0$, even for the gamma subordinator.  If
$\nu(\dd x)=x^{-1}e^{-x}\dd x$, then
$\Phi(q)=\log(1+q)$ and, for $0<r<1$,
\begin{equation}
 m_t(r)=t\mathrm B(r,t)
 =\frac{\Gamma(t+1)\Gamma(r)}{\Gamma(t+r)}.
 \label{eq:smalltime-gamma}
\end{equation}
Thus $m_t(r)\to1$ as $t\downarrow0$, whereas
$m_t(r)\sim\Gamma(r)t^{1-r}$ as $t\to\infty$.

\section{A small-time ratio--Tauberian theorem}
\label{sec:tauberian}

Corollary~\ref{cor:smalltime-positive-power} combines the fractional
power-sum bound with Theorem~\ref{thm:countable-main} to produce a
positive expected power-sum limit from a nondegenerate marked-mean limit.
We now determine exactly which L\'evy tails can produce such a power-sum
limit.  The result is stated for an arbitrary fixed $p>1$ and does not
assume any asymptotic behavior of $q\Phi'(q)/\Phi(q)$.
Its ingredients are an exact push-forward of Campbell's
formula, a Weyl fractional-integral identity, and a Drasin--Shea theorem on
the maximal Mellin interval of the resulting kernel.

This is the small-time counterpart of
Proposition~\ref{prop:breiman-tauberian}: the maximal Mellin interval now has its
finite boundary on the opposite side, and the monotone quantity is built
from the L\'evy tail at zero rather than an iid tail at infinity.  The
large-time expected-power-sum criterion and its antecedents are recalled in
Section~\ref{sec:breiman-tauberian}.

Throughout this section, $V$ is an unkilled, driftless subordinator with
infinite L\'evy measure $\nu$.  We write
\begin{equation}
 \Phi(q)=\int_0^\infty(1-e^{-qx})\,\nu(\dd x),
 \qquad
 \bar\nu(x)=\nu((x,\infty)).
 \label{eq:tauberian-laplace}
\end{equation}
Infinite activity makes $\Phi$ a continuous strictly increasing bijection
of $[0,\infty)$ onto itself.  If $(J_i(t))_{i\geq1}$ are the jumps on
$(0,t]$, then
\[
 Q_i(t)=\frac{J_i(t)}{V_t},
 \qquad
 m_t(p)=\E\sum_{i\geq1}Q_i(t)^p,
 \qquad p>1.
\]
Since there is no drift, $(Q_i(t))$ has total mass one almost surely.

\begin{theorem}[Small-time ratio--Tauberian theorem]
\label{thm:tauberian-main}
Fix $p>1$.  The following assertions are equivalent.
\begin{enumerate}
\item There is a constant $c>0$ such that
\begin{equation}
 m_t(p)\longrightarrow c \qquad (t\downarrow0).
 \label{eq:tauberian-positive-limit}
\end{equation}
\item For a necessarily unique $\alpha\in[0,1)$,
\begin{equation}
 \bar\nu\in\RV_0(-\alpha).
 \label{eq:tauberian-tail-rv}
\end{equation}
\end{enumerate}
In that case,
\begin{equation}
 c=\frac{\Gamma(p-\alpha)}
         {\Gamma(p)\Gamma(1-\alpha)}.
 \label{eq:tauberian-pd-constant}
\end{equation}
In particular, $0<c\leq1$, and $c=1$ if and only if $\alpha=0$.
\end{theorem}

The restriction $p>1$ is intrinsic to the transform argument below.  No
upper bound on $p$ is needed here; the range $p\leq2$ enters only when $p$
is selected by Theorem~\ref{thm:countable-main}.  The zero limit is not
obtained by putting $\alpha=1$ in Theorem~\ref{thm:tauberian-main}.  It is
a distinct boundary regime, treated in
Proposition~\ref{prop:tauberian-dust} below.

\subsection{Campbell's formula and a push-forward Laplace transform}
\label{sec:tauberian-pushforward}

For fixed $p>1$, put
\begin{equation}
 A_p(q)=\int_0^\infty x^p e^{-qx}\,\nu(\dd x),
 \qquad q>0,
 \label{eq:tauberian-Ap}
\end{equation}
and set
\begin{equation}
 C_p(0)=0,
 \qquad
 C_p(q)=\int_0^q s^{p-1}A_p(s)\dd s,
 \quad q>0.
 \label{eq:tauberian-Cp}
\end{equation}
Here $A_p(q)<\infty$ for $q>0$, and $C_p(q)<\infty$ for $q\geq0$.
Indeed, Tonelli's
theorem gives the useful representation
\begin{equation}
 C_p(q)=\int_0^\infty
 \left(\int_0^{qx}z^{p-1}e^{-z}\dd z\right)\nu(\dd x).
 \label{eq:tauberian-Cp-incomplete-gamma}
\end{equation}
The inner integral is $O_q(x^p)$ on $(0,1]$ and is bounded by
$\Gamma(p)$ on $(1,\infty)$.  Thus $C_p$ is continuous,
nondecreasing, and locally bounded, with $C_p(0)=0$.

Lemma~\ref{lem:poisson-campbell} gives, for $p>1$ and $t>0$,
\begin{equation}
 m_t(p)=\frac{t}{\Gamma(p)}
 \int_0^\infty q^{p-1}A_p(q)e^{-t\Phi(q)}\dd q.
 \label{eq:tauberian-campbell}
\end{equation}

Push the locally finite measure $\dd C_p$ forward under $q\mapsto\Phi(q)$,
and denote its cumulative function by
\begin{equation}
 U_p(v)=C_p(\Phi^{-1}(v)),\qquad v\geq0.
 \label{eq:tauberian-Up}
\end{equation}
Equation~\eqref{eq:tauberian-campbell} becomes
\begin{equation}
 \Gamma(p)m_t(p)
 =t\int_{[0,\infty)}e^{-tv}\dd U_p(v).
 \label{eq:tauberian-pushforward}
\end{equation}

\begin{lemma}[Push-forward equivalence]
\label{lem:tauberian-pushforward}
For $L\in[0,\infty)$,
\begin{equation}
 m_t(p)\longrightarrow \frac{L}{\Gamma(p)}
 \quad(t\downarrow0)
 \quad\Longleftrightarrow\quad
 \frac{C_p(q)}{\Phi(q)}\longrightarrow L
 \quad(q\to\infty).
 \label{eq:tauberian-pushforward-equivalence}
\end{equation}
\end{lemma}

\begin{proof}
For $L>0$, this is the index-one case of Karamata's
Laplace--Stieltjes theorem~\cite{BGT} applied to the
nondecreasing function $U_p$:
\[
 t\int_{[0,\infty)}e^{-tv}\dd U_p(v)\longrightarrow L
 \quad\Longleftrightarrow\quad
 U_p(v)\sim Lv.
\]
Together with \eqref{eq:tauberian-Up}, this is precisely
\eqref{eq:tauberian-pushforward-equivalence}.

The case $L=0$ is elementary and will be needed at the dust boundary.
If the transform tends to zero, then
\[
 e^{-1}tU_p(1/t)
 \leq t\int_{[0,1/t]}e^{-tv}\dd U_p(v)\longrightarrow0,
\]
so $U_p(v)=o(v)$.  Conversely, first integrate by parts over
$[\varepsilon,M]$:
\begin{align*}
 \int_{(\varepsilon,M]}e^{-tv}\dd U_p(v)
 &=e^{-tM}U_p(M)-e^{-t\varepsilon}U_p(\varepsilon)\\
 &\quad+t\int_\varepsilon^M e^{-tv}U_p(v)\dd v.
\end{align*}
Here $U_p(\varepsilon)\to U_p(0)=0$.  Moreover,
\eqref{eq:tauberian-pushforward} gives
\[
 \int_{[0,\infty)}e^{-tv/2}\dd U_p(v)
 =\frac{2\Gamma(p)}{t}m_{t/2}(p)<\infty.
\]
Here $m_{t/2}(p)\leq1$.
Consequently,
\[
 e^{-tM}U_p(M)
 \leq e^{-tM/2}\int_{[0,M]}e^{-tv/2}\dd U_p(v)
 \longrightarrow0.
\]
Letting $\varepsilon\downarrow0$ and $M\to\infty$ therefore gives
\begin{equation}
 t\int_{[0,\infty)}e^{-tv}\dd U_p(v)
 =t^2\int_0^\infty e^{-tv}U_p(v)\dd v.
 \label{eq:tauberian-zero-integration-parts}
\end{equation}
If $U_p(v)=o(v)$, fix $V$ so that $U_p(v)\leq\eta v$ for $v\geq V$.
The part over $[0,V]$ is $O(t^2)$, whereas the remaining part is at most
\[
 \eta t^2\int_V^\infty ve^{-tv}\dd v\leq\eta.
\]
Letting first $t\downarrow0$ and then $\eta\downarrow0$ proves the
assertion also for $L=0$.
\end{proof}

In particular, the positive limit in
\eqref{eq:tauberian-positive-limit} is equivalent to
\begin{equation}
 \frac{C_p(q)}{\Phi(q)}\longrightarrow c\Gamma(p).
 \label{eq:tauberian-Cp-Phi-ratio}
\end{equation}

\subsection{The fractional identity and its Mellin kernel}
\label{sec:tauberian-mellin}

The two ratio--Tauberian arguments in the paper use the same Mellin
convolution convention but different endpoint data:
\[
\begin{array}{c|c|c|c}
 &\text{Laplace limit}&\text{kernel support}&
   \text{maximal Mellin strip}\\ \hline
\text{iid weights}&q\downarrow0&(1,\infty)&(-1,\infty)\\
\text{small-time jumps}&q\to\infty&(0,1)&(-\infty,1)
\end{array}
\]
Thus only the left boundary is finite in the first case, while only the
right boundary is finite in the second.

For $s\geq1$ and $q>0$, define
\begin{equation}
 F_s(q)=\int_0^\infty
        \bar\nu(x)x^{s-1}e^{-qx}\dd x.
 \label{eq:tauberian-Fs}
\end{equation}
Tonelli's theorem, used without any assumption that $\nu$ has a
density, yields the exact identities
\begin{equation}
 \Phi(q)=qF_1(q),
 \qquad
 C_p(q)=q^pF_p(q).
 \label{eq:tauberian-tail-identities}
\end{equation}
For example,
\[
 q^pF_p(q)
 =\int_0^\infty\nu(\dd y)
   \int_0^y q^p x^{p-1}e^{-qx}\dd x,
\]
which agrees with \eqref{eq:tauberian-Cp-incomplete-gamma}.

The same method gives the Weyl fractional-integral identity
\begin{equation}
 F_1(q)=\frac1{\Gamma(p-1)}
 \int_0^\infty v^{p-2}F_p(q+v)\dd v.
 \label{eq:tauberian-weyl}
\end{equation}
Indeed,
$
 \Gamma(p-1)^{-1}\int_0^\infty v^{p-2}e^{-vx}\dd v=x^{1-p}
$,
and Tonelli applies throughout.

Suppose now that \eqref{eq:tauberian-Cp-Phi-ratio} holds.  Combining it
with \eqref{eq:tauberian-tail-identities} and
\eqref{eq:tauberian-weyl}, and then setting $v=q(u-1)$, gives
\begin{equation}
 \int_1^\infty (u-1)^{p-2}
 \frac{F_p(qu)}{F_p(q)}\dd u
 \longrightarrow\frac1{c(p-1)}.
 \label{eq:tauberian-weyl-ratio}
\end{equation}
Let
\begin{equation}
 a(q)=q^pF_p(q)=C_p(q)
 \label{eq:tauberian-a}
\end{equation}
and use the Mellin-convolution convention
\begin{equation}
 (f*_{M}K)(q)=\int_0^\infty f(q/w)K(w)\,\frac{\dd w}{w}.
 \label{eq:tauberian-mellin-convention}
\end{equation}
The change of variables $w=1/u$ turns
\eqref{eq:tauberian-weyl-ratio} into
\begin{equation}
 \frac{(a*_{M}K)(q)}{a(q)}
 \longrightarrow\frac1{c(p-1)},
 \qquad
 K(w)=w(1-w)^{p-2}\1_{(0,1)}(w).
 \label{eq:tauberian-mellin-ratio}
\end{equation}
The convolution is finite for every $q>0$: its ratio to $a(q)$ is exactly
the left-hand side of \eqref{eq:tauberian-weyl-ratio}.

\begin{lemma}[The Drasin--Shea step]
\label{lem:tauberian-drasin-shea}
If \eqref{eq:tauberian-mellin-ratio} holds for some $c>0$, then there is
a unique $\alpha\in[0,1)$ such that
\begin{equation}
 C_p\in\RV_\infty(\alpha)
 \quad\text{and}\quad
 \mathrm B(1-\alpha,p-1)=\frac1{c(p-1)}.
 \label{eq:tauberian-drasin-shea-output}
\end{equation}
Consequently,
\begin{equation}
 c=\frac{\Gamma(p-\alpha)}
         {\Gamma(p)\Gamma(1-\alpha)}.
 \label{eq:tauberian-constant-derived}
\end{equation}
\end{lemma}

\begin{proof}
For the convention \eqref{eq:tauberian-mellin-convention}, the Mellin
transform of the kernel is
\begin{align}
 \mathcal M K(z)
 &=\int_0^\infty w^{-z}K(w)\,\frac{\dd w}{w} \notag\\
 &=\int_0^1w^{-z}(1-w)^{p-2}\dd w
 =\mathrm B(1-z,p-1).
 \label{eq:tauberian-kernel-transform}
\end{align}
Its maximal Mellin strip is
\[
 -\infty<\Re z<1.
\]
The only finite endpoint is $1$, and
\begin{equation}
 \mathcal M K(z)\longrightarrow\infty
 \qquad(z\uparrow1).
 \label{eq:tauberian-kernel-divergence}
\end{equation}

The kernel \(K\) is nonnegative, measurable, and nonzero.  Its maximal
Mellin interval is the one just computed, with the required divergence at
its only finite endpoint.  The function \(a=C_p\) is nonnegative and
locally bounded by
\eqref{eq:tauberian-Cp-incomplete-gamma}.  It is also nondecreasing, so
for every $q>0$ and $1\leq\lambda\leq2$,
\begin{equation}
 a(\lambda q)\geq a(q).
 \label{eq:tauberian-one-sided-condition}
\end{equation}
This is the one-sided condition in
Proposition~\ref{prop:poisson-mellin-ratio}, with constant one.  The Mellin
convolution is finite and its ratio has the positive limit in
\eqref{eq:tauberian-mellin-ratio}.  The proposition therefore gives, for
some
$\rho\in(-\infty,1)$,
\[
 a\in\RV_\infty(\rho),
 \qquad
 \mathrm B(1-\rho,p-1)=\frac1{c(p-1)}.
\]

Monotonicity of $a$ rules out $\rho<0$: for $\lambda>1$,
$a(\lambda q)/a(q)\geq1$, while regular variation would make this ratio
tend to $\lambda^\rho<1$.  Hence $0\leq\rho<1$.  Set
$\alpha=\rho$.  The beta identity gives
\[
 \frac{\Gamma(1-\alpha)\Gamma(p-1)}
      {\Gamma(p-\alpha)}
 =\frac1{c(p-1)},
\]
which is \eqref{eq:tauberian-constant-derived}.  For real $z<1$,
$\mathcal M K(z)$ is strictly increasing, since differentiation under the
integral in \eqref{eq:tauberian-kernel-transform} gives an integral with
the strictly positive factor $-\log w$.  The index $\alpha$ is therefore
unique.
\end{proof}

By \eqref{eq:tauberian-Cp-Phi-ratio},
Lemma~\ref{lem:tauberian-drasin-shea} also yields
\begin{equation}
 \Phi\in\RV_\infty(\alpha),
 \qquad 0\leq\alpha<1.
 \label{eq:tauberian-Phi-rv}
\end{equation}

\subsection{Recovering the L\'evy tail}
\label{sec:tauberian-tail-recovery}

Define the integrated tail
\begin{equation}
 I(x)=\int_0^x\bar\nu(y)\dd y.
 \label{eq:tauberian-I}
\end{equation}
It is finite by the L\'evy condition, and $F_1$ is its
Laplace--Stieltjes transform.  From
\eqref{eq:tauberian-tail-identities} and
\eqref{eq:tauberian-Phi-rv},
\[
 F_1(q)=\frac{\Phi(q)}q
 \in\RV_\infty(\alpha-1).
\]
Karamata's Laplace--Stieltjes theorem at zero now gives
\begin{equation}
 I(x)\sim\frac{F_1(1/x)}{\Gamma(2-\alpha)}
 =\frac{x\Phi(1/x)}{\Gamma(2-\alpha)},
 \qquad x\downarrow0.
 \label{eq:tauberian-I-asymptotic}
\end{equation}
Thus $I\in\RV_0(1-\alpha)$.  Since $1-\alpha$ is
strictly positive and $\bar\nu$ is nonincreasing, the
monotone-density theorem applies to
$I(x)=\int_0^x\bar\nu(y)\dd y$ and yields
\begin{align}
 \bar\nu(x)
 &\sim(1-\alpha)\frac{I(x)}x \notag\\
 &\sim\frac{\Phi(1/x)}{\Gamma(1-\alpha)}
 \in\RV_0(-\alpha).
 \label{eq:tauberian-tail-recovery}
\end{align}
This proves the Tauberian direction of
Theorem~\ref{thm:tauberian-main}.  Notice that $\alpha=0$ is included
without an endpoint argument: in that case $I$ has index one, so the
monotone-density theorem is still being used at a positive index.

For the converse, suppose that
$\bar\nu\in\RV_0(-\alpha)$ for some
$0\leq\alpha<1$.  Karamata's theorem for Laplace transforms gives,
for $s=1$ and $s=p$,
\begin{equation}
 F_s(q)\sim
 \Gamma(s-\alpha)q^{-s}\bar\nu(1/q),
 \qquad q\to\infty.
 \label{eq:tauberian-Fs-abelian}
\end{equation}
Using \eqref{eq:tauberian-tail-identities}, we obtain
\begin{equation}
 \frac{C_p(q)}{\Phi(q)}
 =q^{p-1}\frac{F_p(q)}{F_1(q)}
 \longrightarrow
 \frac{\Gamma(p-\alpha)}{\Gamma(1-\alpha)}.
 \label{eq:tauberian-abelian-ratio}
\end{equation}
Lemma~\ref{lem:tauberian-pushforward} completes the Abelian direction
and proves Theorem~\ref{thm:tauberian-main}.

\subsection{The zero limit and the dust boundary}
\label{sec:tauberian-dust}

Put
\[
 Q_*(t)=\sup_iQ_i(t),
 \qquad
 A(x)=\int_{(0,x]}y\,\nu(\dd y).
\]
We now treat the case in which the expected power sum tends to zero.

\begin{proposition}[Vanishing expected power sums and pure dust]
\label{prop:tauberian-dust}
Fix $p>1$.  As $t\downarrow0$, the following are equivalent:
\begin{enumerate}
\item $m_t(p)\to0$;
\item $m_t(s)\to0$ for every $s>1$;
\item $Q_*(t)\to0$ in probability;
\item
\begin{equation}
 \frac{x\bar\nu(x)}{A(x)}\longrightarrow0
 \qquad(x\downarrow0);
 \label{eq:tauberian-maximal-jump-criterion}
\end{equation}
\item $I\in\RV_0(0)$.
\end{enumerate}
Under these conditions,
\begin{equation}
 \frac{C_p(q)}{\Phi(q)}\longrightarrow0,
 \qquad
 \Phi(q)\sim qI(1/q)\in\RV_\infty(1),
 \qquad
 \frac{q\Phi'(q)}{\Phi(q)}\longrightarrow1.
 \label{eq:tauberian-dust-phi}
\end{equation}
No regular variation of $\bar\nu$ with index $-1$ follows.
\end{proposition}

\begin{proof}
For every probability mass partition and every $s>1$,
\begin{equation}
 Q_*^s\leq\sum_iQ_i^s\leq Q_*^{s-1}.
 \label{eq:tauberian-max-power-bounds}
\end{equation}
If the expected $p$-power sum in assertion \textup{(1)} tends to zero, its
lower bound and Markov's inequality give $Q_*(t)\to0$ in probability.
Conversely, the upper bound then tends to zero in probability for every
$s>1$ and is
bounded by one, so it tends to zero in mean.  This proves the equivalence
of the first three assertions.  Moreover, on $\{V_t>0\}$,
\[
 \frac{V_t}{J_{(1)}(t)}=\frac1{Q_*(t)}.
\]
Thus the equivalence of the third and fourth is the $k=0$ case of the
maximal-jump criterion in \cite[Theorem~1.1(iii)]{KM14}.

The fourth and fifth assertions are equivalent by
Lemma~\ref{lem:poisson-integrated-tail}, whose proof treats scale factors
both above and below one and uses the identity
$I(x)=A(x)+x\bar\nu(x)$.

Lemma~\ref{lem:tauberian-pushforward}, in its zero-limit form, gives the
first conclusion in \eqref{eq:tauberian-dust-phi}.  Since $F_1$ is the
Laplace--Stieltjes transform of the slowly varying function $I$,
Karamata's theorem and \eqref{eq:tauberian-tail-identities} yield
\[
 \Phi(q)=qF_1(q)\sim qI(1/q).
\]
Finally, $\Phi$ is concave and regularly varying with the positive index
one.  The monotone-density theorem applied to $\Phi$ therefore gives
$q\Phi'(q)/\Phi(q)\to1$.  This last use of monotone density is at index
one; no density conclusion is drawn from the index-zero behavior of
$I$.
\end{proof}

\begin{remark}[Why subsequences are excluded]
\label{rem:tauberian-endpoint-dichotomy}
The push-forward equivalence uses ordinary Karamata theory and therefore
requires convergence as $t\to\star$, not merely along a subsequence.
\end{remark}

\section{The large-time regime}
\label{sec:largetime}

At infinity, the finite-mean case gives a constant limit.  When the mean
is infinite, bounded jumps are negligible and the remaining
compound-Poisson ratio reduces to the iid problem by depoissonization.

Recall that $V$ is a nonzero, unkilled, driftless subordinator with L\'evy
measure $\nu$, and that
\[
 U_t=\sum_{0<s\leq t}X_s\Delta V_s,
 \qquad
 V_t=\sum_{0<s\leq t}\Delta V_s,
 \qquad
 R_t=\frac{U_t}{V_t}
\]
on $\{V_t>0\}$.  The marks are iid, independent of $V$, nonconstant, and
satisfy $\E|X|<\infty$; we write $\mu=\E X$.  Any fixed convention for
$R_t$ on $\{V_t=0\}$ is immaterial as $t\to\infty$, since this event has
probability tending to zero.

Split the marked and unmarked processes at jump size one:
\[
 V_t=V_t^{\leq1}+V_t^{>1},
 \qquad
 U_t=U_t^{\leq1}+U_t^{>1}.
\]

\begin{lemma}[The finite-mean case]
\label{lem:largetime-finite-mean}
If
\[
 \int_1^\infty x\,\nu(\dd x)<\infty,
\]
then
\[
 R_t\longrightarrow\mu
 \qquad\text{almost surely as }t\to\infty.
\]
\end{lemma}

\begin{proof}
The L\'evy condition already gives
$\int_{(0,1]}x\,\nu(\dd x)<\infty$.  Thus
\[
 m_V:=\int_0^\infty x\,\nu(\dd x)\in(0,\infty).
\]
Moreover, the marked process has integrable total variation on bounded time
intervals because
\[
 \E|X|\int_0^\infty x\,\nu(\dd x)<\infty.
\]
The strong law for integrable L\'evy processes therefore gives
\[
 \frac{V_t}{t}\longrightarrow m_V,
 \qquad
 \frac{U_t}{t}\longrightarrow\mu m_V
 \qquad\text{almost surely},
\]
and the assertion follows by taking the ratio.
\end{proof}

Suppose henceforth that
\begin{equation}
 \int_1^\infty x\,\nu(\dd x)=\infty.
 \label{eq:largetime-infinite-mean}
\end{equation}
Set $\lambda=\bar\nu(1)$.  Then $0<\lambda<\infty$, and the jumps
larger than one form a compound Poisson process.  If $N=(N_t)$ is a
rate-$\lambda$ Poisson process and $Y_1,Y_2,\ldots$ are iid with law
\begin{equation}
 \Pp(Y\in\dd y)
 =\lambda^{-1}\1_{(1,\infty)}(y)\nu(\dd y),
 \label{eq:largetime-jump-law}
\end{equation}
with $N$, $(Y_j)$ and the mark sequence $(X_j)$ mutually independent, then
\begin{equation}
 (U_t^{>1},V_t^{>1})
 \overset{d}{=}
 \left(\sum_{j\leq N_t}X_jY_j,
       \sum_{j\leq N_t}Y_j\right).
 \label{eq:largetime-compound-representation}
\end{equation}

\begin{lemma}[Removal of bounded jumps]
\label{lem:largetime-bounded-removal}
Under \eqref{eq:largetime-infinite-mean},
\[
 \frac{V_t^{\leq1}}{V_t^{>1}}\longrightarrow0,
 \qquad
 \frac{U_t^{\leq1}}{V_t^{>1}}\longrightarrow0
 \qquad\text{in probability}.
\]
Consequently,
\begin{equation}
 \frac{U_t}{V_t}-\frac{U_t^{>1}}{V_t^{>1}}
 \longrightarrow0
 \qquad\text{in probability},
 \label{eq:largetime-ratio-reduction}
\end{equation}
where the second ratio may be defined arbitrarily on
$\{V_t^{>1}=0\}$.
\end{lemma}

\begin{proof}
The components generated by jumps of $V$ of size at most one have finite absolute first moments, so
\[
 V_t^{\leq1}=O_{\Pp}(t),
 \qquad
 U_t^{\leq1}=O_{\Pp}(t).
\]
On the other hand, the variable $Y$ in
\eqref{eq:largetime-jump-law} has infinite mean.  For nonnegative iid
variables of infinite mean,
\[
 \frac1n\sum_{j=1}^nY_j\longrightarrow\infty
 \qquad\text{almost surely};
\]
this follows by applying the strong law to $Y\wedge K$ and then letting
$K\to\infty$.  Since $N_t/t\to\lambda$ almost surely, it follows that
\[
 \frac{V_t^{>1}}{t}\longrightarrow\infty
 \qquad\text{almost surely}.
\]
This proves the first two assertions.

The compound-Poisson marked ratio is tight.  Indeed, conditionally on the
positive weights induced by the $Y_j$'s,
\[
 \E\left(
   \left|\frac{U_t^{>1}}{V_t^{>1}}\right|
   \mathrel{\Big|}(Y_j),N_t
 \right)\leq \E|X|
\]
on $\{N_t>0\}$.  Writing $A=U_t^{>1}$, $B=V_t^{>1}$,
$a=U_t^{\leq1}$, and $b=V_t^{\leq1}$, we have on $\{B>0\}$
\[
 \left|\frac{A+a}{B+b}-\frac AB\right|
 \leq \frac{|a|}{B}+\left|\frac AB\right|\frac bB.
\]
Both terms tend to zero in probability, while
$\Pp\{B=0\}=e^{-\lambda t}\to0$.
\end{proof}

The transfer between deterministic and Poisson sample sizes, requiring no
moment of the jump size $Y$, is Lemma~\ref{lem:breiman-depoissonization}.
We record its estimate explicitly when it is used below.

\begin{proposition}[The nondegenerate large-time branch]
\label{prop:largetime-nondegenerate}
The ratio $R_t$ converges in distribution to a nondegenerate limit as
$t\to\infty$ if and only if, for a unique $\alpha\in[0,1)$,
\begin{equation}
 \bar\nu\in\RV_\infty(-\alpha).
 \label{eq:largetime-regular-tail}
\end{equation}
In that case the limit has the law
\begin{equation}
 W_\alpha\overset{d}{=}
 \begin{cases}
  \displaystyle\sum_{i\geq1}Q_iX_i,
  &(Q_i)^\downarrow\sim\PD(\alpha,0),\quad 0<\alpha<1,\\[2mm]
  X,&\alpha=0,
 \end{cases}
 \label{eq:largetime-limit-law}
\end{equation}
where the weights and marks on the right are independent.
\end{proposition}

\begin{proof}
Lemma~\ref{lem:largetime-finite-mean} excludes a nondegenerate limit when
$\int_1^\infty x\,\nu(\dd x)<\infty$.  Under
\eqref{eq:largetime-infinite-mean}, Lemma
\ref{lem:largetime-bounded-removal} reduces the question to the ratio in
\eqref{eq:largetime-compound-representation}.  For the remainder of the
proof, set $U_t^{>1}/V_t^{>1}=0$ on $\{N_t=0\}$, in agreement with the
convention $T_0=0$.  At the times $t_n=n/\lambda$, this ratio then has
exactly the law of $T_{\Pi_n}$, where $\Pi_n$ is Poisson with mean $n$ and
is independent of the iid sequences defining $T_m$.  In the coupling of
Lemma~\ref{lem:breiman-depoissonization},
\begin{equation}
 \E|T_{\Pi_n}-T_n|
 \leq \frac{2\E|X|}{n}\E|\Pi_n-n|
 =O(n^{-1/2}).
 \label{eq:largetime-depoissonization-rate}
\end{equation}

If the compound-Poisson ratio converges as $t\to\infty$,
Lemma~\ref{lem:breiman-depoissonization} transfers its convergence to
$T_n$.  The iid converse in Theorem~\ref{thm:breiman} therefore gives
\[
 \Pp(Y>x)\in\RV_\infty(-\alpha)
 \quad\text{for a unique }\alpha\in[0,1).
\]
Since $\Pp(Y>x)=\bar\nu(x)/\lambda$ for $x>1$, this is
\eqref{eq:largetime-regular-tail}.

Conversely, suppose that \eqref{eq:largetime-regular-tail} holds.  Since
$-\alpha>-1$, Karamata's theorem gives
\[
 \int_1^x\bar\nu(y)\dd y
 \sim \frac{x\bar\nu(x)}{1-\alpha}
 \longrightarrow\infty.
\]
Tonelli's theorem therefore implies
\[
 \int_{(1,\infty)}y\,\nu(\dd y)
 =\bar\nu(1)+\int_1^\infty\bar\nu(y)\dd y
 =\infty,
\]
so \eqref{eq:largetime-infinite-mean} holds.  Theorem~\ref{thm:breiman}
gives $T_n\Rightarrow W_\alpha$.  The standard random-index argument yields
$T_{N_t}\Rightarrow W_\alpha$, since $N_t\to\infty$ in probability, and
Lemma~\ref{lem:largetime-bounded-removal} restores the bounded jumps.  The
limit law is the one in \eqref{eq:largetime-limit-law}.
\end{proof}

\section{Classification of convergent endpoint limits}
\label{sec:classification}

If a limit exists at either endpoint, the nondegenerate branch forces regular
variation, whereas the constant branch is exactly disappearance of the largest normalized
jump.  We begin with the latter equivalence under the first moment
alone.

List the jumps on $(0,t]$ in nonincreasing order as
$J_{(1)}(t),J_{(2)}(t),\ldots$ and, on $\{V_t>0\}$, put
\[
 Q_i(t)=\frac{J_{(i)}(t)}{V_t},
 \qquad Q_*(t)=Q_1(t).
\]
On $\{V_t=0\}$, retain the global convention $Q^\downarrow(t)=0$ and put
$Q_*(t)=0$.  For the sole purpose of applying the mass-one constant-limit
theorem at infinity, define an auxiliary partition
$\widetilde Q^\downarrow(t)$ by replacing the empty configuration with
$e_1$, and attach a fresh independent mark there.  The resulting sum
$\widetilde R_t$ and maximum $\widetilde Q_*(t)$ differ from $R_t$ and
$Q_*(t)$ only on $\{V_t=0\}$, whose probability tends to zero.  Thus they
have the same endpoint convergence statements.  At zero the empty event is
null under infinite activity.

\subsection{Constant limits and the largest weight}

Theorem~\ref{thm:countable-dust} applies to the marked ranked jumps at zero
and to the auxiliary partition at infinity.  The preceding asymptotic
equivalence transfers its conclusion to the original variables.  In
particular,
\begin{equation}
 R_t\Rightarrow\text{a constant}
 \quad\Longleftrightarrow\quad
 Q_*(t)\longrightarrow0\quad\text{in probability},
 \label{eq:classification-constant-iff-max}
\end{equation}
and the constant is $\mu=\E X$.  The same theorem gives, for any $p>1$,
\begin{equation}
 Q_*(t)\to0\text{ in probability}
 \quad\Longleftrightarrow\quad
 \E\sum_iQ_i(t)^p\to0.
 \label{eq:classification-max-power-equivalence}
\end{equation}
At small time, Corollary~\ref{cor:smalltime-positive-power} and
Theorem~\ref{thm:tauberian-main} turn a positive power-sum limit into the
regular-variation branch;
a zero limit detects dust at either endpoint.

\subsection{The analytic form of the dust condition}

Define
\begin{equation}
 A(x)=\int_{(0,x]}y\,\nu(\dd y),
 \qquad
 I(x)=\int_0^x\bar\nu(y)\dd y.
 \label{eq:classification-A-I}
\end{equation}
Kevei and Mason's maximal-jump theorem
\cite[Theorem~1.1(iii), $k=0$]{KM14} gives, at zero and at infinity,
\begin{equation}
 Q_*(t)\longrightarrow0\text{ in probability}
 \quad\Longleftrightarrow\quad
 \frac{x\bar\nu(x)}{A(x)}\longrightarrow0,
 \label{eq:classification-maximal-jump-criterion}
\end{equation}
with infinite activity imposed at zero.

By Lemma~\ref{lem:poisson-integrated-tail}, the condition in
\eqref{eq:classification-maximal-jump-criterion} is equivalent to
\begin{equation}
 I\in\RV_\star(0).
 \label{eq:classification-integrated-tail-equivalence}
\end{equation}
Here it is the integrated tail $I$, not $\bar\nu$, that is slowly varying.
This does not imply
$\bar\nu\in\RV_\star(-1)$: the monotone-density theorem cannot be
used in that way at an index-zero primitive.

\subsection{The classification theorem}

Write $t\to\star$ for either $t\downarrow0$ or $t\to\infty$, and let
$x\to\star$ denote the corresponding endpoint for jump sizes.  At zero we
assume $\bar\nu(0+)=\infty$, so that $V_t>0$ almost surely for every
$t>0$.  Convergence below is along the specified endpoint, rather than only
along a subsequence.

\begin{theorem}[Classification of convergent endpoint limits]
\label{thm:classification-endpoints}
Let $V$ be a nonzero, unkilled, driftless subordinator, and let the iid marks
be independent of $V$, nonconstant, and integrable.  Fix either endpoint,
with infinite activity at zero.  Then $R_t$ converges in distribution at
that endpoint if and only if exactly one of the following alternatives
holds.
\begin{enumerate}
\item[\textup{(ND)}] For a unique $\alpha\in[0,1)$,
\begin{equation}
 \bar\nu\in\RV_\star(-\alpha).
 \label{eq:classification-regular-branch}
\end{equation}
The limit is nondegenerate and has the law
\begin{equation}
 R_t\Rightarrow
 \begin{cases}
 \displaystyle\sum_{i\geq1}Q_iX_i,
 &(Q_i)^\downarrow\sim\PD(\alpha,0),\quad0<\alpha<1,\\[2mm]
 X,&\alpha=0,
 \end{cases}
 \label{eq:classification-regular-limit}
\end{equation}
where the weights and marks are independent.
\item[\textup{(D)}] The maximal-jump condition holds:
\begin{equation}
 \frac{x\bar\nu(x)}{A(x)}\longrightarrow0.
 \label{eq:classification-dust-condition}
\end{equation}
Then
\[
 R_t\longrightarrow\mu
 \qquad\text{in probability}.
\]
\end{enumerate}
In alternative \textup{(D)}, \eqref{eq:classification-dust-condition} is
equivalent to each
of
\begin{align}
 Q_*(t)&\longrightarrow0&&\text{in probability},
 \label{eq:classification-dust-max}\\
 \E\sum_iQ_i(t)^p&\longrightarrow0
 &&\text{for one, equivalently every, }p>1,
 \label{eq:classification-dust-powers}\\
 I&\in\RV_\star(0).
 \label{eq:classification-dust-I}
\end{align}
\end{theorem}

The case $\alpha=0$ is a one-big-jump regime governed by slow variation of
$\bar\nu$; alternative~\textup{(D)} is an asymptotic-dust regime governed by
slow variation of the integrated tail $I$.

\begin{proof}
Suppose first that $R_t$ has a limit.  The limit is either nondegenerate or
constant.  In the nondegenerate case,
Corollary~\ref{cor:smalltime-positive-power} and
Theorem~\ref{thm:tauberian-main} give
$\bar\nu\in\RV_0(-\alpha)$ for a unique
$\alpha\in[0,1)$ at zero.  At infinity the same conclusion is
Proposition~\ref{prop:largetime-nondegenerate}.  Proposition
\ref{prop:measure:ranked-jumps} and the independent-marking argument in
Proposition~\ref{prop:measure:marking-dust} identify the limit as
$\PD(\alpha,0)$-weighted when $0<\alpha<1$ and as $X$ when $\alpha=0$.

If the limit is constant, apply
Theorem~\ref{thm:countable-dust} to the marks attached to
the augmented ranked partition described above, breaking ties without using
the marks.  We obtain
$Q_*(t)\to0$ and the constant is $\mu$.  The equivalences
\eqref{eq:classification-dust-condition}--\eqref{eq:classification-dust-I}
follow from \eqref{eq:classification-maximal-jump-criterion},
\eqref{eq:classification-max-power-equivalence}, and
Lemma~\ref{lem:poisson-integrated-tail}.

Conversely, Proposition~\ref{prop:measure:ranked-jumps} followed by
Proposition~\ref{prop:measure:marking-dust} proves the stated limit under
\eqref{eq:classification-regular-branch}.  Under
\eqref{eq:classification-dust-condition}, the maximal-jump theorem and
Theorem~\ref{thm:countable-dust} give
$R_t\to\mu$ in probability.
Lemma~\ref{lem:measure:marked-nondegenerate} shows that the limit in the
regular branch is nondegenerate.

The alternatives are disjoint.  Karamata's theorem gives
$I\in\RV_\star(1-\alpha)$ in \textup{(ND)}, with strictly positive
index, whereas Lemma~\ref{lem:poisson-integrated-tail} gives
$I\in\RV_\star(0)$ in \textup{(D)}.
\end{proof}

\begin{remark}[Drift, killing, and finite activity at zero]
\label{rem:classification-scope}
Theorem~\ref{thm:classification-endpoints} concerns unkilled driftless
subordinators.  If the drift is $\delta>0$, the
normalized jump weights have total mass
\[
 1-\frac{\delta t}{\delta t+\sum_{0<s\leq t}\Delta V_s},
\]
and the missing mass is genuine dust; at small time the drift may dominate
the jumps.  A killed subordinator requires either conditioning on survival
or adjoining a cemetery state.  Finally, when $\nu(0,\infty)<\infty$, the
small-time empty configuration has probability
$\exp\{-t\nu(0,\infty)\}\to1$, so any unconditional normalized limit depends
on the convention imposed on that event.  These cases lie outside the
classification.
\end{remark}

\subsection{The constant boundary and the integrated tail}

Theorem~\ref{thm:classification-endpoints} answers the first-moment
question raised by Kevei and Mason \cite[Remark~1]{KM13}.  In the nondegenerate branch,
convergence for one nonconstant integrable mark forces regular variation
of the L\'evy tail.  At the constant boundary, their maximal-jump theorem
\cite[Theorem~1.1(iii), \(k=0\)]{KM14} gives the necessary and sufficient
condition \eqref{eq:classification-maximal-jump-criterion}.  In our
formulation, this is expressed as slow variation of the integrated tail \(I\).

References to \cite{KM13} below use the numbering of arXiv:1210.2411v1.
To relate our formulation to the second-moment condition
in Theorem~1(iii) there, suppose \(\E|X|^p<\infty\) for some \(p>2\),
and write \(T\) for the limiting variable.  Conditional Jensen's inequality
makes the squares of the marked ratios uniformly integrable under this hypothesis.
Moreover, taking \(R_t=0\) on the empty configuration,
\[
 \E R_t=(\E X)\Pp\{V_t>0\},
\]
and the probability on the right is one at zero under infinite activity
and tends to one at infinity.  Thus the identity
\(\E T^2=(\E X)^2\) is precisely the constant-limit case.

In Proposition~6, Case~2, following equation~(57) of that preprint,
the argument establishes slow variation of \(\Phi(q)/q\) as \(q\to\infty\)
(respectively \(q\downarrow0\)), equivalently of \(I(x)\) as \(x\downarrow0\)
(respectively \(x\to\infty\)); Karamata's theorem gives
\(\Phi(q)/q\sim I(1/q)\) in this regime.  At index zero,
slow variation of a primitive does not in general imply regular variation
with index \(-1\) of its monotone density; see
\cite[Theorem~1.1 and Remark~1.4]{Kevei2021}.  Thus regular variation of
\(\bar\nu\) with index \(-1\) is sufficient for the constant limit,
whereas the necessary and sufficient condition is supplied by the
maximal-jump criterion above.  The next two examples illustrate this
distinction at infinity and at zero.

\begin{example}[The gamma subordinator at infinity]
\label{ex:classification-gamma}
Let
\[
 \nu(\dd x)=a x^{-1}e^{-bx}\dd x,
 \qquad a,b>0.
\]
The L\'evy tail decreases exponentially at infinity and therefore is not in
$\RV_\infty(-1)$.  Nevertheless,
$\int_0^\infty x\,\nu(\dd x)=a/b<\infty$, and
Lemma~\ref{lem:largetime-finite-mean} gives $R_t\to\mu$ almost surely.

For Bernoulli$(\theta)$ marks, $0<\theta<1$, thinning makes $U_t$ and
$V_t-U_t$ independent gamma variables with common rate $b$ and respective
shapes $a\theta t$ and $a(1-\theta)t$.  Consequently,
\[
 \frac{U_t}{V_t}\sim
 \operatorname{Beta}(a\theta t,a(1-\theta)t)
 \longrightarrow\theta
\]
in probability.  Since the ratio is bounded by one,
$\E(U_t/V_t)^2\to\theta^2=(\E X)^2$, while
$\bar\nu\notin\RV_\infty(-1)$.
\end{example}

\begin{example}[A monotone oscillatory tail at zero]
\label{ex:classification-oscillatory}
Choose $L_0=6$, put $x_0=e^{-L_0}$, and set
\[
 c_0=\frac{2+\sin L_0}{x_0L_0^2}.
\]
Define the global tail
\begin{equation}
 \bar\nu(x)=
 \begin{cases}
 \displaystyle
 \dfrac{2+\sin\log(1/x)}{x\log^2(1/x)},&0<x<x_0,\\[2mm]
 c_0,&x_0\leq x<2x_0,\\
 0,&x\geq2x_0.
 \end{cases}
 \label{eq:classification-oscillatory-tail}
\end{equation}
This function is right-continuous, has bounded support, and has only a
downward jump, at $2x_0$.  To verify monotonicity on its first branch,
write $L=\log(1/x)$ and
\[
 g(L)=\frac{e^L(2+\sin L)}{L^2}.
\]
Then
\[
 \frac{g'(L)}{g(L)}
 =1+\frac{\cos L}{2+\sin L}-\frac2L
 \geq1-\frac1{\sqrt3}-\frac2L>0
 \qquad(L\geq L_0).
\]
The value $c_0$ makes the first join continuous.  Hence
\eqref{eq:classification-oscillatory-tail} is a nonincreasing,
right-continuous tail and defines a L\'evy measure; indeed,
\[
 \int_0^1\bar\nu(y)\dd y
 =\int_{L_0}^\infty\frac{2+\sin u}{u^2}\dd u+c_0x_0<\infty,
 \qquad \bar\nu(0+)=\infty.
\]
Moreover,
\[
 \int_0^x\bar\nu(y)\dd y
 =\int_L^\infty\frac{2+\sin u}{u^2}\dd u
 =\frac2L+\frac{\cos L}{L^2}+O(L^{-3}),
\]
whereas
\[
 x\bar\nu(x)=\frac{2+\sin L}{L^2}.
\]
Thus the L\'evy condition holds, the activity is infinite, and
\[
 \frac{x\bar\nu(x)}{A(x)}
 \sim\frac{2+\sin L}{2L}\longrightarrow0.
\]
The tail is not regularly varying: multiplication of $x$ by $e^\pi$
reverses the sine phase in \eqref{eq:classification-oscillatory-tail}; in
fact,
\[
 \frac{\bar\nu(e^\pi x)}{\bar\nu(x)}
 \sim e^{-\pi}\frac{2-\sin L}{2+\sin L},
\]
which has no limit.  Hence $R_t\to\mu$ in
probability as $t\downarrow0$, although
$\bar\nu\notin\RV_0(-1)$.  For Bernoulli$(\theta)$ marks the ratios are
bounded, so the convergence also gives
$\E R_t^2\to\theta^2=(\E X)^2$, although
$\bar\nu\notin\RV_0(-1)$.
\end{example}

Both examples satisfy the maximal-jump criterion, or equivalently
\(I\in\RV_\star(0)\), without satisfying
\(\bar\nu\in\RV_\star(-1)\).

\section{Consequences for normalized L\'evy random measures}
\label{sec:measure}

The ranked-jump classification and the independent-marking result give
the corresponding limits for normalized random measures.
Let $E$ be a Polish space and let $H\in\cP(E)$.  Independently mark every
jump of the subordinator $V$ from Section~\ref{sec:poisson-limits} by a
location with law $H$, and set
\begin{equation}
 \mu_t=\sum_{0<s\leq t}\Delta V_s\delta_{Z_s},
 \qquad
 \mathsf P_t=\frac{\mu_t}{V_t}\quad\text{on }\{V_t>0\}.
 \label{eq:measure:crm}
\end{equation}
On $\{V_t=0\}$, put $\mathsf P_t=H$.  This event is absent at zero
under infinite activity and has probability tending to zero at infinity.
After ranking the jumps with the independent tie-breakers fixed in
Section~\ref{sec:poisson-limits}, write $Z_i(t)$ for the location attached
to $J_{(i)}(t)$, padding zero coordinates with fresh iid $H$-locations.
Then
\begin{equation}
 \mathsf P_t
 =\mathcal M_H\bigl(Q^\downarrow(t),Z(t)\bigr),
 \label{eq:measure:marked-representation}
\end{equation}
including on the zero-mass event, where the convention
$Q^\downarrow(t)=0$ makes the missing-mass component equal to $H$.

For $g\in L^1(H)$, write
\[
 \mathsf P_tg=\int_E g\dd\mathsf P_t.
\]
This integral is finite almost surely.  Indeed, conditionally on
the ranked weights,
\[
 \E\left[\sum_iQ_i(t)|g(Z_i(t))|\mathrel{\Big|}Q^\downarrow(t)\right]
 =\left(\sum_iQ_i(t)\right)H|g|\leq H|g|,
\]
and the dust term in \eqref{eq:measure:marking-map} is finite as well.

\begin{corollary}[Endpoint limits of normalized L\'evy random measures]
\label{thm:measure:endpoint-limits}
Let $V$ be nonzero, unkilled, and driftless.
Fix $\star\in\{0,\infty\}$, with infinite activity when $\star=0$, and
let $g_1,\ldots,g_m\in L^1(H)$.

Suppose first that, for some $0\leq\alpha<1$,
\begin{equation}
 \bar\nu\in\RV_\star(-\alpha).
 \label{eq:measure:regular-branch}
\end{equation}
Let $Q^{(\alpha)}\sim\PD(\alpha,0)$ when $0<\alpha<1$, let
$Q^{(0)}=e_1=(1,0,\ldots)$, and let $Z_1,Z_2,\ldots$ be iid with law
$H$, independently of $Q^{(\alpha)}$.  Then, in
$\ell^1\times\cP(E)\times\R^m$,
\begin{equation}
 \bigl(Q^\downarrow(t),\mathsf P_t,\mathsf P_tg_1,\ldots,\mathsf P_tg_m\bigr)
 \Longrightarrow
 \bigl(Q^{(\alpha)},\mathsf P_{\alpha,H},
       \mathsf P_{\alpha,H}g_1,\ldots,\mathsf P_{\alpha,H}g_m\bigr),
 \label{eq:measure:joint-regular}
\end{equation}
where $\cP(E)$ carries the weak topology and
\begin{equation}
 \mathsf P_{\alpha,H}
 :=\mathcal M_H(Q^{(\alpha)},Z)
 \stackrel{d}{=}
 \begin{cases}
  \displaystyle\sum_{i\geq1}Q_i^{(\alpha)}\delta_{Z_i},
       &0<\alpha<1,\\[2mm]
  \delta_{Z_1},&\alpha=0.
 \end{cases}
 \label{eq:measure:random-measure-limit}
\end{equation}
In particular,
\begin{equation}
 Q^\downarrow(t)\Longrightarrow Q^{(\alpha)}
 \quad\text{in }\ell^1,
 \label{eq:measure:l1-limit}
\end{equation}
and $\mathsf P_t\Rightarrow\mathsf P_{\alpha,H}$ in distribution as random
elements of $\cP(E)$ equipped with its weak topology.

Suppose instead that
\begin{equation}
 \frac{x\bar\nu(x)}{A(x)}\longrightarrow0
 \qquad(x\to\star).
 \label{eq:measure:dust-condition}
\end{equation}
Then
\begin{equation}
 Q^\downarrow(t)\xrightarrow{\Pp}0
 \quad\text{in the product topology},
 \qquad
 \mathsf P_t\xrightarrow{\Pp}H
 \quad\text{weakly in }\cP(E),
 \label{eq:measure:dust-limit}
\end{equation}
and
\begin{equation}
 \bigl(\mathsf P_tg_1,\ldots,\mathsf P_tg_m\bigr)
 \longrightarrow(Hg_1,\ldots,Hg_m)
 \quad\text{in }L^1(\Omega;\R^m).
 \label{eq:measure:dust-probes}
\end{equation}
Here $L^1(\Omega;\R^m)$ is formed using the $\ell^1$-norm on
$\R^m$.
The convergences in \eqref{eq:measure:dust-limit} and
\eqref{eq:measure:dust-probes} hold jointly.
\end{corollary}

\begin{proof}
Under \eqref{eq:measure:regular-branch},
Proposition~\ref{prop:measure:ranked-jumps} gives
\eqref{eq:measure:l1-limit}.  Conditionally on $Q^\downarrow(t)$, the ordered
locations $Z_i(t)$ are iid with law $H$; the tie-breakers were chosen
independently of the marks.  Consequently,
\[
 \Law\bigl(Q^\downarrow(t),Z(t)\bigr)
 =\Law(Q^\downarrow(t))\otimes H^{\otimes\mathbb N}.
\]
Proposition~\ref{prop:measure:marking-dust}, applied with the
$\ell^1$ topology on the first coordinate, proves
\eqref{eq:measure:joint-regular}.

Under \eqref{eq:measure:dust-condition}, the maximal-jump criterion
\cite[Theorem~1.1(iii), $k=0$]{KM14} gives
\[
 Q_1(t)=\frac{J_{(1)}(t)}{V_t}\xrightarrow{\Pp}0,
\]
with the ratio set to zero on $\{V_t=0\}$.  Since
$Q_i(t)\leq Q_1(t)$, the weights converge to zero in the product
topology.  Proposition~\ref{prop:measure:marking-dust} and
$\mathcal M_H(0,Z)=H$ give weak convergence of $\mathsf P_t$ to
$H$.  Because the limit is deterministic, this is convergence in
probability.

The $L^1$ conclusion \eqref{eq:measure:marking-dust-L1} of the same
proposition gives \eqref{eq:measure:dust-probes} for the entire finite
family.
\end{proof}

\begin{corollary}[Classification from one mean functional]
\label{cor:measure:one-projection}
Fix $\star\in\{0,\infty\}$, with infinite activity when $\star=0$, and
let $f\in L^1(H)$ be a function that is not constant $H$-almost surely.
Then
$\mathsf P_tf$ converges in distribution as $t\to\star$, without
passing to a subsequence, if and only if exactly one of the following holds:
\begin{enumerate}
\item[(i)] for a unique $\alpha\in[0,1)$,
\[
 \bar\nu\in\RV_\star(-\alpha);
\]
\item[(ii)]
\[
 \frac{x\bar\nu(x)}{A(x)}\longrightarrow0
 \qquad(x\to\star).
\]
\end{enumerate}
In case \textup{(i)}, the limit is
\[
 \mathsf P_{\alpha,H}f\stackrel{d}{=}
 \begin{cases}
  \displaystyle\sum_{i\geq1}Q_i^{(\alpha)}f(Z_i),
       &0<\alpha<1,\\[2mm]
  f(Z_1),&\alpha=0,
 \end{cases}
\]
and is nondegenerate.  In case \textup{(ii)}, the limit is the constant
$Hf$.  In both cases the entire normalized random measure has the
corresponding limit in Corollary~\ref{thm:measure:endpoint-limits}.
\end{corollary}

\begin{proof}
On $\{V_t>0\}$,
\[
 \mathsf P_tf
 =\frac{\sum_{0<s\leq t}f(Z_s)\Delta V_s}{V_t}.
\]
The marks $f(Z_s)$ are iid, integrable, nonconstant, and independent of
$V$.  At zero the denominator is positive almost surely; at infinity the
zero-mass convention changes the ratio only on an event whose probability
tends to zero.  The scalar classification,
Theorem~\ref{thm:classification-endpoints}, therefore gives the two
alternatives, their disjointness, and the uniqueness of $\alpha$.
Corollary~\ref{thm:measure:endpoint-limits} gives the stated measure and
functional limits.  In the regular branch, $Q_1^{(\alpha)}>0$ almost
surely, so Lemma~\ref{lem:measure:marked-nondegenerate} shows that
$\mathsf P_{\alpha,H}f$ is nondegenerate.
\end{proof}

\begin{remark}[Atomic base measures]
\label{rem:measure:atomic-base}
For atomic $H$, the Poisson--Dirichlet law in
\eqref{eq:measure:random-measure-limit} describes the weights before
aggregation at repeated locations.  The ranked masses of the distinct
atoms of $\mathsf P_{\alpha,H}$ need not have that law.
\end{remark}

\section*{Acknowledgments}
The author is grateful to Marinella Mellini and Mauro Lenzi for their
unwavering encouragement and support. He is also grateful to P\'eter Kevei
and David M. Mason, whose work first drew his attention to Breiman's
conjecture and inspired the line of research leading to this paper.

OpenAI's ChatGPT was used for language editing and \LaTeX{} formatting,
and for assistance in devising and formalizing
Examples~\ref{ex:classification-gamma} and~\ref{ex:classification-oscillatory};
responsibility for the mathematical content rests with the author.

\makeatletter
\def\@biblabel#1{\@defaultbiblabelstyle{#1}}
\def\bibsetup{}
\makeatother
\providecommand{\bysame}{\leavevmode\hbox to3em{\hrulefill}\thinspace}

\end{document}